\documentclass[11pt]{amsart}
\usepackage{times}
\usepackage[leqno]{amsmath}
\usepackage{amssymb}
\usepackage{mystyle}
\usepackage[colorlinks=true,linkcolor=blue]{hyperref}
\usepackage{placeins}

\usepackage{tikz}

\author{Stephen DeBacker}
\address{University of Michigan\\
Ann Arbor, MI 48109-1043, USA}
\email{smdbackr@umich.edu}

\makeatletter
\let\@wraptoccontribs\wraptoccontribs
\makeatother

\contrib[with an appendix by]{Jeffrey D. Adler}

  
  \date{\today}
  
  \title{Parameterizing Conjugacy Classes of Unramified Tori via Bruhat-Tits theory}

\newcommand{\toriK}{\mathcal{T}_K}
\newcommand{\toriKGamma}{\mathcal{T}_K^{\Gamma}}
\newcommand{\JGamma}{J^{\Gamma}}
\newcommand{\JGammamax}{J_\mathrm{max}^{\Gamma}}
\newcommand{\GGamma}{G^{\Gamma}}
\newcommand{\MGamma}{M^{\Gamma}}
\newcommand{\affFrW}{W^{\Fr,\aff}}

\usepackage{scalerel}
\newcommand{\Ebtorus}{{\bA}^{\! \scaleto{\# \mathstrut}{6pt}}}
\newcommand{\Ebtorusrat}{{A}^{\! \scaleto{\# \mathstrut}{6pt}}}
\newcommand{\EbStorus}{{\bS}^{\! \scaleto{\# \mathstrut}{6pt}}}
\newcommand{\EbTtorus}{{\bT}^{\! \scaleto{\# \mathstrut}{6pt}}_1}
\newcommand{\EbTtorusrat}{{T}^{\! \scaleto{\# \mathstrut}{6pt}}_1}

\DeclareMathOperator{\ee}{e}
\DeclareMathOperator{\aff}{aff}
\DeclareMathOperator{\X}{X}
\DeclareMathOperator{\characteristic}{char}
\DeclareMathOperator{\red}{red}
\DeclareMathOperator{\Ad}{Ad}
\DeclareMathOperator{\Stab}{Stab}

\begin{document}

\begin{abstract}
  Suppose $k$ is a nonarchimedean local field, $K$ is a maximally unramified extension of $k$, and $\bG$ is a connected reductive $k$-group.  In this paper we provide parameterizations via Bruhat-Tits theory of: the rational conjugacy classes of $k$-tori in $\bG$ that split over $K$;  the rational and stable conjugacy classes of the $K$-split components of the centers of unramified twisted Levi subgroups of $\bG$; and the rational conjugacy classes of unramified twisted generalized Levi subgroups of $\bG$.   We also provide parameterizations of analogous objects for finite groups of Lie type.  
\end{abstract}

\maketitle

\section*{Introduction}

Suppose $k$ is a nonarchimedean local field, $K$ is a maximally unramified extension of $k$, and $\bG$ is a connected reductive $k$-group.   If $\bT$ is a maximal $k$-torus in $\bG$, then $\bT^K$\!,  the maximal $K$-split torus in $\bT$, is defined over $k$ and $\bT$ is a maximal $K$-minisotropic $k$-torus in $\bL = C_\bG(\bT^K)$, the centralizer in $\bG$ of $\bT^K$.   The group $\bL$ is an unramified twisted Levi subgroup of $\bG$; that is, $\bL$ is a $k$-group that occurs as the Levi component of a parabolic $K$-subgroup of $\bG$.  Consequently, an approach to parameterizing the rational conjugacy classes of  maximal tori in $\bG$ is to
\begin{itemize}
    \item parameterize the rational conjugacy classes of unramified twisted Levi subgroups of $\bG$ and
    \item for each  unramified twisted Levi subgroup $\bL$ of $\bG$ parameterize the $\bL(k)$-conjugacy classes of $K$-minisotropic maximal $k$-tori in $\bL$.
\end{itemize}
This paper takes up the former problem.  The latter problem is studied in~\cite{debacker:totally}.  Future work will take up the problem of parameterizing, via Bruhat-Tits theory, the rational classes of  tame tori in a reductive $p$-adic group.

A subgroup $\bL$ of $\bG$ is called an \emph{unramified twisted Levi subgroup in $\bG$} provided that $\bL$ is a $k$-group that occurs as a Levi factor for a parabolic $K$-subgroup of $\bG$.
A $k$-torus $\bS$ is called an \emph{unramified torus in $\bG$} provided that $\bS$ is the $K$-split component of the center of an unramified twisted Levi subgroup in $\bG$.    It follows that  the problem of understanding the set of unramified twisted Levi subgroups up to rational conjugacy is equivalent to the problem of understanding  the set of unramified tori up to rational conjugacy.   A parameterization of the rational conjugacy classes of \emph{maximal} unramified tori in $\bG$ via Bruhat-Tits theory was carried out in~\cite{debacker:unramified}, and this paper may be viewed as a generalization of the results found there.

If one tries to generalize~\cite{debacker:unramified} by na\"ively replacing the role of ``maximal tori in groups over the residue field'' with ``twisted Levis in groups over the residue field,'' it will not work.  One reason for this failure is that any reductive group over the residue field is quasi-split, but unramified twisted Levis do not need to be $k$-quasi-split.  To make it work, one needs to introduce a bit more data, as we now discuss.

In Theorem~\ref{thm:maintheorem}  the rational conjugacy classes of unramified tori are parameterized in terms of equivalence classes of elliptic triples $(F,\theta,w)$ that arise from Bruhat-Tits theory.   When $\bG$ is $k$-split,  the triples $(F,\theta,w)$ may be described as follows. Let $\bA^k$ denote a maximal $k$-split torus in $\bG$.  Then  $F$ is a facet in the apartment of $\bA^k$ in the Bruhat-Tits building of $\bG$, $\theta$ is a subset of a basis for $\Phi(\bG,\bA^k)$, the set of roots of $\bG$ with respect to $\bA^k$, and $w$ is an element of the Weyl group of the reductive quotient at $F$  which preserves the root subsystem  in $\Phi(\bG,\bA^k)$ spanned by $\theta$.   There is a natural notion of equivalence among such triples, and the triple $(F,\theta,w)$ is elliptic provided that $\dim(F) \geq \dim(F')$ for all triples $(F',\theta',w')$ that are equivalent to $(F,\theta,w)$.  When $\bG$ is not $k$-split, the parameterization is modified to account for the action of the Galois group of $K$ over $k$.

If we restrict our attention to the set of triples that are  of the form $(F, \emptyset, w)$, then we
recover the parameterization of  maximal unramified tori of $\bG$ in~\cite{debacker:unramified}.  It is the addition of the datum $\theta$, which has little or no relationship to $F$, that allows the parameterization of this paper to work.   In analogy with~\cite{debacker:unramified}, the dimension of the Lie algebra of the maximal $k$-split torus that occurs in the centralizer of an unramified torus parameterized by $(F,\theta,w)$ is equal to the dimension of $F$.

 We  now discuss the contents of this paper.  
  By way of motivation, in  Section~\ref{sec:leviff} we parameterize the finite-field analogue of unramified twisted Levi subgroups; this can be viewed as a natural generalization of the known parameterization, in terms of Frobenius-conjugacy classes in the Weyl group, of maximal tori in finite groups of Lie type.   In Section~\ref{sec:littlereview} we recall some facts about the relationship between Levi subgroups and various tori in their centers.  This section also verifies that every unramified torus occurs as the maximal $K$-split subtorus of some maximal $k$-torus of $\bG$.
In Section~\ref{sec:Kktori} we parameterize, via Bruhat-Tits theory, \emph{all} 
$k$-tori that split over $K$ and contain the maximal $K$-split torus in the center of $\bG$.   Building on this, in Section~\ref{sec:paramunram}  we parameterize, as discussed above,  the rational conjugacy classes of unramified tori in terms of Bruhat-Tits theory.

In Section~\ref{sec:stableconj}
we define a notion of stable conjugacy for  unramified tori and provide a criterion, in terms of the triple $(F,\theta,w)$, to describe when two rational conjugacy classes of unramified tori are stably conjugate.   We also examine a more general $k$-embedding question. Given an unramified torus $\bS$ in $\bG$, a $k$-morphism $f \colon \bS \rightarrow \bG$ is said to be a \emph{$k$-embedding} provided that there exists $g\in \bG(K)$ such that $f(s) = gsg\inv$ for all $s \in \bS(K)$.   We enumerate, in terms of parameterizing data, the set of  $k$-embeddings of $\bS$  up to rational conjugacy.  Note that the discussion of this result in~\cite{debacker:oberwolfach} is incorrect.

Finally, in Section~\ref{sec:twistedgeneralized} we parameterize the rational conjugacy classes of generalized unramified twisted Levis (see Definition~\ref{defn:utgls}).   This is akin to the parameterization of maximal-rank unramified subgroups of a $K$-split group $\bG$ in~\cite{debacker:unramified}, though less restrictive and, I believe, easier to understand.  We end as we began  by parameterizing the conjugacy classes of  the finite-field analogue of unramified twisted generalized Levi subgroups.

\section*{Acknowledgements}  I thank  Jeffrey Adler, Jacob Haley, David Schwein, Loren Spice, and Cheng-Chiang Tsai for discussions that vastly improved both the content and the exposition of this paper.  I also thank the   American Institute of Mathematics whose hospitality and SQuaRE program created a wonderful  environment for doing mathematics, and it is there that  this research was partially carried out.

\section{Notation}
\subsection{Fields, groups, roots}
Let $k$ denote a field that is complete
with respect to a nontrivial discrete valuation $\nu$.
We assume that the residue field $\ff$ of $k$ is perfect, and outside of Sections~\ref{sec:littlereview} and~\ref{sec:Kktori}, we will assume that the residue field $\ff$ of $k$ is also quasi-finite.  Let $\bar{k}$ denote a fixed separable closure of $k$, and let $K \leq \bar{k}$ be the maximal unramified extension of $k$.  The valuation extends uniquely to $\bar{k}$, and we will also denote this extension by $\nu$.
Let $\ffc$ denote the residue field of $K$; it is an algebraic closure of $\ff$. We will often identify an algebraic $\ff$-group with its group of $\ffc$-points.

If $\mathcal{G}$ is a group and $x,y \in \mathcal{G}$, then $\Int(x) (y)$  and $\lsup{x}y$  are defined to be $xyx\inv $

If $\bH$ is a $k$-group, then we let $\Lie(\bH)$ denote its Lie algebra.

We will use the following font convention: If $\bH$ denotes an algebraic $K$-group, then we will denote its group of $K$-points by $H$.   

We identify $\Gamma = \Gal(K/k)$ with $\Gal(\ffc/\ff)$.  When $\ff$ is quasi-finite we fix a topological generator $\Fr$ for $\Gamma$, and  choose a lift of $\Fr$ to an element, which we will also call $\Fr$, of $\Gal(\bar{k}/k)$. 

Let $\bA$ denote a maximal $K$-split $k$-torus in $\bG$ that contains a maximal $k$-split torus of $\bG$; such a torus exists and is unique up to $G^{\Gamma}$-conjugacy (see~\cite[Theorem~6.1]{prasad:unramified} or in~\cite[Theorem~3.4.1]{debacker:unramified} take an unramified torus corresponding to a pair of the form $(F,\sfT) \in I^m$  with $F$ an alcove).
 We denote by $\Phi = \Phi(\bG, \bA)$ the root system of $\bG$ with respect to $\bA$  and by $W = W(G,\bA)$ the Weyl group $N_G(\bA)/C_G(\bA)$.  We denote by 
$\Psi = \Psi(\bG, \bA, K, \nu)$ the set of affine roots of $\bG$ with respect to $\bA$ and $\nu$.   For $\psi \in \Psi$, we let $\dot{\psi} \in \Phi$ denote the gradient of $\psi$.

Since $\bG$ is $K$-quasi-split, there is a Borel $K$-subgroup, call it $\bB$, that contains $\bA$.   Let $\Delta = \Delta(\bG, \bB ,\bA)$ denote the corresponding set of simple roots in $\Phi$.  Set
$$\Theta = \Theta(\bG,\bA) = \{w \rho \, | \text{$w \in W$ and $\rho \subset \Delta$} \}.$$
The set $\Theta$ is independent of the choice of $\bB$.  Note that  $\Gamma$ acts on $\Phi$, $W$, and $\Theta$. 
If $\theta \in \Theta$, then we let $\Phi_\theta \subset \Phi$ denote the root subsystem generated by $\theta$, we let $W_\theta \leq W$ denote the corresponding Weyl group, and we let 
$\bA_{\theta} = \left(\bigcap_{\alpha \in \theta} \ker(\alpha) \right)^\circ$.

\subsection{Notation for tori}
Suppose $E$ is a Galois extension of $k$, and $\bT$ is a $k$-torus. We let  $\bT^E$ denote the maximal $E$-split subtorus in $\bT$.     

\begin{defn}
A \emph{maximally $k$-split maximal $k$-torus in $\bG$} is a maximal $k$-torus $\bT$ in $\bG$ such that $\dim(\Lie(\bT^k)) \geq \dim(\Lie(\tilde{\bT}^k))$ for all maximal $k$-tori $\tilde{\bT}$ in $\bG$.  
\end{defn}

Note that if $\bG$ is $k$-quasi-split, then a maximally $k$-split maximal $k$-torus in $\bG$ is the centralizer of a maximal $k$-split torus in $\bG$.  In general, 
if $\bT$ is a maximally $k$-split maximal $k$-torus in $\bG$, then $\bT$ is a maximal $k$-torus in $\bG$ that contains a maximal $k$-split torus of $\bG$.  

Suppose $\bH$ is an algebraic $k$-subgroup of $\bG$.  We let $\bZ_{\bH}$ denote the center of $\bH$.   We set $\bZ = \bZ_\bG$, and we let $\bZ^E_{\bH}$ denote the maximal $E$-split subtorus in $\bZ_{\bH}$.

Similar notation applies to objects defined over $\ff$.

\subsection{Buildings}  \label{sec:buildings}
We denote by $\BB(G)$ the (enlarged) building of $G$.  If $F$ is a facet in $\BB(G)$, then we denote the corresponding parahoric subgroup by $G_F$ and its pro-unipotent radical by $G_{F,0^+}$.   The quotient $G_F/G_{F,0^+}$ is the group of $\ff$-points of a connected reductive group $\sfG_F$.  If $F$ if $\Gamma$-stable, then 
$\sfG_F$ is an $\ff$-group.
The group $G^{\Gamma}$ acts simplicially on $\BB(G)^{\Gamma}$, and a fundamental domain for this action is called an \emph{alcove}.

We let $\AA(A)$ denote the apartment in $\BB(G)$ corresponding to $A$.
If $F$ is a facet in $ \AA(A)^{\Gamma}$, then the image of $A \cap G_F$ in $\sfG_F$, which we call $\sfA_F$, is a maximally $\ff$-split
maximal $\ff$-torus in $\sfG_F$.  That is, $\sfA_F$ is a maximal $\ff$-torus in $\sfG_F$ that contains a maximal $\ff$-split torus of $\sfG_F$.   We denote by $W_F$ the Weyl group $N_{\sfG_F}(\sfA_F)/\sfA_F$.  
Note that we may and do naturally identify $W_F$ with a subgroup of $W$.    
Let $\Psi_F$ denote the set of affine roots of ${\bA}$ that vanish on $F$, and let $\Phi_F$ denote the corresponding set of gradients.   Let $ \lsub{F}\bM$ denote
the Levi $k$-subgroup that contains ${\bA}$ and corresponds to $\Phi_F$, and let $(\lsub{F}\bM)$ denote the $G^\Gamma$-conjugacy class of $\lsub{F}\bM$. Note that $\lsub{F}\sfM_F = \sfG_F$.

Let $\affFrW$ denote the affine Weyl group $N_{G^{\Fr}}(\bA)/ (C_G(\bA) \cap G^{\Fr}_F)$  where $F$ is any facet in $\AA(A)^{\Fr}$.  Note that this is the affine Weyl group for $G^{\Fr}$ and not the $\Fr$-fixed points of the affine Weyl group of $G$.

\begin{example}
For the group $\Sp_4(k)$  each alcove is a right isosceles triangle.  For the group $\Gtwo_2(k)$ each alcove is a $30$-$60$-$90$ triangle.  For every facet $F$ in the alcoves pictured in Figure~\ref{fig:one}, we describe the algebraic $\ff$-group $\sfG_F$. 

\begin{figure}[ht]
\centering
\begin{tikzpicture}
\draw (0,0)
  -- (5,0) 
  -- (5,5) 
  -- cycle;
  \draw (0,-.1) node[anchor=east]{$\Sp_4$};
 \draw (5,-.1) node[anchor=west]{$ \SL_2 \times \SL_2$};
  \draw (5,5.1) node[anchor=west]{$\Sp_4 $};
  \draw (2.5,0) node[anchor=north]{$\SL_2 \times \GL_1$};
  \draw (1,2.5) node[anchor=west]{$\GL_2$};
   \draw (5,2.5) node[anchor=west]{$\SL_2 \times \GL_1$};
   \draw (2.5,1.5) node[anchor=west]{$\GL_1 \times \GL_1$};

   \draw (11,0)
  -- (13.9,5) 
  -- (11,5) 
  -- cycle;
  \draw (11,5.1) node[anchor=east]{$\SO_4$};
  \draw (13.9,5.1) node[anchor=west]{$ \SL_3 $};
  \draw (11,0) node[anchor=north]{$\Gtwo_2 $};
  \draw (11,2.5) node[anchor=east]{$\GL_2$};
  \draw (12.45,2.5) node[anchor=west]{$\GL_2$};
   \draw (12.25,5) node[anchor=south]{$\GL_2$};
    \draw (11,4) node[anchor=west]{$\GL_1 \times \GL_1$};

\end{tikzpicture}

\caption{The reductive quotients for $\Sp_4$ and $\Gtwo_2$ \label{fig:one}}
\end{figure}
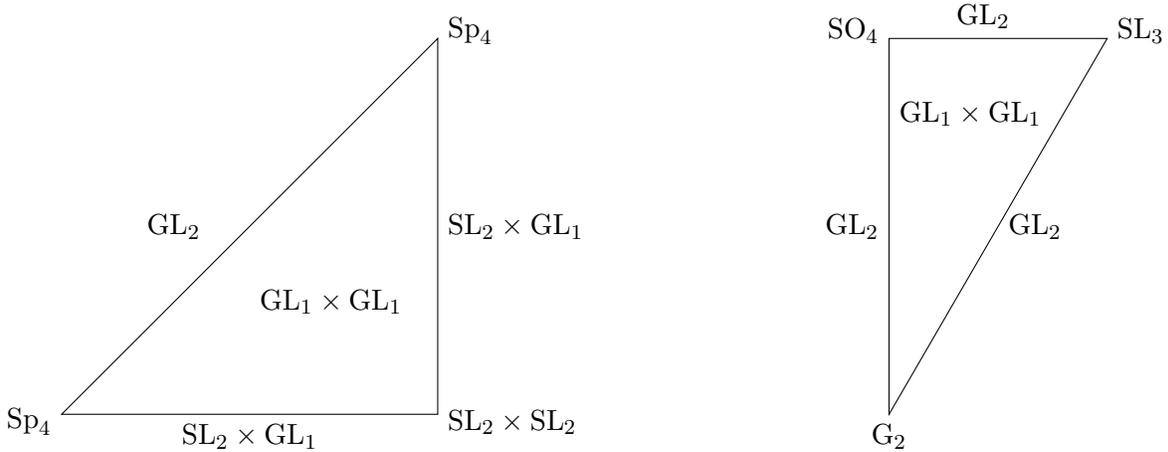
\end{example}

\section{Twisted Levi subgroups for reductive groups over quasi-finite fields} \label{sec:leviff}

Suppose $\ff$ is quasi-finite.

As a way to motivate what happens in the nonarchimedean setting, we first look at an analogous parameterization question for 
a connected reductive group over $\ff$, e.g. a finite group of Lie type.    We first recall an important fact about the Galois cohomology of such groups.

\begin{lemma} \label{lem:qfcohom}   If $\sfH$ is a connected reductive $\ff$-group, then $\cohom^1(\Fr,\sfH) = 1$ and Borel $\ff$-subgroups of $\sfH$ exist.
\end{lemma}

\begin{proof}
Thanks to~\cite[XIII Section 2]{serre:local} the quasi-finite field $\ff$ is of dimension $\leq 1$. The result now follows from Steinberg's Theorem~\cite[III Section 2.3]{serre:galois}.
\end{proof}

\subsection{Notation}

Suppose $\sfG$ is a connected reductive group defined over $\ff$.  Fix a Borel $\ff$-subgroup $\sfB$ of $\sfG$ and a maximal $\ff$-torus $\sfA$ of $\sfB$.  We denote by $\Delta_\sfG = \Delta(\sfG, \sfB, \sfA)$ the corresponding set of simple roots and by $W_\sfG = W(\sfG,\sfA)$ the corresponding Weyl group.   For $\theta \subset \Delta_\sfG$, we let $W_{\sfG,\theta}$ denote the corresponding subgroup of $W_\sfG$.  

\begin{defn}
A reductive subgroup $\sfL$ of $\sfG$ is called a \emph{twisted Levi $\ff$-subgroup of $\sfG$} provided that $\sfL$ is defined over $\ff$ and there exists a parabolic $\ffc$-subgroup of $\sfG$ for which $\sfL$ is a Levi factor.  We let $\LL$ denote the set of twisted Levi $\ff$-subgroups of $\sfG$.
\end{defn}
 
 Every twisted Levi $\ff$-subgroup $\sfL$ of $\sfG$ identifies an $\ff$-torus in $\sfG$: let $\sfS_\sfL$ denote the connected component of the center of $\sfL$.     On the other hand, the centralizer of any $\ff$-torus in $\sfG$ is a twisted Levi subgroup~\cite[Proposition~1.2.2]{digne-michel}.   In light of this, we make the following definition.

 \begin{defn}
 An $\ff$-torus $\sfS$ in $\sfG$ that is equal to  the connected component of the center of $C_\sfG(\sfS)$ will be called a \emph{Levi torus}. 
 \end{defn}
 
 \begin{rem}
 For a twisted Levi subgroup $\sfL$ of $\sfG$, we have that $\sfL = C_{\sfG}(\sfS_\sfL)$~\cite[Proposition~1.2.1]{digne-michel}, so there is a $\sfG^\Fr$-equivariant bijective correspondence between the set of Levi tori in $\sfG$ and the set of twisted Levi $\ff$-subgroups in $\sfG$.
 Consequently, if we  understand $\LL$ up to $\sfG^{\Fr}$-conjugacy, then we will understand the set of Levi tori in $\sfG$ up to 
 $\sfG^{\Fr}$-conjugacy (and \emph{vice-versa}).
 \end{rem}

\subsection{A parameterization} \label{subsec:ffparam}

In this section we provide a parameterization of $\tilde{\LL}$, the set of $\sfG^{\Fr}$-conjugacy classes in $\LL$.  This parameterization may be viewed as a natural generalization of the known classification of the $\sfG^\Fr$-conjugacy classes of maximal $\ff$-tori in $\sfG$ (see~\cite[Proposition~3.3.3]{carter:finite}, \cite[Lemma~4.2.1]{debacker:unramified}, or~\cite[Section 1]{kazhdan-lusztig:fixed}).

Let $I_\sfG$ denote the set of pairs $(\theta,w)$ where $\theta \subset \Delta_\sfG$ and $w \in W_\sfG$ such that $\Fr (\theta) = w \theta$.   For $(\theta',w')$ and $(\theta,w) \in I_\sfG$ we write $(\theta',w') \sim (\theta,w)$ provided that there exists an element $\dot{n} \in W_\sfG$ for which
\begin{itemize}
\item  $   \theta = \dot{n} \theta ' $ and
\item $ w = \Fr(\dot{n}) w' \dot{n}\inv $.
\end{itemize}
One checks that $\sim$ is an equivalence relation on the set $I_\sfG$.

\begin{lemma}  \label{lem:ffversion}
There is a natural bijective correspondence between $I_\sfG / \! \sim$ and $\tilde{\LL}$. 
\end{lemma}

\begin{rem}  \label{rem:ffversion}
The set of equivalence classes in $I_\sfG$ for which the first entry in each pair is the empty set parameterizes the set of $\sfG^{\Fr}$-conjugacy classes of maximal $\ff$-tori in $\sfG$ (see~\cite[Proposition~3.3.3]{carter:finite}, \cite[Lemma~4.2.1]{debacker:unramified}, or~\cite[Section 1]{kazhdan-lusztig:fixed}).   At the other extreme, the singleton containing $(\Delta,1) \in I_\sfG$ is an equivalence class and  parameterizes the $\sfG^{\Fr}$-conjugacy class $\{\sfG\}$.
\end{rem}

\begin{proof}  We begin by defining a map $\varphi \colon I_\sfG \rightarrow \tilde{\LL}$. 
  Suppose that we have a pair $(\theta, w) \in I_\sfG$.   Thanks to~\cite[Section 4.2]{debacker:unramified} we can choose $g \in \sfG$ such that $n_g := \Fr(g\inv) (g) \in N_\sfG(\sfA)$ and the image of  $n_g$ in $W_\sfG$ is $w$.   Let $\sfA_\theta = (\bigcap_{\alpha \in \theta} \ker (\alpha ) )^\circ$ and $\sfM_\theta = C_{\sfG}(\sfA_\theta)$.   Since $\Fr(\theta) = w \theta$, we have
  $$\Fr(\lsup{g} \sfM_\theta ) = \lsup{\Fr(g)} \Fr(\sfM_\theta ) =  \lsup{g n_g\inv} (\sfM_{\Fr(\theta)})  = \lsup{g}(\lsup{ n_g\inv} (\sfM_{w\theta})) = \lsup{g}  \sfM_\theta.$$
 Thus,  $\lsup{g} \sfM_\theta$ is a twisted Levi $\ff$-subgroup of $\sfG$.

 The only choice in the above construction was $g$.  We need to show that a different choice results in a twisted Levi $\ff$-subgroup that is $\sfG^\Fr$-conjugate to $\lsup{g}\sfM_\theta$.  Suppose $h \in \sfG$ is chosen so that $n_h := \Fr(h)\inv h $ has image $w$ in $W_\sfG$.   Choose $a \in \sfA$ so that $n_h = n_g  a$.    Since $\Fr(hg\inv)\inv hg\inv = {\lsup{g} a} \in {\lsup{g}\sfA}$ and $\cohom^1(\Fr,\lsup{g}\sfA)$ is trivial, there exists $t \in {\lsup{g}\sfA}$ such that $\Fr(hg\inv)\inv h g\inv = \Fr(t)\inv t$.  Thus $tgh\inv = \Fr(tgh\inv)$, so $tgh\inv \in \sfG^{\Fr}$ and
 $$\lsup{g}\sfM_\theta = \lsup{tg}\sfM_\theta = \lsup{(tgh\inv) h}\sfM_\theta.$$
 This shows that $\lsup{h}\sfM_\theta$ is rationally conjugate to $\lsup{g}\sfM_\theta$, and so $\varphi$ is well defined.

We now show that $\varphi$ descends to an injective map from $I_\sfG/\!\sim$ to $\tilde{ \LL}$; we shall call this map $\varphi$ as well.  Suppose  $(\theta, w)$ and $(\theta',w') \in I_\sfG$.  Choose $g \in \sfG$  (resp.\ $g' \in \sfG$) so that the image of  $n_g = \Fr(g)\inv g$ (resp.\ $n_{g'} = \Fr(g')\inv g'$) in $W_\sfG$ is $w$ (resp.\ $w'$).  If $\varphi(\theta,w) = \varphi(\theta',w')$, then there is $x \in \sfG^\Fr$ so that $\lsup{g}\sfM_\theta = \lsup{xg'} \sfM_{\theta'}$.  Without loss of generality, we may replace $g'$  by $xg'$.  Since $\lsup{g} \sfM_\theta = {\lsup{g'} \sfM_{\theta'}}$, we have  that both $\lsup{g} \sfA$ and $\lsup{g'} \sfA$ are maximal $\ff$-tori in  $\lsup{g} \sfM_\theta$.   Consequently, there exists $m' = \lsup{g}m$ with $m \in \sfM_\theta$ such that $\lsup{m'g'}\sfA = \lsup{g}\sfA$ and $\lsup{m'g'}(\sfB \cap \sfM_{\theta'}) = \lsup{g}(\sfB \cap \sfM_\theta)$.  Since we also have $\lsup{\Fr(m')}(\lsup{g'}\sfA) = \lsup{g}\sfA$, we conclude that $m' \Fr(m')\inv \in N_{\lsup{g}\sfM_\theta}(\lsup{g}\sfA)$, which implies $n_g \inv \Fr(m) n_g m\inv  \in N_{\sfM_\theta}(\sfA)$.  Set $n = mg\inv g' \in N_\sfG(\sfA)$.  Note that $n \theta' = \theta$  and 
$$\Fr(n)( n_{g'}) n\inv  = \Fr(mg\inv g') \Fr(g')\inv g' (mg\inv g')\inv = \Fr(m) ( n_g) m\inv =  n_g (n_g\inv \Fr(m) ( n_g) m\inv).$$
By looking at images in $W_\sfG$, we conclude that $\Fr(\dot{n})w' \dot{n}\inv \in w W_{\sfG,\theta}$ where $\dot{n}$ denotes the image of $n$ in $W_\sfG$.
Choose $x \in W_{\sfG,\theta}$ so that $\Fr(\dot{n})w' \dot{n}\inv = wx$.  Note that 
$$x \theta = w\inv \Fr(\dot{n})w' \dot{n}\inv \theta = w\inv \Fr(\dot{n})w' \theta' = w\inv \Fr(\dot{n}) \Fr(\theta') =  w\inv \Fr(\theta) = \theta.$$
Since the action of $W_{\sfG,\theta}$ on the set of  bases for the root system spanned by $\theta$ is simply transitive, we must have $x = 1$.

Finally, we show that $\varphi$ is surjective.  Suppose $\sfL \in \LL$.
Let $\sfS_{\sfL}$ denote the connected component of the center of $\sfL$.  Choose a Borel $\ff$-subgroup $\sfB_{\sfL}$ in $\sfL$ and a maximal $\ff$-torus $\sfA_{\sfL}$ in $\sfB_{\sfL}$.  Denote by $\Delta_{\sfL} = \Delta(\sfL, \sfB_\sfL, \sfA_\sfL)$ the corresponding set of simple roots.
Choose a Borel $\ffc$-subgroup $\sfB' \leq \sfG$ such that $\sfB_{\sfL} = \sfB' \cap \sfL$ and $\Delta_{\sfL} \subset \Delta(\sfG, \sfB', \sfA_\sfL)$.  Choose $g \in \sfG$ so that $\sfA_\sfL = \lsup{g} \sfA$ and $\sfB' = \lsup{g}\sfB$.
Define $\theta_{\sfL} = g\inv \cdot \Delta_{\sfL}$; note that $\theta_\sfL \subset \Delta_\sfG$.   Let $w_{\sfL}$ denote the image of $\Fr(g\inv)g$ in $W_\sfG$ and put  $\sfA_{\theta_\sfL}  = (\bigcap_{\alpha \in \theta_{\sfL} } \ker (\alpha ) )^\circ \leq \sfA$.   
We have 
\begin{itemize}
\item $\sfS_\sfL = {\lsup{g} \sfA_{\theta_\sfL}}$ and
\item $\Fr (\theta_\sfL) = w_\sfL \theta_\sfL$.
\end{itemize}
By construction, $\varphi(\theta_\sfL, w_\sfL)$ is the $\sfG^\Fr$-conjugacy class of $\sfL$.
\end{proof}

\begin{example}  \label{ex:g2finiteinitial} For the group $\sfG = \G2$ let $\Delta_\sfG = \{\alpha, \beta \}$ where $\alpha$ is short.  Let $w_\alpha$ and $w_\beta$ denote the corresponding simple reflections in $W_\sfG$, and let $c$ denote the Coxeter element $w_\alpha w_\beta$.  In Table~\ref{table:gtwofinite}  we provide a complete list of representatives for the elements of $I_\sfG/ \! \sim$.   We also indicate the type of the corresponding twisted Levi $\ff$-subgroup of $\G2$.    Groups of type $A_1$ for a long root are ornamented with a tilde (e.g. $\widetilde{U}(2)$).

\begin{table}[ht]
\centering
\begin{tabular}{ |c||c| }
 \hline
 Pair & Type of twisted Levi $\ff$-group \\
 \hline \hline
 $(\emptyset, 1)$ & $\GL_1 \times \GL_1$ \\
 $(\emptyset, w_\alpha)$ & Coxeter torus in  ${\GL}_2$ \\
 $(\emptyset, w_\beta)$ & Coxeter torus in $\widetilde{\GL}_2$ \\
 $(\emptyset, c )$ & Coxeter torus in $\G2$ \\
 $(\emptyset, c^2)$ & Coxeter torus in $\SL_3$\\
 $(\emptyset,c^3)$ & Coxeter  torus in $\SO_4$ \\
 $(\{\alpha \}, 1)$ & $\GL_2$ \\
 $(\{\alpha \},  w_\beta w_\alpha w_\beta w_\alpha w_\beta)$ & $U(2)$ \\
 $(\{\beta\}, 1)$ &  $\widetilde{\GL}_2$ \\
 $(\{\beta \}, w_\alpha w_\beta w_\alpha w_\beta w_\alpha)$ &  $\widetilde{U}(2)$ \\
 $( \Delta, 1)$ & $\G2$ \\
 \hline
 \end{tabular}
 \caption{$\G2$: A set of representatives for $I_\sfG/ \! \sim$ \label{table:gtwofinite}}
\end{table}
\end{example}

\begin{example}  For the group $\sfG = \SU(3)$ let $\Delta_\sfG = \{\alpha, \beta \}$ with  $\Fr(\alpha) = \beta$.  Let $w_\alpha$ and $w_\beta$ denote the corresponding simple reflections in $W_\sfG$.  In Table~\ref{table:suthreefinite}  we provide a complete list of representatives for the elements of $I_\sfG/ \! \sim$.   We also indicate the type of the corresponding twisted Levi $\ff$-subgroup of $\SU(3)$.   When viewing the table, we find it useful to remember that $(\emptyset,1)$ is equivalent to $(\emptyset, w_\alpha w_\beta)$.

\begin{table}[ht]
\centering
\begin{tabular}{ |c||c| }
 \hline
 Pair & Type of twisted Levi $\ff$-group \\
 \hline \hline
 $(\emptyset,1)$ & maximally $\ff$-split maximal $\ff$-torus in $\SU(3)$ \\
 $(\emptyset, w_\alpha)$ & maximal $\ff$-anisotropic torus in $\U(2)$ \\
  $(\emptyset, w_\alpha w_\beta w_\alpha )$ & maximal $\ff$-anisotropic torus in $\SU(3)$\\
  $(\{\alpha\}, w_\alpha w_\beta)$ & $\U(2)$ \\
  $( \Delta, 1)$ & $\SU(3)$ \\
 \hline
 \end{tabular}
 \caption{$\SU(3)$: A set of representatives for $I_\sfG/ \! \sim$  \label{table:suthreefinite}}
\end{table}
\end{example}

\section{Tori and Levi subgroups}  \label{sec:littlereview}

In this section we introduce some notation and recall some facts about tori and Levi subgroups.

\subsection{Facts about Levi \texorpdfstring{$(E,k)$-subgroups}{Ek subgroups}}
For this subsection suppose that $k$ is any field and $E$ is a Galois extension of $k$.
A subgroup $\bM$ of $\bG$ will be called a \emph{Levi $(E,k)$-subgroup} provided that it is a $k$-subgroup that occurs as a Levi factor for a parabolic $E$-subgroup of $\bG$. 

\begin{lemma} \label{lem:morelevi}
If $\bM$ is a Levi $(E,k)$-subgroup and $\bZ_{\bM}^E$ is the maximal $E$-split torus in the center of $\bM$, then $\bZ_{\bM}^E$ is defined over $k$ and $\bM = C_\bG(\bZ_{\bM}^E)$. 
\end{lemma}

\begin{proof}  Since $\bM$ is defined over $k$, the center of $\bM$ is defined over $k$.   Thus $\bZ_{\bM}^E$, the unique maximal $E$-split sub-torus of the center of $\bM$,  is also defined over $k$.

Since $\bM$ is a Levi $(E,k)$-subgroup of $\bG$, there is a parabolic $E$-subgroup $\bP$ of $\bG$ such that $\bM$ is a Levi factor for $\bP$.
Thanks to the proof of~\cite[Proposition~16.1.1 (ii)]{springer:algebraic} $\bP$ contains an  $E$-split torus $\bS$ of $\bG$ such that $C_{\bG}(\bS)$ is a Levi factor of $\bP$.  Since $\bM$ and $C_{\bG}(\bS)$ are $\bP(E)$-conjugate~~\cite[Proposition~16.1.1 (ii)]{springer:algebraic}, we may assume, after conjugating $\bS$ by an element of $\bP(E)$, that  $\bM = C_{\bG}(\bS)$.   Since $\bS$ is in the center of $\bM$ and is $E$-split, we conclude that  $\bS \leq \bZ_{\bM}^E$.  Thus
$\bM \leq C_{\bG}(\bZ_{\bM}^E) \leq C_{\bG}(\bS) = \bM$,
and we conclude that $\bM = C_{\bG}(\bZ_{\bM}^E)$.
\end{proof}

\begin{cor}  \label{cor:3.1.2}
Suppose $\bM$ is a Levi $(E,k)$-subgroup with center $\bZ_\bM$.  If $\bH$ is a $k$-subgroup of $\bG$ that lies between $\bZ_{\bM}^E$ and $\bZ_{\bM}$, then $\bM = C_{\bG}(\bH)$.
\end{cor}

\begin{proof}
Since $\bZ_{\bM}^E \leq \bH \leq \bZ_{\bM}$, we have
$\bM \leq C_\bG(\bZ_\bM) \leq C_{\bG}(\bH) \leq C_{\bG}(\bZ_{\bM}^E) = \bM$.
\end{proof}

\begin{lemma}\label{lem:!levi}  If  $\bT$ is a maximal $k$-torus in $\bG$, then $C_\bG(\bT^E)$ is the unique Levi $(E,k)$-subgroup in $\bG$ that is minimal among Levi $(E,k)$-subgroups that contain $\bT$.
\end{lemma}

\begin{rem}
The condition of maximality on $\bT$ is necessary for uniqueness.
\end{rem}

\begin{proof}  
Let $\bM = C_{\bG}(\bT^E)$.    

We first show that $\bM$ is a Levi $(E,k)$-subgroup in $\bG$.
Since $\bT^E$ is a torus, $\bM$ is a Levi subgroup (see~\cite[Proposition~1.2.2]{digne-michel}).  Since $\bT^E$ is the unique maximal $E$-split subtorus in $\bT$,  it is defined over $k$.  Hence $\bM$ is defined over $k$.
If $\bZ^E_{\bM}$ denotes the maximal $E$-split torus in the center of $\bM$, then $\bT^E \leq \bZ^E_{\bM}$.  Moreover, since $\bT \leq \bM$, we also have $\bZ^E_{\bM} \leq \bT$, and so by maximality of $\bT^E$ we conclude that $\bZ^E_{\bM} \leq \bT^E$.  Thus $\bT^E = \bZ^E_{\bM}$, and so $\bM = C_{\bG}(\bT^E)
=C_{\bG}(\bZ^E_\bM)$.

Let $\bT'$ denote a maximal $E$-split torus in $\bG$ that contains $\bT^E$.   Choose $\lambda \in \X_*^E(\bT')$ that is non-trivial on every root of $\bG$ with respect to $\bT'$; let $\bP_\lambda$ denote the corresponding parabolic $E$-subgroup of $\bG$.   Since $\bT' \leq \bM$, the subgroup $\bM \bP_\lambda$ of $\bG$ is a parabolic $E$-subgroup of $\bG$ for which $\bM$ is a Levi factor.  

We now show that  $\bM$ is the unique minimal Levi $(E,k)$-subgroup in $\bG$ that contains $\bT$.
Suppose $\bM'$ is another Levi $(E,k)$-subgroup that contains $\bT$. Let $\bZ^E_{\bM'}$ denote the maximal $E$-split torus in the center of $\bM'$. From Lemma~\ref{lem:morelevi} we have $\bM' = C_{\bG}(\bZ^E
_{\bM'})$.  Since $\bT \leq \bM'$, we have $\bZ^E_{\bM'} \leq \bT$ and so $\bZ^E_{\bM'} \leq \bT^E = \bZ^E_{\bM}$; hence $\bM \leq \bM'$.
\end{proof}

\subsection{On unramified twisted Levi subgroups and unramified tori}
We again assume that $k$ is complete with respect to a nontrivial discrete valuation and $\ff$ is perfect.   When $E$ is the maximal unramified extension $K$, we use the following language. 

\begin{defn}
A subgroup $\bL$ of $\bG$ is called an \emph{unramified twisted Levi subgroup in $\bG$} provided that $\bL$ is a Levi $(K,k)$-subgroup of $\bG$.
\end{defn}

\begin{defn}
A $k$-torus $\bS$ is called an \emph{unramified torus in $\bG$} provided that $\bS$ is the $K$-split component of the center of an unramified twisted Levi subgroup in $\bG$.  
\end{defn}

\begin{rem}
Lemma~ \ref{lem:morelevi} tells us that 
if $\bL$ is an unramified twisted Levi subgroup in $\bG$ and $\bS$ is the $K$-split component of the center of $\bL$, then $\bL = C_{\bG}(\bS)$.
\end{rem}

\begin{rem}  \label{rem:421}
Since two Levi $(K,k)$-subgroups are $G^{\Gamma}$-conjugate if and only if the $K$-split components of their centers are $G^{\Gamma}$-conjugate, any parameterization of $G^{\Gamma}$-conjugacy classes of unramified tori is also a  parameterization of  $G^{\Gamma}$-conjugacy classes of Levi $(K,k)$-subgroups.
\end{rem}

\begin{lemma}  \label{lem:torusexists}  Suppose $\ff$ is finite.
A torus $\bT$ in $\bG$ is unramified in $\bG$ if and only if there exists
a maximal $k$-torus $\bT'$ in $\bG$ such that $\bT$ is the maximal $K$-split subtorus of $\bT'$.  
\end{lemma}

\begin{proof}
Suppose $\bT$ is the $K$-split component of the center of a Levi $(K,k)$-subgroup $\bL$.   
From Appendix~\ref{sec:appendixone} 
there is a  $K$-minisotropic maximal $k$-torus, call it $\bT'$, in $\bL$.   Then $\bT$ is the maximal $K$-split subtorus of $\bT'$.

Suppose 
there exists a maximal $k$-torus $\bT'$ in $\bG$ for which $\bT$ is the maximal $K$-split subtorus of $\bT'$.   Since $\bT'$ is defined over $k$, so too is $\bT$.    Let $\bL = C_{\bG}(\bT)$.  From Lemma~\ref{lem:!levi}, we know $\bL$ is the unique minimal Levi $(K,k)$-subgroup containing $\bT'$.   By construction, $\bT$ is contained in $\bZ_\bL$, the center of $\bL$.  Thus $\bT = \bT^K \leq \bZ_\bL^K$.   Since $\bT'$ contains $\bZ_\bL$ and $\bT$ is the maximal $K$-split subtorus of $\bT'$, we conclude that $\bZ_\bL^K$ is contained in  $\bT$.  Hence $\bT = \bZ_\bL^K$, and so $\bT$ is an unramified torus in $\bG$.
\end{proof}

\subsection{A question about unramified tori} 
Suppose $\ff$ is quasi-finite and 
$\bS$ is a torus that occurs as the center of a Levi $(K,K)$-subgroup. Based on our experience with maximal $K$-split tori of $\bG$, it is natural to ask: 
\begin{question} \label{quest:frunramified}
If $\Fr(\bS)$ is $G$-conjugate to $\bS$, then does there exist 
a $g \in G$ for which $\lsup{g}\bS$ is defined over $k$?   That is, does $\lsup{G}\bS$ contain an unramified torus?
\end{question} 

Unfortunately, the answer is no as the following example illustrates. 

Let $\bH$ be a connected reductive group of type $\typeA_{2}$ such that $H^{\Fr} \cong \SL_1(D)$ where $D$ is an unramified  division algebra of index $3$ over $k$.  Recall that we may identify  $H = \bH(K)$ with $\SL_3(K)$.  Let $\bA$ denote a maximal $K$-split $k$-torus in $\bH$; it is unique up to $H^\Fr$-conjugacy.   Suppose $C$ is the alcove in $\AA(A) \leq \BB(H)$ for which $C^{\Fr} \neq \emptyset$.  Let $\{\psi_0, \psi_1 ,  \psi_{2} \}$ be the simple affine $K$-roots determined by $\bH$, $\bA$, $\nu$, and $C$.   We assume that the $\psi_i$ are labeled so that $\Fr(\psi_i) = \psi_{i+1}$ mod $3$.  

Let $\bA_j = (\ker \dot{\psi}_j)^\circ$ for $j \in \{1,2\}$.  Note that $\Fr(\bA_1) = \bA_2$.   Since there exists $n \in N_H(\bA)$ for which $n \dot{\psi}_2 = \dot{\psi}_1$, we conclude that  $\Fr (\bA_1)$ is $H$-conjugate to  $\bA_1$.  Suppose some conjugate, call it $\bS$, of $\bA_1$ is defined over $k$.   Let $\bT$ be a maximal $k$-torus in the Levi $k$-subgroup $C_\bH(\bS)$.   Then $\bT$ corresponds to an extension $E \leq D$ of degree $3$ over $k$, and $\bS$ corresponds to a quadratic extension of $k$ that lies in $E$.  Since no such quadratic extension exists, we conclude that the answer to Question~\ref{quest:frunramified} is no.

\section{On \texorpdfstring{$(K,k)$- tori}{Kk tori}} \label{sec:Kktori}

In this section we assume that the residue field $\ff$ of $k$ is perfect.

 If $E$ is a tame Galois extension of $k$ and $M$ is the group of $K$-rational points of a Levi $(E,k)$-subgroup of $G$, then we may and do identify $\BB(M)$ with a subset of $\BB(G)$.   There is no canonical way to do this, but all such identifications have the same image.

Recall that  $\bZ=\bZ_\bG$ denotes the center of $\bG$.

\begin{defn}
A torus in $\bG$ will be called a \emph{$(K,k)$-torus in $\bG$} provided that it is a $K$-split $k$-torus that contains $\bZ^K$.
A \emph{maximally $k$-split maximal $(K,k)$-torus $\bT$ in $\bG$}   is  a maximal $(K,k)$-torus $\bT$  in $\bG$  such that $\dim(\Lie(\bT^k)) \geq \dim(\Lie(\tilde{\bT}^k))$ for all maximal $(K,k)$-tori $\tilde{\bT}$ in $\bG$.
\end{defn}

To ease the notation, if $S = \bS(K)$ for a $(K,k)$-torus $\bS$ in $\bG$, then we will call $S$ a $(K,k)$-torus as well.

We let $\toriK$ denote the set of $(K,K)$-tori in $\bG$.  The set $\toriK$ carries a natural action of $\Gamma$, and we denote the set of $\Gamma$-fixed points in $\toriK$ by $\toriKGamma$.

The goal of this section is to describe the $G^{\Gamma}$-conjugacy classes in $\toriKGamma$.

\subsection{Some indexing sets}
\label{sec:someindexing}

To understand the elements of $\toriKGamma$, we introduce indexing sets that arise naturally from Bruhat-Tits theory.
For a facet $F$ in $\BB(G)$, we let $\sfZ_F$ denote the group 
corresponding to the image of $G_F \cap \bZ^{K}(K)$ in  $ G_F/G_{F,0^+} = \sfG_F$.
Consider the indexing set
$$J := \{(F, \sfS) \, | \, \text{$F$ is a facet in $\BB(G)$ and $\sfS$ is 
a torus in $\sfG_F$ that contains $\sfZ_F$} \}.$$

The following definition provides a link between $\toriK$ and $J$.
\begin{defn}  \label{defn:lift}
A $K$-split torus $\bS \in \toriK$ is said to be a \emph{lift} of $(F,\sfS) \in J$ provided that
\begin{enumerate}
\item $F \subset \BB(C_G(S))$ 
\item  the image of $S \cap G_F$ in $\sfG_F = G_F/G_{F,0^+}$ is $\sfS$.
\end{enumerate}
\end{defn}

\begin{rem} \label{rem:starone}
It follows from~\cite[4.4.2]{debacker:nilpotent} that if $\bS \in \toriK$ and $y\in \BB(C_G(S))$, then a point $x \in \BB(G)$ is $(S\cap G_y)$-fixed if and only if $x \in \BB(C_G(S))$.
\end{rem}

Suppose $(F,\sfS) \in J$.  Note that if $\Gamma(F) = F$, then $\sfG_F$ is defined over $\ff$.  In this situation, it makes sense to consider  $\Gamma(\sfS)$.
We define
$$\JGamma := \{(F, \sfS) \in J \, | \,  \text{$\Gamma(F) = F$ \text{and} $\Gamma(\sfS) = \sfS$} \}.$$

\begin{defn}
A pair $(F, \sfS) \in \JGamma$ will be called \emph{maximal}  provided that whenever a facet $H \subset \BB(G)$ is both $\Gamma$-stable and contains $F$ in its closure, then $\sfS$ belongs to the $\ff$-parabolic subgroup $G_H/G_{F,0^+}$ of $\sfG_F = G_F/G_{F,0^+}$ if and only if $F = H$.  
We let $\JGammamax$ denote the subset of maximal pairs in $\JGamma$.
\end{defn}

\begin{example} \label{ex:notintuitive}
Suppose $(F,\sfS) \in \JGammamax$.  If $\sfS$ is $\ff$-split, then $F^\Gamma$ is a $G^\Gamma$-alcove in $\BB(G)^\Gamma$.   
\end{example}

More generally,
suppose $(F,\sfS) \in  \JGammamax$.   As in Section \ref{sec:buildings} one can attach to $F$ a $\GGamma$-conjugacy class $(\lsub{F}\bM)$ of Levi $(k,k)$-subgroups.  From Example~\ref{ex:notintuitive}, one expects that if $\bS \in \toriKGamma$ is a lift of $(F,\sfS)$, then $(\lsub{F}\bM)$ is the minimal conjugacy class of Levi $(k,k)$-subgroups for which $\bS \leq \bM$ for some $\bM \in (\lsub{F}\bM)$.  This is true and will be proved in Section~\ref{sec:toriandlevi2}.

\subsection{Passing between tori over \texorpdfstring{$\ff$ and tori over $k$}{f and tori over k}}

Suppose $(F,\sfS) \in \JGamma$.  Our next two lemmas show that there is an element of $\toriKGamma$ that  lifts $(F,\sfS)$, and any two lifts of $(F,\sfS)$ are conjugate by an element of $G_{F,0^+}^{\Gamma}$.

  \begin{lemma}  \label{lem:AandS} Set $\sfM = C_{\sfG_F}(\sfS)$ and let $\sfT$ denote 
  a maximally $\ff$-split
  maximal $\ff$-torus in $\sfM$.   There is a torus $\bT \in \toriKGamma$ that lifts  $(F, \sfT)$.  Moreover, for all $\bT \in \toriKGamma$ lifting $(F,\sfT)$ there exists a unique lift $\bS \in \toriKGamma$ of $(F, \sfS)$ with the property that $\bS \leq \bT$.
  \end{lemma}

\begin{rem}
If $\sfM$ is $\ff$-quasi-split,
then $\sfT$ is the centralizer of a maximal  $\ff$-split torus in $\sfM$.
\end{rem}

\begin{proof}
Suppose $\bT \in \toriKGamma$ is a lift of  $(F,\sfT)$.  (Such a torus $\bT$ exists by~\cite[Lemma~2.3.1]{debacker:unramified}.)
Note that $\X_*(\sfT) = \X_*(\bT)$ as $\Gamma$-modules, and we can therefore choose a subtorus $\bS$ of $\bT$  corresponding to the image of $\X_*(\sfS)$ under  $\X_*(\sfS) \hookrightarrow \X_*(\sfT) = \X_*(\bT)$.   We have that $\bS \in \toriKGamma$.  Since $T \leq C_G(S)$ and $F \subset \BB(T) \subset \BB(C_G(S))$, we conclude that $S$ is a lift of $(F, \sfS)$.

If $\bS' \in \toriKGamma$ is another lift of $(F, \sfS)$ that lies in $\bT$,  then $\X_*({\bS'}) = \X_*(\bS) = \X_*(\sfS)$ in $\X_*(\bT)$, and so $\bS' = \bS$.
\end{proof}

\begin{cor}  \label{cor:heart}
If $\bS, \bS' \in \toriKGamma$ both lift $(F,\sfS)$, then there exists an element $g \in G_{F,0^+}^{\Gamma}$ such that $\lsup{g}\bS = \bS'$.
\end{cor}

\begin{proof}
We will use the notation of Lemma~\ref{lem:AandS} and its proof.

Set  ${\bM'} = C_{\bG}(\bS')$.   Note that $F \subset \BB(M')$ and the image of $M' \cap G_F$ in $\sfG_F$ is $\sfM = C_{\sfG_F}(\sfS)$.   Let $\bT' \leq \bM'$ be a lift of $(F,\sfT)$. 
Since $\bS'$ is in the center of $\bM'$, we have $\bS' \leq \bT'$.   Since $\bS'$ (resp. $\bT'$) is a $K$-split torus lifting $(F,\sfS)$ (resp. $(F,\sfT)$), from Lemma~\ref{lem:AandS} we conclude that $\bS'$ is the unique lift of $(F,\sfS)$ in $\bT'$.
By~\cite[Lemma~2.2.2]{debacker:unramified}, there is an element $g \in G_{F,0^+}^{\Gamma}$ such that $\lsup{g}\bT = \bT'$.  The result follows from Lemma~\ref{lem:AandS}.
\end{proof}

\begin{rem}  \label{rem:handsonM}
Suppose $\sfC$ is the maximal $\ff$-split component of the center of $\sfG_F$.   If $\bC$ is a lift of $(F, \sfC \sfZ_F)$, then $C_\bG(\bC)$ is $G^{\Gamma}_{F,0^+}$-conjugate to $\lsub{F}\bM$.
\end{rem}

Thanks to Lemma~\ref{lem:AandS} and Corollary~\ref{cor:heart} we can  define an action of $\GGamma$ on $\JGammamax$.  Suppose $g \in \GGamma$ and $(F, \sfS) \in \JGammamax$.  Let $\bS$ be a lift of $(F,\sfS)$.   Let $\lsup{g}\sfS$ denote the image of $\lsup{g}S \cap G_{gF}$ in $\sfG_{gF}$ and set  $g(F,\sfS) := (gF, \lsup{g}\sfS) \in \JGammamax$.

The following lemma allows us to move in the opposite direction: from tori over $k$ to tori over $\ff$.

\begin{lemma}  \label{lem:P}
For all $\bS \in \toriKGamma$ there exists $(F,\sfS) \in \JGammamax$ such that $\bS$ lifts $(F, \sfS)$.    
\end{lemma}

\begin{proof}
Fix $\bS \in \toriKGamma$.  Let $\bM = C_{\bG}(\bS)$.  Note that $\bM$ is a Levi $(K,k)$-subgroup of $\bG$.

Choose a $\Gamma$-stable $M$-facet $H$ in $\BB(M)$ so that  $H^\Gamma$ is an $M^\Gamma$-alcove; the choice of $H$ is unique up to $M^\Gamma$-conjugation.  Since $H$ can be written as the disjoint union of $G$-facets in $\BB(G)$, we may choose  a $\Gamma$-stable $G$-facet $F$ in $H$ so that  $\dim(F^\Gamma) \geq \dim(\tilde{F}^\Gamma)$ for  all $\Gamma$-stable $G$-facets $\tilde{F}$ in $\bar{H}$. 
In fact, $\dim(F^\Gamma) \geq \dim(\tilde{F}^\Gamma)$ for all $\Gamma$-stable $G$-facets $\tilde{F}$ in $\BB(M)$: if $\tilde{F}$ is a $\Gamma$-stable  $G$-facet in $\BB(M)$, then, since $M^\Gamma$-alcoves are $M^\Gamma$-conjugate, there is some $m \in M^\Gamma$ that carries $\tilde{F}$ into $\overline{H}$; thus $\dim(\tilde{F}^\Gamma) = \dim((m \tilde{F})^\Gamma) \leq \dim(F^\Gamma)$.

Let $\sfS$ be the $\ff$-torus in $\sfG_F$ corresponding to the image of $S \cap G_F$ in $\sfG_F$.  By construction, the torus $\bS$ is a lift of the pair $(F,\sfS)$.  We need to show that $(F,\sfS) \in \JGammamax$.   Suppose $F' \subset \BB(G)$ is a $\Gamma$-stable $G$-facet with $F \subset \overline{F}'$ and $F \neq F'$.  Note that $S \cap G_{F'} \subset S \cap G_{F}$.    If $\sfS$ belongs to the proper parabolic  $\ff$-subgroup $G_{F'}/G_{F,0^+}$ of $\sfG_F = G_F/G_{F,0^+}$, then $S \cap G_F = S \cap G_{F'}$ fixes $F'$ and $(F',\sfS') \in \JGamma$ where $\sfS'$ is the $\ff$-torus in $\sfG_{F'}$ corresponding to the image of $S \cap G_{F'}$ in $\sfG_{F'}$.  From Remark~\ref{rem:starone} we conclude that $F' \subset \BB(M)$.  But the dimension of $F'^\Gamma$ is strictly larger than that of $F^\Gamma$, contradicting that $\dim(F^\Gamma) \geq \dim(\tilde{F}^\Gamma)$ for all $\Gamma$-stable $G$-facets $\tilde{F}$ in $\BB(M)$.
\end{proof}

\subsection{An equivalence relation}

Thanks to Lemma~\ref{lem:P} and Corollary~\ref{cor:heart}, we have a surjective map $\varphi$ from $\JGammamax$ to the set of $\GGamma$-conjugacy classes in $\toriKGamma$.   We introduce an equivalence relation, to be denoted $\sim$, on $\JGammamax$ so that $\varphi$ descends to a bijection between $\JGammamax / \! \sim $  and the  set of $\GGamma$-conjugacy classes in $\toriKGamma$.

Suppose  $\AA$ is an apartment in  $\BB(G)^{\Gamma}$.  This means there is a maximal $k$-split torus $\bS \leq \bG$  such that $\AA = \AA(\bS,k)$.  Note that every such apartment is equal to $\AA(\lsup{g}A)^\Gamma$ for some $g \in G^\Gamma$.  

If $\Omega \subset \AA$, then we denote the smallest affine subspace of $\AA$ that contains $\Omega$ by  $A(\AA,\Omega)$.  When $F_1, F_2$ are two $G^{\Gamma}$-facets in $\AA$ for which $\emptyset \neq A(\AA,F_1) = A(\AA,F_2)$, then, since an affine root vanishes on $F_1$ if and only if it vanishes on $F_2$, there is a natural identification of  $\sfG_{F_1}$ with $\sfG_{F_2}$ as algebraic $\ff$-groups.  When this happens, we write
 $\sfG_{F_1} \overset{i}{=} \sfG_{F_2}$.

\begin{defn} \label{defn:531} Suppose $(F_i, \sfS_i) \in \JGamma$.  We write $(F_1, \sfS_1) \sim (F_2, \sfS_2)$ provided that there exist an element $g \in \GGamma$ and an apartment $\AA$ in $\BB(G)^{\Gamma}$ such that 
\begin{enumerate}
\item $\emptyset \neq A(\AA,F_1^\Gamma) = A(\AA,gF_2^\Gamma)$
\item $\sfS_1 \overset{i}{=}  {\lsup{g}\sfS_2}$ in $\sfG_{F_1} \overset{i}{=} \sfG_{gF_2}$.
\end{enumerate}
\end{defn}

\begin{lemma}  The relation $\sim$ is an equivalence relation on $\JGammamax$. 
\end{lemma}

\begin{proof} The proof is nearly identical to the material in ~\cite[Section 3.6]{debacker:nilpotent} or ~\cite[Section 3.2]{debacker:unramified}. 
\end{proof}

\subsection{A bijective correspondence}

Suppose $\bS \in \toriKGamma$ is a lift of $(F, \sfS) \in \JGammamax$.  Let $\bM = C_\bG(\bS)$.  Recall from Definition~\ref{defn:lift} that since $\bS$ is a lift of $(F,\sfS)$, we have $F \subset \BB(M)$.   

\begin{lemma} \label{lem:alcoveinM}
Let $C$ denote a $\Gamma$-stable $M$-facet in $\BB(M)$ that contains $F$ in its closure.  The $M^{\Gamma}$-facet $C^{\Gamma}$ is an alcove in $\BB(M)^{\Gamma}$ and $F^{\Gamma}$ is an open subset of $C^{\Gamma}$.
\end{lemma}

\begin{proof}
It will be enough to show that $F^{\Gamma}$ is a maximal $\GGamma$-facet in $\BB(M)^{\Gamma}$.  Choose a $G^{\Gamma}$-facet $D \subset \BB(M)^{\Gamma}$ such that $F^\Gamma \subset \overline{D}$.   If $F^\Gamma \neq D$, then as $\bS$ is in the center of $\bM$, the image of $S \cap G_F = S \cap G_D$ in $G_F/G_{F,0^+}$ belongs to the parabolic $\ff$-subgroup $G_D/G_{F,0^+}$.  This contradicts that $(F,\sfS) \in \JGammamax$.
\end{proof}

\begin{lemma}
Suppose $(F_i, \sfS_i) \in \JGammamax$ with lift $\bS_i \in \toriKGamma$.  If there exists $g \in \GGamma$ such that $\lsup{g}\bS_1 = \bS_2$, then $(F_1, \sfS_1) \sim (F_2, \sfS_2)$.  
\end{lemma}

\begin{proof}
After replacing $(F_1,\sfS_1)$ with $(gF_1, \lsup{g}\sfS_1)$ we may and do assume that $\bS_1 = \bS_2$.  Without loss of generality, we assume $\bS = \bS_1 = \bS_2$, and set $\bM = C_\bG(\bS)$.  
Since $\bS$ is a lift of $(F_i, \sfS_i)$, from  Definition~\ref{defn:lift} we have $F_i \subset \BB(M)$.      Let $C_i$ denote the $M$-facet in $\BB(M)$ to which $F_i$ belongs.    By Lemma~\ref{lem:alcoveinM}, $C_i^{\Gamma}$ is an $M^{\Gamma}$-alcove in $\BB(\MGamma)$.  Since $\MGamma$ acts transitively on the alcoves in $\BB(M)^\Gamma$, there exists an $m \in \MGamma$ such that $m C_1 = C_2$.    We may and do replace $(F_1, \sfS_1)$ by $(mF_1, \lsup{m} \sfS_1)$.  Since $F_1^\Gamma$ and $F_2^\Gamma$ are open in $C_1^\Gamma = C_2^\Gamma$, for any apartment $\AA$ in $\BB(M)^\Gamma \subset \BB(G)^\Gamma$ we have $\emptyset  \neq A(\AA,F_1^\Gamma) = A(\AA,F_2^\Gamma)$.    Since $\lsup{m}\bS = \bS$, we conclude that $(F_1, \sfS_1) \sim (F_2, \sfS_2)$.
\end{proof}

\begin{cor}  \label{cor:toogeneralbij}
There exists a bijection between  $ \JGammamax / \sim$  and the set of $\GGamma$-conjugacy classes in   $ \toriKGamma$.
\end{cor}

\begin{proof} 
The only thing remaining to check is that if $(F_1, \sfS_1), (F_2, \sfS_2) \in \JGammamax$ with $(F_1, \sfS_1) \sim (F_2, \sfS_2)$, then they have lifts that are $\GGamma$-conjugate.   Suppose  $(F_1, \sfS_1), (F_2, \sfS_2) \in \JGammamax$ with $(F, \sfS_1) \sim (F_2, \sfS_2)$.   Then there exist an element $g \in \GGamma$ and an apartment $\AA$ in $\BB(G)^\Gamma$ such that 
\begin{enumerate}
\item $\emptyset \neq A(\AA,F_1^\Gamma) = A(\AA,gF_2^\Gamma)$
\item $\sfS_1 \overset{i}{=} {\lsup{g}\sfS_2}$ in $\sfG_{F_1} \overset{i}{=} \sfG_{gF_2}$
\end{enumerate}
We may and do assume that $\AA = \AA(A)^\Gamma$ and that  $g$ is the identity.

Let $\bM_i = \lsub{F_i}\bM$. 
Since   $A(\AA,F_1^\Gamma) = A(\AA,F_2^\Gamma)$, we have $\bM_1 = \bM_2$.  Set $\bM = \bM_1$.  By construction, the image of $M \cap G_{F_i}$ in $\sfG_{F_i}$ is $\sfG_{F_i}$ itself (that is, $\sfM_{F_i} =  \sfG_{F_i}$).

Since $\sfS_1 \overset{i}{=} \sfS_2$ in $\sfG_{F_1} \overset{i}{=} \sfG_{F_2}$, we can find $\bT \in \toriKGamma$ so that the image of $T \cap \lsub{F_1}M \cap \lsub{F_2}M$ in $\sfM_{F_i} =  \sfG_{F_i}$ is a 
maximally $\ff$-split $\ff$-torus containing $\sfS_i$.   As in Lemma~\ref{lem:AandS}  there is exactly one lift $\bS$ of $(F_i,\sfS_i)$ in $\bT$.   Since the image of  $S \cap \lsub{F_1}M \cap \lsub{F_2}M$ in $\sfM_{F_i} =  \sfG_{F_i}$ is $\sfS_i$, the proof is complete.
\end{proof}

\subsection{\texorpdfstring{$(K,k)$-tori and Levi $(k,k)$-subgroups}{Kk tori and Levi kk subgroups}}
\label{sec:toriandlevi2}

If $\bM'$ is a Levi $(k,k)$-subgroup of $\bG$, then we let $(\bM')$ denote the   $G^\Gamma$-conjugacy class of $\bM'$.  Set $\tilde{\mathcal{M}} = \{ (\bM') \colon \text{$\bM'$ is a Levi $(k,k)$-subgroup of $\bG$}\}$.   If $(\bM_1), (\bM_2) \in \tilde{\mathcal{M}}$, then  we write $(\bM_1) \leq (\bM_2)$ provided that there exists $\bL_i \in (\bM_i)$ such that $\bL_1 \leq \bL_2$; this defines a partial order on $\tilde{\mathcal{M}}$.

\begin{lemma}    \label{lem:prelimforM}
Fix $(F,\sfS) \in \JGammamax$ with $F \subset \AA(A)$,  and let $\bS\in \toriKGamma$ be a lift of $(F,\sfS)$.   There exists $\bM' \in (\lsub{F}\bM)$ such that $\bS \leq \bM'$.
\end{lemma}

\begin{proof}
Let $\bM = C_\bG(\bS)$.   
Choose a maximally $k$-split maximal $(K,k)$-torus $\bT$ in $\bM$.    Note that $\bS \leq \bM' := C_\bG(\bT^k)$.
From Lemma~\ref{lem:alcoveinM}, $F^\Gamma$ is a subset of an $M^\Gamma$-alcove; thus, we can replace $\bT$ with an $M^\Gamma$-conjugate and assume $F^\Gamma \subset \BB(T)^{\Gamma} \subset \BB(M)^{\Gamma}$.

From Lemma~\ref{lem:alcoveinM}, $F^\Gamma$ is a maximal $G^{\Gamma}$-facet in $\BB(M)$.   Thus, if $\sfC$ denotes the image of $\bT^k(K) \cap G_F$ in $\sfG_F$, then $\sfC$ is the maximal $\ff$-split component in the center of $\sfG_F$.   Since $\bT^k \bZ^K$ is a lift of $(F,\sfC \sfZ_F)$, from Remark~\ref{rem:handsonM} we have $\bM' = C_\bG(\bT^k)$ is $G_{F,0^+}^\Gamma$-conjugate to $\lsub{F}\bM$.
\end{proof}

\begin{lemma}   \label{lem:howelliptic}
Fix $(F,\sfS) \in \JGammamax$ with $F \subset \AA(A)$, and let $\bS \in \toriKGamma$ be a lift of $(F,\sfS)$.  If $\bM''$ is a Levi $(k,k)$-subgroup of $\bG$ that contains $\bS$, then $(\lsub{F}\bM) \leq (\bM'')$.
\end{lemma}

\begin{proof}
Let $\bM = C_\bG(\bS)$. 

Suppose $\bM''$ is a Levi $(k,k)$-subgroup that contains $\bS$. Recall that  $\bZ_{\bM''}^k$ denotes the $k$-split component of the center of $\bM''$.   Since $\bZ_{\bM''}^k$ commutes with $\bS$, we have $\bZ_{\bM''}^k \leq \bM$.   Choose a maximally  $k$-split maximal $(K,k)$-torus $\bT$ in $\bM$ that contains $\bZ_{\bM''}^k$.

From Lemma~\ref{lem:alcoveinM}, $F^\Gamma$ is a subset of an $M^\Gamma$-alcove; thus, we can replace $\bT$ and $\bM''$ with  $M^\Gamma$-conjugates and assume $F^\Gamma \subset \BB(T)^\Gamma \subset \BB(M)^\Gamma$.  Since, from Lemma~\ref{lem:alcoveinM},  
$F^\Gamma$ is a maximal $G^\Gamma$-facet in $\BB(M)$, we have, as in the proof of Lemma~\ref{lem:prelimforM},  $(\lsub{F}\bM) = (C_\bG(\bT^k))$.  

Since  $\bZ_{\bM''}^k \leq \bT^k$, we conclude that 
 $(\lsub{F}\bM) \leq (\bM'')$.
 \end{proof}

\begin{cor}
Fix $(F,\sfS) \in \JGammamax$ with $F \subset \AA(A)$, and let $\bS \in \toriKGamma$ be a lift of $(F,\sfS)$.  We have
\[ \pushQED{\qed} 
(\lsub{F}\bM) = \min \{ (\bM')\in \tilde{\mathcal{M}} \colon \text{ there exists $\bL' \in (\bM')$ such that $\bS \leq \bL'$} \}. \qedhere 
\popQED
\]  
\end{cor}

\section{A parameterization of unramified tori}
\label{sec:paramunram}

We again assume that $\ff$, the residue field of $k$, is quasi-finite. 

We want to understand unramified tori, that is, those $k$-tori $\bS$ in $\bG$ for which $\bS$ is the $K$-split component of the center of $C_\bG(\bS)$.   

\begin{rem}
Tori such as $\{ \diag(t,t^2,t^{-3}) \}$ in $\SL_3$ are $(K,k)$-tori but are not unramified tori in our sense.
\end{rem}

We begin by recalling a known result (see, for example,~\cite[Lemma~2.3.1]{debacker-reeder:depth-zero}).

\begin{lemma}  \label{lem:parahoriccohom}
If  $F$ is a facet in $\BB(G)^\Fr$, then $\cohom^1(\Fr, G_F) = 1$.
\end{lemma}

\begin{proof}
Note that  $\cohom^1(\Fr, G_{F,0^+}) = 1$ and from Lemma~\ref{lem:qfcohom} we have  $\cohom^1(\Fr, \sfG_F) = 1$.
Since 
$$1 \longrightarrow G_{F,0^+} \longrightarrow G_F \longrightarrow \sfG_F \longrightarrow 1$$
is exact, from~\cite[I Section 5.5: Proposition 38]{serre:galois} we have
$$\cohom^1(\Fr, G_{F,0^+}) \longrightarrow \cohom^1 (\Fr, G_F) \longrightarrow \cohom^1(\Fr,\sfG_F) $$
is exact.  The result follows. 
\end{proof}

\subsection{Comparison with the case of maximal unramified tori}

We already know (see, for example,~\cite{debacker:unramified}) how to describe the set of $G^{\Fr}$-conjugacy classes of  maximal unramified tori of $\bG$ (the situation where $ C_{\bG}(\bS)$ is abelian) in terms of Bruhat-Tits theory.  In this case, there is a bijective correspondence between the set of $G^{\Fr}$-conjugacy classes of maximal unramified tori in $\bG$ and the set of  equivalence classes (as in Definition~\ref{defn:531})
of pairs $(F,\sfS)$ where $F$ is a facet in $\BB(G)^{\Fr}$ and $\sfS$ is a maximal $\ff$-minisotropic torus in $\sfG_F$:
$$\{ \text{maximal unramified tori of $\bG$} \}/ \text{$G^{\Fr}$-conjugacy} \longleftrightarrow \{ \text{$(F,\sfS)$}  \} / \text{equivalence}.$$
The basic idea of the correspondence is that given a pair $(F,\sfS)$, there is a lift of $\sfS$ to a maximal $K$-split $k$-torus $\bS$ in $\bG$ and any two lifts of $\sfS$ are conjugate by an element of $G_{F,0^+}^{\Fr}$.  In the other direction, given a maximal $K$-split $k$-torus $\bS$ in $\bG$, the building of $S$ embeds into that of $G$.  We let $F$ be a maximal $G^{\Fr}$-facet in the building of $S$ and we let $\sfS$ be a maximal $\ff$-torus in $\sfG$ whose group of $\ffc$-points coincides with the image of $S \cap G_F$ in $\sfG_F$.  

\begin{example}  \label{ex:initialsp4}
For $\bG = \Sp_4$, there are nine $G^{\Fr}$-conjugacy classes of maximal unramified tori of $\bG$.   Since $G^{\Fr}$ acts transitively on the alcoves in $\BB(G)^{\Fr}$, to see how the above correspondence works it is enough to restrict our attention to a single alcove.   For a finite group of Lie type, the conjugacy classes of maximal tori in the finite group are in bijective correspondence with the $\Fr$-conjugacy classes in the Weyl group (see Remark~\ref{rem:ffversion}).  Here is what the notation in Figure~\ref{fig:Sp4pairs} represents: for a conjugacy class of $\ff$-minisotropic torus in $\sfG_F$ we list one element in the corresponding conjugacy class in the Weyl group of $\sfG_F$; 
 we have chosen a set of simple roots $\alpha$ and $\beta$ with $\alpha$ short and $\beta$ long so that in the diagram below the hypotenuse lies on a hyperplane defined by an affine root with gradient $\alpha$ and the horizontal edge lies on a hyperplane defined by an affine root with gradient $\beta$; the symbol $w_\alpha$ denotes the simple reflection in $W$ corresponding to $\alpha$, the notation $w_\beta$ denotes the reflection  corresponding to $\beta$, and $c$ denotes the Coxeter element $w_\alpha w_\beta$. 
In this way we  enumerate the nine pairs $(F,\sfS)$ that occur, up to equivalence.  
\begin{figure}[ht]
\centering
\begin{tikzpicture}
\draw (0,0) node[anchor=north]{\vphantom{$2^{2^2}$}\shortstack{$c$\\$c^2$} \hphantom{hello}}
  -- (5,0) node[anchor=north]{$\hphantom{hello} c^2$}
  -- (5,5) node[anchor=west]{\hphantom{h}\shortstack{$c$\\$c^2$} $ \vphantom{P_{P_{P_{P_{P}}}}}$}
  -- cycle;
  \draw (2.5,0) node[anchor=north]{$w_\beta$};
  \draw (1,2.5) node[anchor=west]{$w_\alpha$};
   \draw (5,2.5) node[anchor=west]{$w_{2\alpha +  \beta}$};
   \draw (3,1.5) node[anchor=west]{$1$};
\end{tikzpicture}
\caption{A labeling of the pairs $(F,\sfS)$ for $\Sp_4$ \label{fig:Sp4pairs}}
\end{figure}
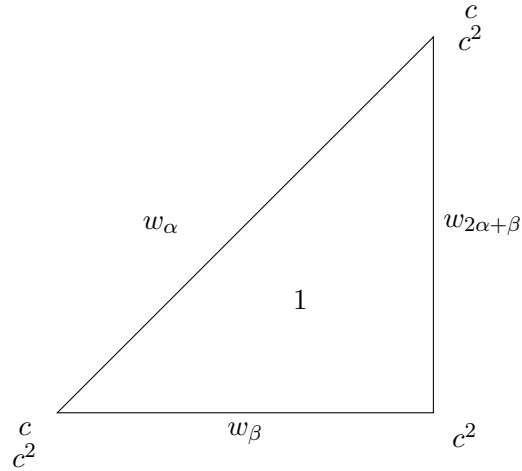
\end{example}

  In order to understand all unramified tori, not just the maximal ones, we should consider pairs $(F,\sfS)$ where $F$ is a facet in $\BB(G)^{\Fr}$ and $\sfS$ is an $\ff$-torus in $\sfG_F$ that lifts to an unramified torus in $\bG$.  The problem with this approach is that $F$ cannot ``see'' which $\ff$-tori $\sfS$ are relevant. 
We present an outline of how to overcome this problem, and then spend the remainder of this section fleshing out this outline.

\subsection{Outline}  \label{sec:overview}
Suppose that $(F,\sfS)$ is a pair with $F$ a $G^{\Fr}$-facet in $\AA(A)^{\Fr}$ and $\sfS$ an $\ff$-torus in $\sfG_F$ that contains $\sfZ_F$.  Let $\sfS'$ be a maximal $\ff$-torus in $\sfG_F$ that contains $\sfS$.   We can lift $\sfS'$ to a maximal unramified torus $\bS'$ in $\bG$, and we let $\bS$ be the subtorus of $\bS'$ corresponding to $\sfS$.  We can choose $g \in G_F$ so that $\bS' = \lsup{g} \bA$.   Since $\bS'$ is defined over $k$, we have that  $\Fr(g)^{-1} g$ belongs to the normalizer of $\bA$ in $G$, and we let $w \in W_F \leq W$ denote its image in the Weyl group.  If $\bS$ is going to be unramified, that is, if $\bS$ is going to be the $K$-split component of the center of $C_\bG(\bS)$, then one checks that there exists $\theta \in \Theta$ such that $\bS = {\lsup{g} \bA_\theta}$ and $\Fr(\Phi_\theta) = w \Phi_\theta$.
Thus, our pair $(F,\sfS)$ corresponds to a triple $(F,\theta,w)$ where $F$ is a facet in $\AA(A)^{\Fr}$, $\theta \in \Theta$, $w \in W_F$, and $\Fr(\Phi_\theta) = w \Phi_{\theta}$.

This approach leads to  a bijective correspondence between the set of $G^{\Fr}$-conjugacy classes of unramified tori in $\bG$ and the set of equivalence
classes of elliptic
triples $(F,\theta,w)$:
$$\{ \text{unramified tori} \}/ \text{$G^{\Fr}$-conjugacy} \longleftrightarrow \{ \text{elliptic $(F,\theta,w)$} \} / \text{equivalence}.$$
When we restrict this correspondence to maximal unramified tori in $\bG$, the triples under consideration are of the form $(F,\emptyset,w)$.

\subsection{An indexing set over \texorpdfstring{$\ff$}{f}} \label{sec:anindexingset}
Set
$$\dot{I} = \{ (\theta,w) \, | \, \theta \in \Theta(\bG,\bA) \text{, } w \in W \text{, and } \Fr(\Phi_\theta) = w \Phi_\theta \}.$$

Suppose $F \subset \AA(A)^{\Fr}$ is a $G^\Fr$-facet.
Set
$$\Phi(F) = \{ \dot{\psi} \, | \, \text{$\psi \in \Psi$ and $\res_F \psi$ is constant} \} \quad\text{and}\quad \bA(F) := \bigl(\bigcap_{\alpha \in \Phi(F)} \ker (\alpha)   \bigr)^\circ.$$
Recall that $\lsub{F}\bM = C_{\bG}(\bA(F))$ is a  Levi $(k,k)$-subgroup  of  $\bG$ and the image of $A(F) \cap G_{F}$ in $\sfG_F$ is the group of $\ffc$-points of the $\ffc$-split component of the center of $\sfG_F \simeq \lsub{F}\sfM_F$.

\begin{rem}
If $F$ is a minimal facet in $\AA(A)^{\Fr}$, then $\Phi(F) = \Phi$ and $\lsub{F}\bM = \bG$.
\end{rem}

Recall that $\sfA_F$ denotes the maximally $\ff$-split maximal $\ff$-torus in $\sfG_F$ whose group of $\ffc$-points is equal to the image of $A \cap G_F$ in $\sfG_F$.   Via the natural bijective correspondence between the characters of $\bA$ and the characters of $\sfA_F$, we (canonically) identify  $\Phi(F)$ with a subset of the character lattice of $\sfA_F$.  Note that $\Phi(F)$ will, in general, be strictly larger than the set of roots of $\sfG_F$ with respect to $\sfA_F$.

Let 
$$\dot{I}(F) := \{(\theta, w)  \in\dot{I} \, | \,  w \in W_F \leq W \}.$$
For $(\theta',w'), (\theta,w) \in \dot{I}(F)$ we write $(\theta',w') \stackrel{F}{\sim} (\theta,w)$
provided that there exists $n \in W_F$ for which
\begin{itemize}
\item  $\Phi_{\theta '} =  n  \Phi_{\theta}$ 
and
\item $ \Fr(n) w {n} ^{-1}  \in w' (W_F \cap W_{\theta'})$.
\end{itemize}

\begin{lemma}
The relation $\stackrel{F}{\sim}$ is an equivalence relation on $\dot{I}(F)$. \qed
\end{lemma}

We will say that $(\theta, w) \in \dot{I}(F)$ is \emph{$F$-elliptic} provided that  for all $(\theta',w') \in \dot{I}(F)$ with $(\theta, w) \stackrel{F}{\sim} (\theta',w')$ we have that $w'$ does not belong to a $\Fr$-stable proper parabolic subgroup of $W_F$.
We set
$$\dot{I}^{\ee}(F) := \{(\theta, w)  \in \dot{I}(F) \, | \,  \text{$(\theta,w)$ is $F$-elliptic} \}.$$

\begin{rem}  Suppose $(\theta,w)\in \dot{I}(F)$. If
$(\theta,w)\in \dot{I}^{\ee}(F)$, then $(\emptyset,w)\in \dot{I}^{\ee}(F)$.   The converse is false.  Consider for example a non-maximal $F$ and $(\emptyset,w) \in \dot{I}^{\ee}(F)$; the pair $(\Delta,w) \in \dot{I}(F)$ is not elliptic.
\end{rem}

 \begin{lemma}\label{lem:gbackwards}  Suppose $(\theta, w) \in \dot{I}(F)$.  We can choose $g \in G_F$ such that the image of  $n= \Fr(g)^{-1} g \in N_{G_F}(\bA)$ in $W_F$ is $w$.
\end{lemma}

\begin{proof}
Choose $\bar{h} \in \sfG_F$ such that the image of $\Fr(\bar{h})\inv \bar{h}$ in $W_F$ is $w$.  Note that $\sfS = \lsup{\bar{h}} \sfA_F$ is a maximal $\ff$-torus in $\sfG_F$.   Let $\bS$ be a lift of $(F,\sfS)$.  Since $\bS$ is a maximal $K$-split $k$-torus of $\bG$ and $F \subset \AA(S)$, there exists $x \in G_F$ such that $\lsup{x}A = S$.  Let $\bar{x}$ denote the image of $x$ in $\sfG_F$.  Since $\sfS = \lsup{\bar{x}} \sfA_F$, from Lemma~\ref{lem:ffversion} and Remark~\ref{rem:ffversion} the image of $\Fr(\bar{x})\inv \bar{x}$ in $W_F$ is of the form  $\Fr(w')\inv w {w'}$ for some $w' \in W_F$. Let $n' \in N_{G_F}(\bA)$ be a lift of $(w')\inv$.  Set $g = xn'$. 
\end{proof}

\begin{rem}
In Lemma~\ref{lem:gbackwards} we can choose $g$ in $(\lsub{F}M)_F$.
\end{rem}

\subsection{Relevant tori over \texorpdfstring{$\ff$}{f}}

Recall the set $J$ defined in Section~\ref{sec:someindexing}.
Also recall that Corollary~\ref{cor:heart} shows that if $\bS, \bS'$ both lift $(F,\sfS)$, then there exists an element $g \in G_{F,0^+}^{\Gamma}$ such that $\lsup{g}\bS = \bS'$.

\begin{defn}  Suppose
$(F,\sfS) \in J^{\Fr}$.  Let $\bS$ be a lift of $(F,\sfS)$.  We will say that $\sfS$ is \emph{relevant in $\sfG_F$}
provided that $\bS$ is the $K$-split component of the center of $C_{\bG}(\bS)$.    Let $\RR(F)$ denote the set of relevant tori in $\sfG_F$.
\end{defn}

Fix $\iota = (\theta, w) \in \dot{I}(F)$.  
Thanks to Lemma~\ref{lem:gbackwards}, we may choose $g \in G_F$ such that the image of  $n= \Fr(g)^{-1} g \in N_{G_F}(\bA)$ in $W_F$ is $w$.  Let $\bar{g}$ denote the image of $g$ in $\sfG_F$.  
 Let 
$$\sfA_\theta = ( \bigcap_{\alpha \in \theta} \ker(\alpha) )^\circ \leq \sfA_F.$$  
   Let $\sfS_{\iota} = {\lsup{\bar{g}}\sfA_{\theta}}$ and $\bS_{\iota} = {\lsup{{g}}\bA_{\theta}}$.
Then $S_{\iota}$ is a lift of $(F,\sfS_{\iota})$.   Set $\sfL_\iota = C_{\sfG_F}(\sfS_{\iota})$ and $\bL_\iota = C_G(\bS_\iota)$.   Note that $ F \subset \BB(L_\iota)$ and
$$\Phi_\theta = \lsup{{g^{-1}}}\Phi(\bL_\iota, \lsup{g}\bA).$$
\begin{rem}
In general, it is not true that $\Phi_\theta = \lsup{\bar{g}^{-1}}\Phi(\sfL_\iota, \lsup{\bar{g}}\sfA_F)$.
\end{rem}  
Since $\bS_\iota$ is the $K$-split component of the center of $\bL_\iota$, $\sfS_\iota$ is relevant.

\begin{rem}
Here is a way to think about the information that $\iota$ carries:
The subset $\theta$ (or $\Phi_\theta$) determines, up to isogeny, the derived group of $\bL_{\iota}$, the action of $w\inv \circ \Fr$ on $\Phi_\theta$ determines the $k$-structure of the derived group of $\bL_{\iota}$,  and  the action of $w\inv \circ \Fr$ on $\Phi$ determines the $k$-structure of the $K$-split component of the center of $\bL_{\iota}$, i.e. how $k$-anisotropic the center of $\bL_\iota$ is.
\end{rem}

\begin{rem}
    Recall that $\iota = (\theta,w)$.  Note that  $\bL_\iota = C_{\bG}(\bS_\iota)$ is a torus if and only if $\theta = \emptyset$.
\end{rem}

\begin{lemma} \label{lem:well-definedtoN}
The map that sends $\iota \in \dot{I}(F)$ to the $\sfG^{\Fr}_F$-conjugacy class of $\sfS_\iota$ is well defined.   Similarly, the map that sends $\iota \in \dot{I}(F)$ to the  $G^{\Fr}_F$-conjugacy class of $\bS_\iota$ is well defined. 
\end{lemma}

\begin{proof}
We need to show that the $\sfG^{\Fr}_F$-conjugacy class of $\sfS_\iota$ and the $G_F^{\Fr}$-conjugacy class of $\bS_\iota$ are independent of the choice of $g$.
Suppose $g' \in G_F$ such that the image of  $\Fr(g')^{-1} g' \in N_{G_F}(\bA)$ in $W_F$ is also $w$  and let $\bar{g}'$ denote the image of $g'$ in $\sfG_F$.    Let $\sfS'_{\iota} = {\lsup{\bar{g}'}\sfA_{\theta}}$ and $\bS'_{\iota} = {\lsup{{g'}}\bA_{\theta}}$.  Then $\bS'_{\iota}$ is a lift of $(F,\sfS'_{\iota})$.  
Since  $\Fr(g')^{-1} g'$ and  $\Fr(g)^{-1} g$ have image $w$ in $W_F$, there exists $a' \in C_{G}(A) \cap G_F$ such that 
$\Fr(g')^{-1} g' a' = \Fr(g)^{-1} g $.   Let $x = g' g^{-1} \in G_F$.  For all $s \in S_\iota$ we have
$$\Fr(\lsup{x}s) = \lsup{\Fr(g') \Fr(g^{-1})} \Fr(s) = \lsup{g'g^{-1}} (\lsup{\lsup{g}{a'}} \Fr(s)) = \lsup{x} \Fr(s).$$
Hence $\Int(x)$ and $\Int(\Fr(x))$ both carry $S_{\iota}^{\Fr}$ to ${S'_\iota}^{\Fr}$, hence they carry $\bS_{\iota}$ to $\bS'_{\iota}$.  Moreover, $\Fr(x)\inv x \in C_{G_F}(\bS_\iota) = (L_\iota)_F$.  
From Lemma~\ref{lem:parahoriccohom}, we have $\cohom^1(\Fr, (L_\iota)_F) = 1$.  Thus there exists $\ell \in C_{G_F}(\bS_{\iota})$ such that $\Fr(x)^{-1} x = \Fr(\ell)^{-1} \ell$ modulo $G_F^+$.
Thus $\bar{y}$, the image of $x \ell^{-1} \in G_F^{\Fr}$ in $\sfG_F$, belongs to $\sfG_F^{\Fr}$.  Note that  $\sfS_\iota$ and $\sfS_\iota'$ are $\sfG_F^\Fr$-conjugate by $\bar{y}$ while $\bS_\iota$ and $\bS_\iota'$ are $G_F^\Fr$-conjugate by ${x \ell^{-1}}$.
\end{proof}

\begin{lemma}  \label{lem:lem645}
The map that sends $\iota \in \dot{I}(F)$ to the $\sfG^{\Fr}_F$-conjugacy class of $\sfS_\iota$ descends to a bijective map from $ \dot{I}(F)/\stackrel{F}{\sim}$ to the set of $\sfG^{\Fr}_F$-conjugacy classes in $\RR(F)$.
\end{lemma}

\begin{proof}
We first show that the map is injective.
Suppose $\iota_i = (\theta_i,w_i) \in \dot{I}(F)$ and $g_i \in  G_F$ 
such that the image of $\Fr(g_i)^{-1} g_i$ in $W_F$ is $w_i$. Set $\bS_i = {\lsup{g_i}\bA_{\theta_i}}$ and $\sfS_i =  {\lsup{\bar{g}_i}}\sfA_{\theta_i}$ where $\bar{g}_i$ is the image of $g_i$ in $\sfG_F$.
Note that $\bS_i$ is a lift of $(F,\sfS_i)$.  Suppose there exists  $\bar{h} \in \sfG_F^{\Fr}$ such that  $\sfS_{1} = {\lsup{\bar{h}} \sfS_{2}}$. Then from Corollary~\ref{cor:heart}, there exists a lift $h \in G^{\Fr}_F$ of $\bar{h}$ for which $\bS_1 = \lsup{h}\bS_2$.   Without loss of generality, replace $g_2$ by $hg_2$, so that $\bS_1 = \bS_2$ and $\sfS_1 = \sfS_2$.  
Let $\sfL_1 = C_{\sfG_F}(\sfS_1)$ and let $\bL_1 = C_{\bG}(\bS_1)$.
There exists $\bar{\ell} \in \sfL_1$ for which $\lsup{\bar{g}_2}\sfA_F = {\lsup{\bar{\ell}\bar{ g}_1} \sfA_F}$. 
Thanks to Corollary~\ref{cor:heart} there is a lift $\ell \in L_1 \cap G_F$  of $\bar{\ell}$ for which $\lsup{g_2}\bA = {\lsup{\ell g_1} \bA}$.
Choose $m \in N_{G}(\bA)$ for which $\ell g_1 = g_2 m$.   Note that $m = g_2\inv \ell g_1 \in G_F$.  Let $\bM_{\theta_i} = C_\bG(\bA_{\theta_i})$.  We have
\[
\begin{split}
\Phi_{\theta_1} &= \Phi(\bM_{\theta_1}, \bA) \\
 & = g_1\inv \Phi(\bL_1, {\lsup{g_1}}\bA) =    g_1\inv  \ell\inv \Phi(\bL_1 , {\lsup{\ell g_1}}\bA) =    g_1\inv  \ell\inv \Phi(\bL_1, {\lsup{g_2}}\bA)  \\
  & =   g_1\inv  \ell\inv g_2 \Phi(\bM_{\theta_2},  \bA) =  m\inv \Phi(\bM_{\theta_2} , \bA)  =  \Phi(\bM_{m\inv \theta_2},\bA)   \\
  &= m\inv \Phi_{\theta_2},
\end{split}
\]
i.e. $m \Phi_{\theta_1} = \Phi_{\theta_2}$.

 Since the image of $\lsup{g_1^{-1}}(\Fr(\ell)^{-1} \ell ) \in N_{G}(\bA)$ in $W$ belongs to $W_{\theta_1}$, it follows that  
 $$\Fr(m)^{-1} w_2 m = \Fr({\ell} {g}_1)^{-1} {\ell} {g}_1 A = w_1 g_1^{-1} (\Fr (\ell)^{-1} \ell) g_1 A \in w_1 W_{\theta_1}.$$
 Since $\Fr(m)^{-1} w_2 m = \Fr(\bar{\ell} \bar{g}_1)^{-1} \bar{\ell} \bar{g}_1 \sfA_F \in W_F$,
 we conclude that $\iota_1 \stackrel{F}{\sim} \iota_2$.

 We now show that the map is surjective.  Suppose $\sfT \leq \sfG_F$ belongs to $\RR(F)$.
 Let $\sfT'$ be a maximal $\ff$-torus in $\sfG_F$ that contains $\sfT$, and that has the largest possible $\ff$-split rank among tori in $\sfG_F$ that contain $\sfT$.
 By definition $\sfT$ contains the center of $\sfG_F$. There exist  lifts $\bT$ of $(F,\sfT)$ and $\bT'$ of $(F,\sfT')$ such that $\bL = C_\bG(\bT)$ is a Levi $(K,k)$-subgroup, $\bT$ is the $K$-split component of the center of $\bL$, and $\bT \leq \bT' \leq \bL$.  
 Let $\bB_\bL \leq \bL$ be a Borel $K$-subgroup of $\bL$ that contains $\bT'$. 
 Since $\bT'$ is a lift of $(F,\sfT')$, there is a  $g \in G_F$ such that $\lsup{g}\bA = \bT'$.   Let  ${\theta}  = g ^{-1} \Delta(\bL, \bB_\bL, \bT') \in \Theta$.   Let $w$ denote the image of $\Fr(g)^{-1} g$ in $W_F$.   The pair $(\theta,w)$ belongs to $\dot{I}(F)$ and corresponds to $\sfT$.
\end{proof}

\begin{lemma}  \label{lem:anisotropicasexpected}
If $\sfS \in \RR(F)$ corresponds to $(\theta, w) \in \dot{I}^{\ee}(F)$, then $ C_{\sfG_F}(\sfS)$ is an $\ff$-minisotropic  maximal torus in $\sfG_F$ that corresponds to $(\emptyset, w) \in \dot{I}^{\ee}(F)$.
\end{lemma}
 
\begin{proof}
Suppose $\sfS \in \RR(F)$ corresponds to $(\theta,w) \in \dot{I}^{\ee}(F)$.  Choose $g \in \sfG_F$ so that the image of $\Fr(g)^{-1} g$ in $W_F$ is $w$. Without loss of generality, assume $\sfS = {\lsup{g}\sfA_{\theta}}$.  Since $(\theta,w) \in \dot{I}^{\ee}(F)$, no $\Fr$-conjugate of $w$ in $W_F$ can belong to a proper $\Fr$-stable parabolic subgroup of $W_F$, so $\sfS' := {\lsup{g}\sfA_F}$ is an $\ff$-minisotropic maximal torus in $\sfG_F$.
 
Let $\sfM$ denote the centralizer of $\sfS$ in $\sfG_F$.  Note that $\sfM$ is a Levi $(\ffc,\ff)$-subgroup of $\sfG_F$, it contains $\sfS'$,  and it is $\ff$-quasi-split.  If $\sfM = \sfS'$, then $\sfS'$ corresponds to $(\emptyset,w)$ and we are done.  So, it will be enough to show that $\sfS' = \sfM$.

Suppose $\sfS' \neq \sfM$.  Let $\sfT_\sfM$ be a maximally $\ff$-split maximal $\ff$-torus in $\sfM$ and recall that $\sfT^\ff_{\sfM}$ denotes the unique maximal $\ff$-split torus in $\sfT_\sfM$.  Since $\sfM$ is $\ff$-quasi-split and is not a torus,  we have $\sfT_\sfM = C_{\sfM}(\sfT_{\sfM}^\ff) \lneq \sfM$. Hence $\sfL := C_{\sfG_F}(\sfT_\sfM^\ff)$  cannot be $\sfG_F$, so $\sfL$
is a proper Levi $(\ff,\ff)$-subgroup of $\sfG_F$.  Note that $\sfL$ contains $\sfS$ and $\sfT_\sfM$.   After replacing $\sfS$ with a $\sfG_F^\Fr$-conjugate, we may also assume that $\sfT_{\sfM}^\ff \leq \sfA_F$ and so  $\sfA_F \leq \sfL$.

Choose $h \in \sfM$ so that $\lsup{h^{-1} g}\sfA_F = \sfT_\sfM$. Note that $\sfS = {\lsup{h^{-1} g} \sfA_{\theta}}$ and $w'$, the image of $\Fr(h^{-1} g)^{-1} h^{-1} g$ in $W_F$, has image in $w(W_\theta \cap W_F)$.
That is, $(\theta,w')$ and $(\theta,w)$ are equivalent in $\dot{I}(F)$.

Choose $\ell \in \sfL$ so that $\lsup{\ell}\sfA_F = \sfT_\sfM$. Let $w''$ denote the image of $\Fr(\ell)^{-1} \ell$ in $W_F$.   By construction $w''$ belongs to a $\Fr$-stable proper parabolic subgroup of $W_F$.  Moreover, since $\sfT_\sfM = {\lsup{\ell} \sfA_F} = {\lsup{h\inv g} \sfA_F}$ there exists an $n \in W_F$ for which $\Fr(n)  w'' n^{-1} = w'$.   Note that $(n^{-1} \theta, w'') \stackrel{F}{\sim} (\theta,w') \stackrel{F}{\sim} (\theta,w)$, contradicting that $(\theta,w) \in \dot{I}^{\ee}(F)$.
 \end{proof}
 
 \begin{cor}   \label{cor:improved541}
 We maintain the notation of Lemma~\ref{lem:anisotropicasexpected}.  If $\bT$ is a lift of $(F,C_{\sfG_F}(\sfS))$, then $C_{\bG}(\bT^k) \in (\lsub{F}\bM)$.
 \end{cor}
 
 \begin{proof}
 Since $\sfT = C_{\sfG_F}(\sfS)$ is an $\ff$-minisotropic maximal $\ff$-torus in $\sfG_F$, we have that $\sfC = \sfT^\ff$ is the $\ff$-split component of the center of $\sfG_F$.  Since $\bT^k \bZ^K$ is a lift of $(F, \sfC \sfZ_F)$, from Remark~\ref{rem:handsonM} we conclude that $C_\bG(\bT^k\bZ^K) \in (\lsub{F}\bM)$.   Since $C_\bG(\bT^k\bZ^K) = C_\bG(\bT^k)$, the result follows.
 \end{proof}

\begin{cor}
We maintain the notation of Lemma~\ref{lem:anisotropicasexpected}.   If $\bS$ is a lift of $(F,\sfS)$, then the dimension of a maximal $k$-split torus in $\bL = C_\bG(\bS)$ is equal to the $\mathbb{R}$-dimension of $A(\AA,F)^{\Fr}$.   
\end{cor}

\begin{proof} 
Let $d_L$ denote the dimension of a maximal $k$-split torus in $\bL$ and let $d_F$ denote 
the dimension of $A(\AA,F)^{\Fr}$.  We need to show $d_L = d_F$.

Suppose $\bA_L \leq \bL$ is a maximally $k$-split maximal $K$-split $k$-torus in $\bL$.  The dimension of $\bA^k_L$ is equal to $d_L$.  Since all maximally $k$-split maximal $K$-split $k$-tori in $\bL$ are $L^\Fr$-conjugate, we may assume $F$ is a subset of the apartment of  $A_L$.  Let $\sfA_L$ denote the $\ff$-torus whose group of $\ffc$-points coincides with the image of $A_L \cap G_{F,0}$ in $\sfG_F$.  We have that $d_L$ is equal to the $\ff$-dimension of $\sfA^{\ff}_L$. Since $\bA_L \leq \bL$, we have $\sfA_L \leq C_{\sfG_F}(\sfS)$.   From Lemma~\ref{lem:anisotropicasexpected} we know that $C_{\sfG_F}(\sfS)$ is an $\ff$-minisotropic maximal $\ff$-torus in $\sfG_F$.    Since $C_{\sfG_F}(\sfS)$ is $\ff$-minisotropic in $\sfG_F$, we conclude that $\sfA^\ff_L$ is the maximal $\ff$-split torus in the center of $\sfG_F$.   But the dimension of the latter is $d_F$.   Thus, $d_L = d_F$.
\end{proof}

\subsection{Parameterizing \texorpdfstring{$G^{\Fr}$-conjugacy classes of unramified tori in $\bG$}{Fr fixed conjugacy classes of unramified tori in G}}
Define
$$\tilde{I} = \{ (F, \theta, w) \colon \text{$F \subset \AA(A)^{\Fr}$ is a $G^\Fr$-facet and $(\theta, w) \in \dot{I}(F)$} \}$$
and let $\UU$ denote the set of $G^{\Fr}$-conjugacy classes of unramified tori in $\bG$.  Thanks to Lemma~\ref{lem:well-definedtoN} and Corollary~\ref{cor:heart} we can define a function $j \colon \tilde{I} \rightarrow \UU$ as follows.   For  $(F,\theta,w) \in \tilde{I}$, let $\sfS \in \RR(F)$ be a relevant torus associated to $(\theta,w)$ and let $j((F,\theta,w))$ be the $G^{\Fr}$-conjugacy class of any lift of $(F,\sfS)$.

 For $(F', \theta',w'), (F, \theta,w) \in \tilde{I}$ we write $(F', \theta',w') \approx (F,\theta,w)$ provided that there exists an element $n \in \affFrW$ for which $A(\AA(A)^\Fr, F') = A(\AA(A)^{\Fr},nF)$  and with the identifications of  $\sfG_{F'} \overset{i}{=} \sfG_{nF}$  and $\X^*(\sfA_{F'}) \overset{i}{=} \X^*(\sfA_{nF}) = \X^*(\bA)$ thus induced we have that $(\theta',w') \stackrel{F'}{\sim} (n\theta, \lsup{n}w)$ in $\dot{I}(F') \overset{i}{=} \dot{I}(nF)$.

\begin{lemma}
The relation $\approx$ is an equivalence relation on $\tilde{I}$. \qed
\end{lemma}

\begin{defn}
We will say that $(F, \theta, w) \in \tilde{I}$ is \emph{elliptic} provided that $(\theta, w) \in \dot{I}^{\ee}(F)$, that is, for all $(\theta',w') \in \dot{I}(F)$ with $(\theta, w) \stackrel{F}{\sim} (\theta',w')$ we have that $w'$ does not belong to a $\Fr$-stable proper parabolic subgroup of $W_F$.
We set
$$\tilde{I}^{\ee} := \{(F, \theta, w)  \in \tilde{I} \, | \,  \text{$(F, \theta, w)$ is elliptic} \}.$$
\end{defn}

\begin{rem}
Suppose $(F_i,\theta_i,w_i) \in \tilde{I}$ for $i \in \{1,2\}$ with $(F_1,\theta_1,w_1) \approx (F_2,\theta_2,w_2)$.   Then $(F_1,\theta_1,w_1) \in \tilde{I}^{\ee}$ if and only if $(F_2,\theta_2,w_2) \in \tilde{I}^{\ee}$.
\end{rem}

\begin{theorem}  \label{thm:maintheorem}
The map $j$ induces a  bijection from ${\tilde{I}^{\ee}/\!\approx}$ to $\UU$.
\end{theorem}

\begin{proof}
We first show that $j$ is surjective.   Let $\bS$ be an unramified torus in $\bG$.  Let  $\bL$ denote the centralizer of $\bS$ and let $F$ be a maximal $G^{\Fr}$-facet in $\BB(L)^{\Fr} \subset \BB(G)^{\Fr}$.  Choose a   maximal $K$-split $k$-torus  $\bS'$  of $\bL$ such that  $F \subset \BB(\bS')^{\Fr}$.  Note that   $\bS'$ contains $\bS$.  Fix a Borel $K$-subgroup  $\bB_\bL$ of $\bL$ that contains $\bS'$.    Choose $h \in G^{\Fr}$ so that $hF \subset \AA(A)^{\Fr}$.  After replacing $\bS$ with $\lsup{h}\bS$, we may assume that $F \subset \AA(A)^{\Fr}$.  

Let $\sfS$ denote the $\ff$-torus in $\sfG_F$ whose group of $\ff$-rational points coincides with the image of $S \cap G_F$ in $\sfG_F$.   There exists $g \in G_F$ such that $\bS' = {\lsup{g}\bA}$.  Let $\theta = g^{-1} \Delta(\bL,\bB_\bL,\bS') \in \Theta$ and let $w$ denote the image of $\Fr(g)^{-1} g$ in $W_F$.    Note that $\bS$ belongs to $j((F,\theta, w))$.   

To complete the proof of surjectivity, we need to show that $(F,\theta, w)$ is elliptic.  If it is not elliptic, then there exist $(\theta',w') \in \dot{I}(F)$ with $(\theta,w) \stackrel{F}{\sim} (\theta', w')$ and a $G^{\Fr}$-facet $H \subset \AA(A)^{\Fr}$ with $F \subset \bar{H}$ so that $w'$ lies in $W_H$.   Since $w' \in W_H$, there exists $h \in G_{H} \subset G_F$ such that  the image of $\Fr(h)^{-1} h$ in $W_{H} \leq W_F$ is $w'$.   
Since  $(\theta,w) \stackrel{F}{\sim} (\theta', w')$, from Lemma~\ref{lem:lem645} we have 
$\lsup{h}\bA_{\theta'} = \lsup{x}\bS$ for some $x \in G_F^{\Fr}$.  Set $k = x\inv h \in  G_F$.  Note that $\lsup{k}\bA \leq \bL$.   Hence,  $k H \subset \AA(\lsup{k}\bA)^\Fr \leq \BB(L)^{\Fr}$, contradicting the maximality of $F$.

We now show that if $(F_i,\theta_i, w_i)$ for $i \in \{1,2\}$ are two elements of  $\tilde{I}^{\ee}$ with $j((F_1 ,\theta_1 , w_1)) = j((F_2, \theta_2 , w_2))$, then $(F_1, \theta_1, w_1)  \approx (F_2, \theta_2, w_2).$    Choose $\sfS_i \in \RR(F_i)$ corresponding to $(\theta_i, w_i) \in \dot{I}^{\ee}(F_i)$ and let $\bS_i$ be a lift of $(F_i,\sfS_i)$.   Note that $(F_i, \sfS_i) \in \JGammamax$.  Since $j((F_1 , \theta_1 , w_1)) = j((F_2 , \theta_2 , w_2))$, we have that $\bS_1$ is $G^{\Fr}$-conjugate to  $\bS_2$.  Thanks to Corollary~\ref{cor:toogeneralbij}, we know that there exist $g \in G^{\Fr}$ and an apartment $\AA'$ in $\BB(G)^{\Fr}$ such that  $ \emptyset \neq A(\AA',F_1) = A(\AA',gF_2)$
and $\sfS_1 \overset{i}{=} {\lsup{g}\sfS_2}$ in $\sfG_{F_1} \overset{i}{=} \sfG_{gF_2}$.   After conjugating everything in sight by an element of $G_{F_1}^{\Fr}$ we may assume that $\AA' = \AA(A)^{\Fr}$.  Thanks to the affine Bruhat decomposition, we may choose $n \in N_{G^{\Fr}}(\bA)$ so that $n^{-1} g \in G^{\Fr}_{F_2}$.   Then there exists $x \in \sfG^{\Fr}_{F_2}$ such that after replacing, as we may, $\sfS_2$ by $\lsup{x}\sfS_2$ we may assume $ A(\AA(A)^\Fr, F_1 ) = A( \AA(A)^\Fr , nF_2 )$
and $\sfS_1 \overset{i}{=} {\lsup{n}\sfS_2}$ in $\sfG_{F_1} \overset{i}{=} \sfG_{nF_2}$.  Identifying $n$ with its image in $\affFrW$, the fact that $\sfS_1 \overset{i}{=} {\lsup{n}\sfS_2}$ in $\sfG_{F_1} \overset{i}{=} \sfG_{nF_2}$ means
$(\theta_1,w_1) \stackrel{F}{\sim} (n\theta_2, \lsup{n}w_2)$ in $\dot{I}(F_1)\overset{i}{=} \dot{I}(nF_2)$.
\end{proof}

\begin{cor}
There is a natural bijection between ${\tilde{I}^{\ee}/\!\approx}$ and the set of $G^\Fr$-conjugacy classes of unramified Levi subgroups in $\bG$.
\end{cor}

\begin{proof}
This follows from Theorem~\ref{thm:maintheorem}  and Remark~\ref{rem:421}
\end{proof}

\subsection{Example: Rational classes of Levi \texorpdfstring{$(k,k)$-subgroups}{(k,k)-subgroups}} \label{sec:ratlevis}
    Suppose $\bM$ is a Levi $(k,k)$-subgroup of $\bG$, that is, $\bM$ is the Levi component of a parabolic $k$-subgroup of $\bG$.  After conjugating $\bM$ by an element of $G^{\Fr}$,  we may and do assume that $\bA \leq \bM$.   Choose a basis $\theta \subset \Theta$ for $\Phi(\bM,\bA)$.   Since $\bM$ is a Levi $(k,k)$-subgroup and $\Phi_{\theta} = \Phi(\bM,\bA)$, we have $\Fr(\Phi_{\theta}) = \Phi_{\theta}$.   So, for any facet $F \subset \AA(A)^{\Fr}$ we have $(\theta,1) \in \dot{I}(F)$.  Thus $(\theta,1) \in \dot{I}^{\ee}(F)$ if and only if $F$ is an alcove in $\AA(A)^{\Fr}$.  By construction we have $\bM = C_\bG(\bA_\theta)$ and  $\bA_\theta \in j((C,\theta,1))$ where $C$ is any alcove in $\AA(A)^{\Fr}.$

Here is a systematic method, inspired by~\cite[pages 4 and 5]{solleveld:conjugacy}, for identifying, up to equivalence, the possible $\theta$ that can arise in a triple $(C,\theta,1)$ such that $j((C,\theta,1))$ parameterizes a rational conjugacy class of Levi $(k,k)$-subgroups.    Suppose $\bG^*$ is a $k$-quasi-split inner form of $\bG$ with Frobenius acting  on $G^*$ by $\Fr^*$.   Choose a Borel $k$-subgroup $\bB^*$ in $\bG^*$ and a maximally $k$-split maximal $K$-split $k$-torus $\bA^* \leq \bB^*$.  Let $\Delta^* = \Delta(\bG^*,\bB^*,\bA^*)$.   Fix a $\Fr^*$-stable alcove $D$ in $\AA(A^*) \subset \BB(G^*)$ such that every element of $\Delta^*$ occurs as the gradient of some affine simple root of $\bG$  with respect to $\bA^*$, $K$, $\nu$, and $D$.  Let $D'$ denote the image  of $D$ in the reduced building of $G^*$.   Since $\bG$ is an inner form of $\bG^*$, there exists $\Ad(m) \in N_{G^*_{\ad}}(\bA^*) \cap \Stab_{G^*_{\ad}}(D')$ such that $\bG$ is $k$-isomorphic to $\bG^*$ twisted by $\Ad(m) \circ \Fr^*$ (see, for example,
~\cite[Remark~3.4.5]{debacker:totally}).
 That is, we may assume $\bG_K = \bG^*_K$,  $G = G^*$ as abstract groups, and $\Fr(g) = \Ad(n) \circ \Fr^*(g)$ for all $g \in G$.

Without loss of generality, $\bB = \bB^*_K$. Since $\Ad(m) \circ \Fr^* (\bA^*) = \bA^*$, we conclude that $\bA^*$ twisted by $\Ad(m) \circ \Fr^*$ is a maximal $K$-split $k$-torus in $\bG$.   Since $\Ad(m) \circ \Fr^*(D) = D$, we have that $\bA^*$ twisted by $\Ad(m) \circ \Fr^*$ is a maximally $k$-split maximal $K$-split $k$-torus in $\bG$, that is, we can, without loss of generality, take $\bA$ to be $\bA^*$ twisted by $\Ad(m) \circ \Fr^*$.   We therefore have $\Delta = \Delta^*$, and we can take $C = D^{\Fr} \subset \AA(A)^{\Fr}$\!.
Define $\theta_m \subset \Delta \in \Theta$ by 
$$\theta_m := \{ \alpha \in \Delta \, | \, \bA^k \subset \ker(\alpha)\}. $$
The set of $\theta_m$ may also be characterized as the set of gradients of affine roots $\psi  \in \Psi(\bG^*,\bA^*,K,\nu)$ for which $\dot{\psi} \in \Delta$ and $ \res_{C}\psi$ is constant.

As we now show,  the set of triples $(C, \theta, 1)$ with $\theta \subset \Delta$ such that  $\theta_m \subset \theta$ and ${\theta}$ is $\Fr^*$-stable contains  a complete set of representatives (possibly with duplicates) of the equivalence classes in $\tilde{I}^{\ee}$ that parameterize the $G^{\Fr}$-conjugacy classes of Levi $(k,k)$-subgroups of $\bG$.

Note that $\Phi_{\theta_m}$ is $\Fr$-stable, $\bA^k \leq \bA_{\theta_m}$, and $\bA^k = \bA_{\theta_m}^k$.   Thus, by Corollary~\ref{cor:3.1.2}  we have that
$\bL_{\theta_m} := C_{\bG}(\bA_{\theta_m}) =  C_{\bG}(\bA^k)$ is 
a minimal Levi $(k,k)$-subgroup  of $\bG$.
Choose $\lambda \in \X_*(\bA^k)$ such that $\langle \lambda, \beta \rangle >0$ for all $\beta \in \Delta \setminus \Delta_m$.   Then 
$\bP_{\theta_m} := \bP(\lambda)$, the parabolic $k$-subgroup of $\bG$ generated by $\bL_{\theta_m}$ and those root groups $\bU_\alpha$ for $\alpha \in \Phi(\bG,\bA)$ such that $\langle \lambda,\beta \rangle >0$, is a minimal parabolic $k$-subgroup in $\bG$ that contains $\bB$ and has Levi factor $\bL_{\theta_m}$.   Since $\bB^* = \Fr^*(\bB^*)$, we conclude that, as $K$-groups, we have  $\bB \leq \Fr^*(\bP_{\theta_m}) = \Ad(m\inv) (\bP_{\theta_m})$.
Thus $\bP_{\theta_m}$ and $\Ad(m\inv)(\bP_{\theta_m})$ are conjugate standard parabolic $K$-subgroups, hence equal.  Since parabolic subgroups are self normalizing, we conclude that  $\Ad(m) \in P_{\ad}(\lambda)$.  Similarly,  since
$\bA  \leq \Fr^*(\bL_{\theta_m} )= \Ad(m\inv) (\bL_{\theta_m}) \leq \Ad(m\inv)(\bP_{\theta_m}) =  \bP_{\theta_m}$ as $K$-groups, we conclude that $\Ad(m\inv) (\bL_{\theta_m}) = \bL_{\theta_m}$, and so $\Ad(m) \in (L_{\ad})_{\theta_m}$.

 We are interested in those  $\theta$ such that $(C,\theta,1)$ corresponds to a maximal $K$-split torus in the center of a Levi $(k,k)$-subgroup of $\bG$.  Such a Levi belongs to a parabolic $k$-subgroup of $\bG$.  After conjugating by an element of $G^{\Fr}$, we may assume $\bA$ is contained in this Levi  and the parabolic subgroup is a standard parabolic  subgroup. Thus, we can assume $\theta_m \subset \theta \subset \Delta$.  
 Note  that if $\theta' \subset \Delta$ with $\theta_m \subset \theta'$, then, since $\Ad(m)$ has image in $W_{\theta_m}$,   we have $\Fr(\Phi_{\theta'}) = \Phi_{\theta'}$ if and only if $\Fr^* (\Phi_{\theta'}) = \Phi_{\theta'}$. Since $\theta' \subset \Delta^*$, we have that   $\Fr^*(\Phi_{\theta'}) = \Phi_{\theta'}$ if and only if $\Fr^* ({\theta'}) = {\theta'}$. Thus, the $\theta$ we seek are those subsets $\theta$ of $\Delta$ such that $\theta_m \subset \theta$ and $\Fr^* ({\theta}) = {\theta}$.

\subsection{Some applications of Theorem~\ref{thm:maintheorem}}

\begin{cor} \label{cor:corto5.5.4}  Suppose $\theta \in \Theta$.   A $G$-conjugate of $\bA_\theta$ is defined over $k$ if and only if there exists a facet $F \subset \AA(A)^{\Fr}$ such that $\Fr(\Phi_\theta) = w( \Phi_{\theta})$ for some $w \in W_F$. \qed
\end{cor}

\begin{cor} \label{cor:cortocorto5.5.4}Suppose $\bG$ is $k$-quasi-split, $\bM$ is a Levi $(K,K)$-subgroup of $\bG$, and $\bS$ is the maximal $K$-split torus in the center of $\bM$.   If $\Fr(\lsup{G}S) = \lsup{G}S$, then there exists $h \in G$ such that $\Fr(\lsup{h}S) = \lsup{h}S$.  Moreover, we may assume that $C_{\bG}(\lsup{h}\bS)$ is a $k$-quasi-split unramified twisted Levi subgroup of $\bG$.
\end{cor}

\begin{rem} The converse to Corollary~\ref{cor:cortocorto5.5.4} is trivially true, even when $\bG$ is not $k$-quasi-split.
\end{rem}

\begin{proof}
Without loss of generality $\bS = \bA_{\theta}$ for some $\theta \in \Theta$.   Choose $h \in G$ such that $\Fr(A_\theta) = \lsup{h}A_\theta$.  Since $\Fr(A) = A$, we have $A_\theta, \lsup{h}A_\theta \subset A$.  Thus $\lsup{h}\bA$ and $\bA$ are maximal $K$-split tori in $C_{\bG}(\lsup{h}\bA_{\theta})$.  Consequently, there exists $\ell \in C_{G}(\lsup{h}\bA_{\theta})$ such that $\lsup{\ell h} \bA = \bA$.  That is, $\ell h \in N_G(\bA)$.  Let $w$ denote the image of $\ell h$ in $W$.   Note that $\Fr(A_\theta) = A_{w \theta}$, that is $\Fr(\Phi_\theta) = w(\Phi_\theta)$.  

Since $\bG$ is $k$-quasi-split, there exists an absolutely special vertex $x_0$  that belongs to the image of $\AA(A)$ in  $\BB^{\red}(G)^{\Fr}$.  Let $n \in N_{G_{x_0,0}}(\bA)$ be a lift of $w$.   From Corollary~\ref{cor:corto5.5.4} with $F$ being the preimage in $\BB(G)^{\Fr}$  of $x_0$, we conclude that a $G$-conjugate of $A_{\theta}$ is $\Fr$-fixed.

We now show that we may assume that the centralizer in $\bG$ of this $\Fr$-fixed  $G$-conjugate of $A_{\theta}$ is $k$-quasi-split.  Indeed, we can choose $h \in G_{x_0,0}$ such that the image in $W$ of $\Fr(h)\inv h \in N_{G_{x_0,0}}(\bA)$  is $w$.  Note that both $\lsup{h}A_\theta$ and $\lsup{h}A$ are $\Fr$-stable.   Let $\bar{h}$ denote the image of $h$ in $\sfG_{x_0}$. Since $x_0$ is absolutely special, the root system $\Phi(C_{\sfG_{x_0}}(\lsup{\bar{h}}\sfA_\theta), \lsup{\bar{h}}\sfA)$ has a $\Fr$-invariant basis if and only if the root system  $\Phi(C_{\bG}(\lsup{h}\bA_\theta), \lsup{h}\bA)$ has a $\Fr$-invariant basis.  Since $C_{\sfG_{x_0}}(\lsup{\bar{h}}\sfA_\theta)$ is $\ff$-quasi-split, we conclude that $C_{\bG}(\lsup{h}\bA_\theta)$ is $k$-quasi-split.
\end{proof}

\begin{defn}  We will say that $\gamma \in G$ is an \emph{unramified semisimple element} provided that $C_\bG(\gamma)$ is a Levi $(K,K)$-subgroup and $\gamma$ belongs to the group of $K$-points of the maximal $K$-split torus in the center of $C_\bG(\gamma)$.
\end{defn}

\begin{cor}  \label{cor:ratpointsunramss} Suppose $\bG$ is $k$-quasi-split and $\gamma \in G$ is an unramified semisimple element. If $\Fr(\lsup{G} \gamma) = \lsup{G}\gamma$ then there exists $h \in G$ such that $\Fr(\lsup{h}\gamma) = \lsup{h}\gamma$.  Moreover, we may assume that $C_{\bG}(\lsup{h}\gamma)$ is a $k$-quasi-split unramified twisted Levi subgroup of $\bG$.
\end{cor}

\begin{rem} The converse to Corollary~\ref{cor:ratpointsunramss} is trivially true, even when $\bG$ is not $k$-quasi-split.
\end{rem}

\begin{rem}
When the derived group of $\bG$ is simply-connected, Corollary~\ref{cor:ratpointsunramss} may be derived from~\cite[Theorem~4.1 and Lemma~3.3]{kottwitz:rational}.
\end{rem}

\begin{proof}  Suppose $\gamma \in G$ is an unramified semisimple element.   Let $\bS$ denote the maximal $K$-split torus in the center of $C_{\bG}(\gamma)$.  The assignment $\gamma \mapsto \bS$ defines a $G$-equivariant function from the set of unramified semisimple elements to the set of tori that arise as maximal $K$-split tori in the center of  Levi $(K,K)$-subgroups in $\bG$.   Moreover, we have $\Fr(\lsup{g}\gamma) \mapsto \lsup{\Fr(g)} \Fr(\bS)$ for all $g \in G$.   Consequently, since $\Fr(\lsup{G}\gamma) = \lsup{G}\gamma$, we conclude that $\Fr(\lsup{G}S) = \lsup{G}S$.   From Corollary~\ref{cor:cortocorto5.5.4} there exists $h \in G$ such that $\Fr(\lsup{h}S) = \lsup{h}S$. From Theorem~\ref{thm:maintheorem} there exists $(F,\theta,w) \in \tilde{I}^{\ee}$ such that $\lsup{h}S \in j((F,\theta,w))$.  Thus, from the construction of the map $j$, there exists $g \in G$ such that $\lsup{g} \bS = \bA_{\theta}$, the image of $\Fr(g)\inv g \in N_{G_{F,0}}(\bA)$ in $W$ is $w$, and  $\lsup{g}\gamma \in \bA_{\theta}$.

Since $\bG$ is $k$-quasi-split, there exists an absolutely special vertex $x_0$ that belongs to the image of $\AA(A)$ in  $\BB^{\red}(G)^{\Fr}$.  Choose $\ell \in G_{x_0,0}$ such that the image of $\Fr(\ell)\inv \ell \in N_{G_{x_0,0}}(\bA)$ in $W$ is $w$.  
Replacing $\gamma$ by $\lsup{\ell g}\gamma$ and $\bS$ by $\lsup{\ell} A_\theta$, we can assume $\gamma \in \lsup{\ell} A_\theta$.   Since $\Fr(\lsup{\ell} A_\theta) = \lsup{\ell} A_\theta$, we must also have $\Fr(\gamma) \in \lsup{\ell} A_\theta$. 

Since $\Fr(\lsup{G}\gamma) = \lsup{G}\gamma$, there exists $m \in G$ such that $\lsup{m}\gamma = \Fr(\gamma)$.  Note that $\gamma, \lsup{m} \gamma = \Fr(\gamma) \in  \lsup{\ell} A_\theta  \leq \lsup{\ell}A$.   It follows that $\lsup{m \ell}\bA$ and $\lsup{\ell}\bA$ are maximal $K$-split tori in $C_{\bG}(\lsup{m}\gamma)$.   Thus, there exists $r \in C_{G}(\lsup{m}\gamma)$  such that $\lsup{r m \ell} \bA = \lsup{\ell} \bA$.   Choose $n \in N_{G_{x_0,0}}(\bA)$ such that $n$ and $\ell\inv r m \ell$ have the same image in $W$.    Since $\ell n \ell\inv \in G_{x_0,0}$, by Lang-Steinberg there exists $h \in G_{x_0,0}$ such that $\Fr(h)\inv h= \ell n \ell\inv$.  Since $\lsup{r} (\lsup{m} \gamma) = \lsup{m} \gamma$ and $\lsup{n}(\lsup{\ell\inv}\gamma ) = \lsup{\ell\inv r m \ell}(\lsup{\ell\inv}\gamma ) $, we have
$$\Fr(\lsup{h} \gamma ) = \lsup{\Fr(h) m} \gamma = \lsup{\Fr(h) rm} \gamma = \lsup{(\Fr(h) \ell)(\ell\inv rm  \ell)(\ell\inv)} \gamma= \lsup{\Fr(h) (\ell n \ell\inv)} \gamma = \lsup{h} \gamma. $$

The proof that we may assume that  $C_{\bG}(\lsup{h}\gamma)$ is a $k$-quasi-split unramified twisted Levi subgroup of $\bG$ is nearly identical to the proof of the similar result for $C_{\bG}(\lsup{h}\bS)$ in Corollary~\ref{cor:cortocorto5.5.4}
\end{proof}

\subsection{Unramified tori and Levi \texorpdfstring{$(k,k)$-subgroups}{kk subgroups}}

Recall from Section~\ref{sec:toriandlevi2}
 that if $\bM'$ is a Levi $(k,k)$-subgroup of $\bG$, then we let $(\bM')$ denote the   $G^\Fr$-conjugacy class of $\bM'$.

\begin{lemma}   Suppose $(F,\theta,w) \in \tilde{I}^{\ee}$.  If $\bS \in j((F,\theta,w))$, then there exists $\bM' \in (\lsub{F}\bM)$ such that $\bS \leq \bM'$.  Moreover,  if $\bM$ is a Levi $(k,k)$-subgroup of $\bG$ that contains $\bS$, then $(\lsub{F}\bM) \leq (\bM)$.
\end{lemma}

\begin{proof}
Choose $\sfS \in \RR(F)$ that corresponds to $(\theta,w) \in \dot{I}^{\ee} (F)$.  The pair $(F,\sfS)$ belongs to $\JGammamax$, and $\bS$ is  $G^\Fr$-conjugate to a lift of $(F,\sfS)$.   The result follows from Lemmas~\ref{lem:prelimforM} and~\ref{lem:howelliptic}.
 \end{proof}

\begin{defn}
Suppose $E$ is a Galois extension of $k$.  A Levi $(E,k)$-subgroup $\bL$ is called \emph{elliptic} provided that any maximal $k$-split torus in $\bL$  coincides with the maximal $k$-split torus in the center of $\bG$.
\end{defn}
 
\begin{cor}
An unramified twisted Levi corresponding to the parameterizing data $(F, \theta, w) \in \tilde{I}^{\ee}$ is elliptic if and only if $F$ is a minimal facet in $\BB(G)^\Fr$. \qed
\end{cor}

\subsection{A more concrete realization of the parameterization}
To parametrize the elements of $\UU$, one only needs to look at $G^{\Fr}$-facets in $\AA(A)^\Fr$ up to equivalence, and on each $G^{\Fr}$-facet, look at $\dot{I}(F)$ up to 
the equivalence given by the natural action of  $N_{\dot{W}}(W_F)$ on $\dot{I}(F) / \stackrel{F}{\sim}$.  Here $\dot{W}$ denotes the image of $N_{G^{\Fr}}(\bA)/C_{G^{\Fr}}(\bA)$ in $W^{\Fr}$.   More specifically, one can reduce to the following situation:

For  $G^{\Fr}$-facets $F, F'$ in $\AA(A)^{\Fr}$ we will say that $F$ and $F'$ are equivalent provided that there exists $n \in \affFrW$ such that  $\emptyset \neq A(\AA(A)^\Fr, F) = A(\AA(A)^\Fr, nF')$.  Fix a set of representatives $\FF$ for the equivalence classes determined by this equivalence relation.

Fix an alcove $C$ in $\AA(A)^{\Fr}$. Without loss of generality, if $F \in \FF$, then $F$ is in the closure of $C$.  For $F \in \FF$, we say that $(\theta,w), (\theta',w') \in \dot{I}^{\ee}(F)$ are equivalent provided that there exists $m \in N_{\dot{W}}(W_F) \cdot W_F$ such that 
\begin{itemize}
    \item $m \Phi_\theta = \Phi_{\theta'}$
    \item $\Fr(m)wm^{-1} \in w'(W_F \cap W_{\theta'})$
\end{itemize}

For each $F \in \FF$ choose a set of representatives $\iota(F)$ in $\dot{I}^{\ee}(F)$  for the action described above. Without loss of generality, we also require that if $(\theta,w) \in \iota(F)$, then $w$ does not lie in a proper $\Fr$-parabolic subgroup of $W_F$.  The set 
$$\{(F,\theta,w) \, | \text{$ F \in \FF$ and $(\theta,w) \in \iota(F)$} \}$$
indexes the $G^{\Fr}$-conjugacy classes of unramified tori in $\bG$.

\begin{example}
We consider $\Sp_4$ and adopt the notation of Example~\ref{ex:initialsp4}.  There are sixteen $G^{\Fr}$-conjugacy classes of unramified tori.   Since $G^{\Fr}$ acts transitively on the alcoves in $\BB(G)^{\Fr}$, to see how the above correspondence works it is enough to restrict our attention to a single alcove.  In Figure \ref{fig:Sp4tori}, we enumerate the sixteen triples $(F, \theta, w)$ that occur, up to equivalence.  
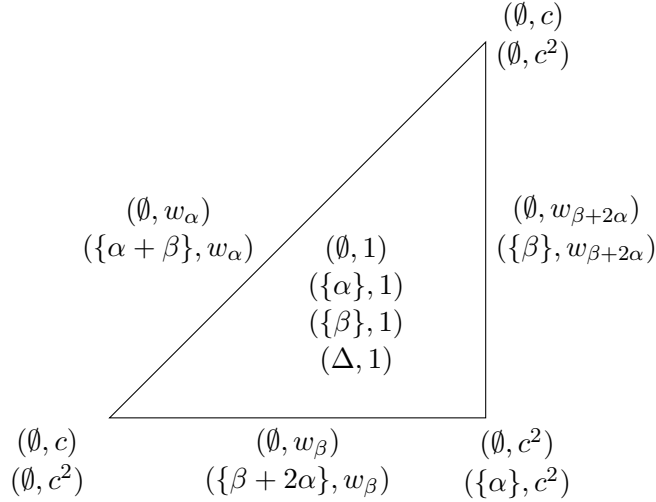
\begin{figure}[ht]
\centering
\begin{tikzpicture}
\draw (0,0) node[anchor=north]{\shortstack{$(\emptyset, c)$\\$(\emptyset, c^2)$} \hphantom{hellodolly}}
  -- (5,0) node[anchor=north]{$\hphantom{hello}$\shortstack{ $(\emptyset, c^2)$ \\ $(\{\alpha\}, c^2)$}}
  -- (5,5) node[anchor=west]{\shortstack{$(\emptyset,c)$\\$(\emptyset,c^2)$} $ \vphantom{P_{P_{P_{P_{P}}}}}$}
  -- cycle;
  \draw (2.5,0) node[anchor=north]{\shortstack{$(\emptyset, w_\beta)$ \\ $(\{\beta + 2 \alpha \}, w_\beta)$} };
  \draw (-.5,2.5) node[anchor=west]{\shortstack{$(\emptyset, w_\alpha)$ \\$(\{\alpha + \beta\},w_\alpha)$} $\hphantom{hellodollythree}$};
   \draw (5,2.5) node[anchor=west]{\shortstack{$(\emptyset, w_{\beta + 2 \alpha})$ \\ $(\{\beta\}, w_{\beta + 2 \alpha})$}};
   \draw (2.5,1.5) node[anchor=west]{\shortstack{$(\emptyset, 1)$ \\ $(\{\alpha\}, 1)$\\ $(\{\beta\},1)$\\$(\Delta, 1)$}};
\end{tikzpicture}
\caption{A parameterization  of the rational conjugacy classes of unramified tori  for $\Sp_4$ \label{fig:Sp4tori}}
\end{figure}
The centralizer of the unramified torus corresponding to the pair $(\{\alpha + \beta\}, w_\alpha)$ is unramified  $U(1,1)$ while the centralizer of the unramified torus corresponding to the pair $(\{\alpha\},c^2)$ is unramified $U(2)$  (using Jabon's notation~\cite{jabon:thesis}).  The centralizers of the tori corresponding to the pairs $(\{\beta\}, w_{\beta + 2 \alpha})$ and $(\{\beta + 2 \alpha\}, w_\beta)$ are of the form $\SL_2 \times \bS$ where $\bS(k)$ is the group of norm-one elements of an unramified quadratic extension of $k$.   The four unramified tori with labels of the form $(\theta, 1)$ are the $k$-split components of the centers of the four (up to rational conjugacy) distinct $k$-subgroups of $\Sp_4$ that occur as a Levi factor for a parabolic $k$-subgroup of $\Sp_4$.

\end{example}

\begin{example}
Let $\bG$ be a connected reductive group of type $\typeA_{n-1}$ such that $G^{\Fr} \cong \SL_1(D)$ where $D$ is a division algebra of index $n$ over $k$.  Recall that we may identify  $\bG(K)$ with $\SL_n(K)$.    Suppose $C$ is the alcove in $\AA(A) \leq \BB(G)$ for which $C^{\Fr} \neq \emptyset$.  Let $\{\psi_0, \psi_1 , \ldots , \psi_{n-1} \}$ be the simple affine $K$-roots determined by $\bG$, $\bA$, $\nu$, and $C$.   We assume that the $\psi_i$ are labeled so that $\Fr(\psi_i) = \psi_{i+1}$ mod $n$.

Suppose $1 \leq j \leq n$ and $n = jm$ for some $m \in \mathbb{N}$.  For $1 \leq \ell \leq (j-1)m$ set $$\alpha_\ell = \sum_{i=1}^m \dot{\psi}_{\ell + i}.$$
Note that for $1 \leq i \leq m$ the set $$\theta_j^i = \{ \alpha_i, \alpha_{i+m}, \ldots , \alpha_{i+(j-2)m} \}$$
is a basis for a root subsystem  of type $\typeA_{j-1}$ in $\Phi(\bG,\bA)$.  Moreover, the roots in $\theta_j^i$ are orthogonal to those in $\theta_j^{i'}$ for $i \neq i'$.

Define $\theta_j = \cup_{i = 1}^m \theta_j^i$. Note that if $j=1$, then   $\theta_j = \theta_1 = \emptyset$, while if $j = n$, then $\theta_j = \theta_n = \{ \dot{\psi}_2, \dot{\psi}_3,  \cdots , \dot{\psi}_{n-2}, \dot{\psi}_{n-1}, \dot{\psi}_0\}$. We have that  $C_\bG(\bA_{\theta_j})$ is a Levi $K$-subgroup in $\bG$ of type $\typeA_{j-1} \times \typeA_{j-1} \times \cdots \times \typeA_{j-1}$ where there are $m$ copies of $\typeA_{j-1}$ in this product.

Since $\Fr(\alpha_\ell) = \alpha_{\ell+1} \in \theta_j$ for $1 \leq \ell \leq (j-1)m -1$ and $\Fr(\alpha_{(j-1)m} ) = - (\alpha_1 + \alpha_{1+m} + \cdots + \alpha_{1 +(j-1)m}) \in \Phi_{\theta_j^1} \subset \Phi_{\theta_j}$, we can conclude that $\Phi_{\theta_j}$ is $\Fr$-stable.  
Hence, $C_\bG(\bA_{\theta_j})$ is a Levi $(K,k)$-subgroup; that is,  $\bA_{\theta_j}$ is an unramified torus in $\bG$.

The set $\{(C^{\Fr}, \theta_j, 1) \, | \, \text{$1 \leq j \leq n$ and $j$ divides $n$} \}$ is a complete set of representatives for ${\tilde{I}^{\ee}/\!\approx}$.   Indeed, $(C^{\Fr}, \theta_j, 1)$ corresponds to the $G^\Fr$-conjugacy class of $\bA_{\theta_j}$.  In turn, these tori  correspond to the maximal unramified extensions that occur in  fields $E \leq D$ that split over a degree $n$ extension of $k$  and whose maximal unramified subfield has degree $n/j$ over $k$.
\end{example}

\section{Stable conjugacy}
\label{sec:stableconj}

Suppose $\bS$ is an unramified torus in $\bG$.   A $k$-embedding of $\bS$ into $\bG$ is a $k$-morphism $f \colon \bS \rightarrow \bG$ for which there exists $g \in G$ such that $f(s) = \lsup{g}s$ for all $s \in S$.
In this section we investigate the $k$-embeddings of $\bS$ into $\bG$ and enumerate these $k$-embeddings up to $G^{\Fr}$-conjugacy.

\subsection{Two indexing sets over \texorpdfstring{$k$}{k}}
Recall from Section~\ref{sec:anindexingset} that we have defined
$$\dot{I} = \{ (\theta,w) \, | \, \theta \in \Theta(\bG,\bA) \text{, } w \in W \text{, and } \Fr(\Phi_\theta) = w \Phi_\theta \}.$$
For $(\theta',w'), (\theta,w) \in \dot{I}$ we write $(\theta',w') \dot{\sim} (\theta,w)$ provided that there exists a $n \in W$ for which
\begin{itemize}
\item  $\Phi_{\theta '} =  n  \Phi_{\theta}$ and
\item $w' \in \Fr(n) w {n} ^{-1} W_{\theta'}$.
\end{itemize}

\begin{lemma}
The relation $\dot{\sim}$ is an equivalence relation on $\dot{I}$. \qed
\end{lemma}

There is a smaller indexing  set that carries much of  the same information as $\dot{I}$.  Let $I$ denote the set of pairs $(\theta,w)$ where $\theta \subset \Delta$ and $w \in W$ such that $ \Fr (\theta) = w \theta$; we  require neither that $w^{-1} \circ \Fr$ fix $\theta$ point-wise nor that $\Fr(\theta)$ and $w \theta$ are subsets of $\Delta$.   For $(\theta',w'), (\theta,w) \in I$ we write $(\theta',w') \sim (\theta,w)$ provided that there exists a $n \in W$ for which
\begin{itemize}
\item  $\theta ' =  n  \theta$ and
\item $w' = \Fr(n) w {n} ^{-1}$.
\end{itemize}

\begin{lemma}
The relation $\sim$ is an equivalence relation on $I$. \qed
\end{lemma}

\begin{example}
For $\Sp_4(k)$ representatives of the ten equivalence classes in $I$ can be taken to be: $(\emptyset, 1)$; 
$(\emptyset, w_\alpha)$; 
$(\emptyset, w_\beta)$; 
$(\emptyset, w_\alpha w_\beta)$; 
$(\emptyset,  w_\alpha w_\beta w_\alpha w_\beta)$; 
$(\{\alpha\}, 1)$; 
$(\{\alpha\}, w_\beta w_\alpha w_\beta)$; 
$(\{\beta\}, 1)$; 
$(\{\beta\}, w_\alpha w_\beta w_\alpha)$;   and 
$(\Delta, 1)$.
\end{example}

\begin{lemma}  \label{lem:yexists}
Suppose $\theta \subset \Delta$ and $w \in W$.   If $ \Fr (\Phi_\theta) = w \Phi_\theta$, then there exists a unique $y \in W_\theta$ for which $\Fr (\theta) = wy \theta$.  
\end{lemma}

\begin{rem}  If there exists  $y \in W_\theta$ for which $\Fr (\theta) = wy \theta$, then we have that $ \Fr (\Phi_\theta) = w \Phi_\theta$.
\end{rem}

\begin{proof}
Note that $\Fr(\theta)$ is a basis for $\Fr (\Phi_\theta)$ and $w^{-1} \Fr(\theta)$ is a basis for  $ w^{-1} \Fr (\Phi_\theta) = \Phi_\theta$.  Since $W_\theta$ acts simply transitively on the set of bases for $\Phi_\theta$, there exists a unique $y \in W_\theta$ for which $y ^{-1} w^{-1}\Fr (\theta) = \theta$.
\end{proof}

\begin{lemma}
The natural inclusion $I \hookrightarrow \dot{I}$ induces a bijection between ${I/\!\sim}$ and $\dot{I}/\dot{\sim}$.
\end{lemma}

\begin{proof}
Consider the map $\iota \colon I \rightarrow \dot{I}/\dot{\sim}$ defined by sending $(\theta, w) \in I$ to the equivalence class of $(\theta,w)$ in $\dot{I}/\dot{\sim}$.

We first show that $\iota$ is surjective.  Suppose $(\theta',w') \in \dot{I}$.   There exists $n \in W$ such that $n\inv \theta' \subset \Delta$.    Set $w = \Fr(n)\inv w' n$ and $\theta = n\inv \theta'$.  Note that
$$w \Phi_\theta =  \Fr(n)\inv w' n \Phi_{n\inv \theta'} = \Fr(n)\inv w' \Phi_{\theta'} = \Fr(n \inv) \Fr(\Phi_{\theta'}) = \Fr(n\inv \Phi_{\theta'}) = \Fr(\Phi_\theta).$$
Thanks to  Lemma~\ref{lem:yexists}, there exists $y \in W_\theta$ such that $\Fr(\theta) = wy \theta$.    Hence $(\theta,wy) \in \iota(I)$.   We need to show $(\theta',w') \dot{\sim} (\theta,wy)$.

Since $n\theta = \theta'$ we have $\Phi_{\theta'} = n \Phi_{\theta}$.   Since $w' = \Fr(n) w n\inv$, we have
$$w' = (\Fr(n) w n\inv) (nyn\inv) (n y\inv n\inv) \in (\Fr(n) w n\inv) (nyn\inv) W_{\theta'} =  \Fr(n) wy n\inv W_{\theta'}.$$
Hence $(\theta',w') \dot{\sim} (\theta,wy)$, and we conclude that  $\iota$ is surjective.

Suppose $(\dot{\theta},\dot{w}), (\hat{\theta},\hat{w}) \in I$ with $\iota((\dot{\theta},\dot{w})) = \iota((\hat{\theta},\hat{w}))$.    Then there  exists an $n \in W$ for which
 $\Phi_{\hat{\theta}} =  n  \Phi_{\dot{\theta}}$ and
 $\hat{w} \in \Fr(n) \dot{w} {n} ^{-1} W_{\hat{\theta}}$. 
 Since $n\dot{\theta} \subset \Phi_{\hat{\theta}}$ and  $W_{\hat{\theta}}$ acts (simply) transitively on the bases of $\Phi_{\hat{\theta}}$, there exists $y \in W_{\hat{\theta}}$ such that $yn \dot{\theta} = \hat{\theta}$.  Let $\tilde{n} = yn$.   Note that  $\hat{\theta} = \tilde{n} \dot{\theta}$.
 Since $y \in W_{\hat{\theta}}$, we have $\Fr(y) \in  W_{\Fr (\hat{\theta})} = \hat{w} W_{\hat{\theta}} \hat{w}\inv$. Consequently,
  $$\tilde{n} \dot{w}\inv \Fr(\tilde{n})\inv \hat{w}  = y (n \dot{w}\inv \Fr(n) \inv) \Fr(y)\inv \hat{w} \in y (W_{\hat{\theta}} \hat{w}\inv) \Fr(y)\inv \hat{w} =  (y W_{\hat{\theta}})  (\hat{w}\inv \Fr(y)\inv \hat{w}) =  y W_{\hat{\theta}} = W_{\hat{\theta}}.$$
  Since $W_{\hat{\theta}}$ acts simply transitively on the bases of $\Phi_{\hat{\theta}}$ and
 $$(\tilde{n} \dot{w}\inv \Fr(\tilde{n})\inv \hat{w}) 
 \hat{\theta} 
 = \tilde{n} \dot{w}\inv \Fr(\tilde{n})\inv \Fr(\hat{\theta})
 = \tilde{n} \dot{w}\inv \Fr(\tilde{n}\inv \hat{\theta})
  = \tilde{n} \dot{w}\inv \Fr(\dot{\theta})
    = \tilde{n} \dot{\theta} 
    =\hat{\theta},$$
 we conclude that $\tilde{n} \dot{w}\inv \Fr(\tilde{n})\inv \hat{w} = 1$, that is $\hat{w} = \Fr(\tilde{n}) \dot{w} \tilde{n}\inv$.
 Hence $(\dot{\theta}, \dot{w}) \sim (\hat{\theta},\hat{w})$, and the induced map is  injective.
\end{proof}

\subsection{Stable conjugacy of unramified tori}   \label{sec:6.2}

\begin{defn}  Suppose $\bS$ and $\bS'$ are unramified tori.  We will say that $\bS$ and $\bS'$ are \emph{stably conjugate} provided that there exists $h \in G$ such that $\lsup{h} \bS = \bS'$ and $\Int(h) \colon \bS \rightarrow \bS'$ is a $k$-isomorphism.
\end{defn}

\begin{rem}  With notation as above, $\Int(h) \colon \bS \rightarrow \bS'$ is a $k$-isomorphism if and only if $\Fr(\lsup{h}s) = {\lsup{h}s}$ for all $s\in S^{\Fr}$.
\end{rem}

\begin{rem}  \label{rem:restrictiontoKOK}
    It would perhaps be more standard to say that $\bS$ and $\bS'$ are {stably conjugate} provided that there exists $g \in \bG(\bar{k})$ such that $\lsup{g} \bS = \bS'$ and $\Int(g) \colon \bS \rightarrow \bS'$ is a $k$-isomorphism.  To see the equivalence of this definition to the one given, suppose we have such a $g$.    Let $\bM$ be the Levi $(K,k)$-subgroup $C_{\bG}(\bS)$.  Choose $s \in S^{\Fr}$ such that $C_{\bG}(s) = C_{\bG}(\bS)$.   For all $\gamma \in \Gal(\bar{k}/K)$, we have $\lsup{g}s = \gamma(\lsup{g}s) = \lsup{\gamma(g)}s$; hence $g\inv \gamma(g) \in \bM(\bar{k})$.  Since $\cohom^1(\Gal(\bar{k}/K),\bM)$ is trivial, there exists $m \in \bM(\bar{k})$ such that $ g\inv \gamma(g) = m \gamma(m\inv)$ for all $\gamma \in I$.   Hence $gm \in G$ and for all $\tilde{s} \in S$ we have $\lsup{gm}\tilde{s} = \lsup{g}\tilde{s}$. 
\end{rem}

\begin{lemma} \label{lem:6.2.4}
Suppose the unramified tori $S_i$ for $i \in \{1,2\}$ belong to $G^{\Fr}$-conjugacy classes parameterized by data $(F_i,\theta_i,w_i) \in \tilde{I}$.  We have  $(\theta_1,w_1) \dot{\sim} (\theta_2,w_2)$ in $\dot{I}$ if and only if $S_1$ and $S_2$ are stably conjugate.
\end{lemma}

\begin{proof}
Since the $G^{\Fr}$-conjugacy class of  $S_i$ is parameterized by   $(F_i,\theta_i,w_i)$ we can assume that $S_i = {\lsup{g_i} A_{\theta_i}}$ with $g_i \in G_{F_i}$ and $n_i := \Fr(g_i)^{-1} g_i$ having image $w_i$ in $W_{F_i} \leq W$.  

Suppose first that $(\theta_1,w_1) \dot{\sim} (\theta_2,w_2)$ in $\dot{I}$.   Then there exists $w \in W$ for which  $w \Phi_{\theta_1} = \Phi_{\theta_2}$ and $w_2^{-1} \Fr(w) w_1 w^{-1} \in W_{\theta_2}$.   Let $n \in N_G(A)$ be a  lift of  $w$.  Note that
$$\Fr(g_2 n g_1 ^{-1}) = g_2 (g_2 ^{-1} \Fr(g_2)) \Fr(n) (\Fr(g_1)^{-1} g_1) g_1 ^{-1} = 
g_2 n g_1^{-1} \cdot \lsup{g_1n^{-1}} (n_2^{-1} \Fr(n) n_1 n^{-1}).$$
Since $n_2^{-1} \Fr(n) n_1 n^{-1} \in N_G(\bA)$ has image in $W_{\theta_2}$ and $\lsup{n\inv} W_{\theta_2} = W_{\theta_1}$,  the element $\lsup{g_1n^{-1}} (n_2^{-1} \Fr(n) n_1 n^{-1})$ acts trivially on $S_1$, and so 
$$\Fr( \lsup{g_2 n g_1\inv} s) = \lsup{g_2 n g_1\inv}s$$
for all $s \in S_1^{\Fr}$.

For the other direction, suppose that $S_1$ and $S_2$ are stably conjugate by $h \in G$; i.e. $\Int(h) \colon \bS_1 \rightarrow \bS_2$ is a $k$-isomorphism.  
Let $\bL_i = C_{\bG}(\bS_i)$.  
Since $\lsup{h g_1}\bA$ and $\lsup{ g_2}\bA$ are maximal $K$-split tori in $\bL_2$, there exists $\ell_2 \in L_2$ such that $\lsup{\ell_2 h g_1}\bA = {\lsup{ g_2}\bA}$.  
Note that $m = {g_2 ^{-1} \ell_2 h g_1} \in N_G(A)$.  Let $n$ denote the image of $m$ in $W$.  Since 
$\lsup{\ell_2 h  g_1}\bM_{\theta_1} = \lsup{ g_2}\bM_{\theta_2}$,
we have $n \Phi_{\theta_1} = \Phi_{\theta_2}$.   Note that 
$\Fr(n)w_1n^{-1}$ is the image in $W$ of
$$\Fr(g_2)  ^{-1} \Fr(\ell_2) \Fr(h)  \Fr(g_1) (\Fr(g_1)^{-1} g_1 ) g_1 ^{-1} h^{-1} {\ell_2}^{-1} g_2$$
which is
$$ (\Fr(g_2) ^{-1} g_2) \cdot g_2 ^{-1} (\Fr( h^{-1} {\ell_2}^{-1} )^{-1}      h^{-1} {\ell_2}^{-1}       )g_2.$$
Since  $\Int(h\inv) \colon \bS_2 \rightarrow \bS_1$ is also a $k$-isomorphism, we have $\Fr(\lsup{h\inv}s_2) = \lsup{h\inv}s_2$ for all $s_2 \in S_2^{\Fr}$, hence $\Fr(h) h\inv \in L_2$.  Thus, $\Fr( h^{-1} {\ell_2}^{-1} )^{-1}      h^{-1} {\ell_2}^{-1}  \in L_2$.   Since $\lsup{g_1} \bA$ and $\lsup{g_2}\bA$ are maximal $K$-split $k$-tori in $\bG$, we have 
$\lsup{g_1}\bA = \lsup{h^{-1} \ell_2 ^{-1}} (\lsup{g_2}\bA) = \lsup{\Fr(h)^{-1} \Fr(\ell_2) ^{-1}} (\lsup{g_2}\bA)$.   We conclude that $g_2 ^{-1} (\Fr( h^{-1} {\ell_2}^{-1} )^{-1}      h^{-1} {\ell_2}^{-1}       )g_2 \in N_G(\bA)$ has image in $W_{\theta_2}$.  Hence, we have $\Fr(n)w_1n^{-1} \in w_2 W_{\theta_2}$.  
\end{proof}

\begin{cor}
There is a bijective correspondence between $I/\!\sim$ and the set of stable classes of unramified tori in $\bG$.   \qed
\end{cor}

\subsection{Example: Stable classes of unramified tori associated to Levi \texorpdfstring{$(k,k)$-subgroups}{(k,k) subgroups}}

   For $i \in \{1,2\}$ let $\bM_i$ denote a Levi $(k,k)$-subgroup of $\bG$. After conjugating $\bM_i$ by an element of $G^{\Fr}$,  we may and do assume that $\bA \leq \bM_i$.   As in Section~\ref{sec:ratlevis} we may associate to $\bM_i$ a triple
    $(C,\theta_i,1)  \in \tilde{I}^{\ee}$ where $C$ is an alcove in $\AA(A)^{\Fr}\!$, $\theta_i \in \Theta$ such that  $\Phi_{\theta_i} = \Fr (\Phi_{\theta_i})$, 
 $\bM_i = C_\bG(\bA_{\theta_i})$, and  $\bA_{\theta_i} \in j((C,\theta,1))$.     From Lemma~\ref{lem:6.2.4} we know that $A_{\theta_1}$ is stably conjugate to $A_{\theta_2}$ if and only if there exists $n\in W$ such that $n\Phi_{\theta_1} = \Phi_{\theta_2}$ and $\Fr(n)n\inv \in W_{\theta_2}$.     From Theorem~\ref{thm:maintheorem} we know that $A_{\theta_1}$ is $G^{\Fr}$-conjugate to $A_{\theta_2}$ if and only if there exists $x \in N_{G^{\Fr}}(\bA)/C_{G^{\Fr}}(\bA)$ such that $x\Phi_{\theta_1} = \Phi_{\theta_2}$.

   Thanks to the work of Solleveld~\cite{solleveld:conjugacy} we know that 
   $A_{\theta_1}$ and $A_{\theta_2}$ are stably conjugate if and only if they are $G^{\Fr}$-conjugate.  In fact, Solleveld shows:  
   $A_{\theta_1}$ and $A_{\theta_2}$ are $\bG(\bar{k})$-conjugate if and only if they are $G^{\Fr}$-conjugate.

   The equivalence of stable and rational conjugacy of $A_{\theta_1}$ and $A_{\theta_2}$ can also be established as follows.
   
   Suppose that $\bG$ is $k$-quasi-split and there   exists $\bar{n}\in W$ such that $\bar{n}\Phi_{\theta_1} = \Phi_{\theta_2}$ and $\Fr(\bar{n})\bar{n}\inv \in W_{\theta_2}$.  
    Since $\bG$ is $k$-quasi-split, there exists an absolutely special vertex $x_0$ that belongs to the image of $\AA(A)$ in  $\BB^{\red}(G)^{\Fr}$.
    Choose $n \in N_{G_{x_0,0}}(\bA)$ that lifts $\bar{n}$.  Note that $\Fr(n) n\inv \in M_2$ and so $\Int(n) \colon  \bA_{\theta_1} \rightarrow \bA_{\theta_2}$ is a $k$-isomorphism.  Since $n \in N_{G_{x_0,0}}(\bA)$, by Lang-Steinberg there exists $k \in G_{x_0,0} \cap M_2$ such that $\Fr(k) k\inv = \Fr(n) n\inv$.  Thus $k\inv n \in G^{\Fr}$ and $\lsup{k\inv n}\bA_{\theta_1} = \bA_{\theta_2}$.
    
We adapt Waldspurger's argument~\cite[Lemma~4]{solleveld:conjugacy} to extend this result to any reductive $k$-subgroup.  We adopt the notation of Section~\ref{sec:ratlevis}.  In particular, $\bG^*$ is a quasi-split inner form of $\bG$, $\bA^* \leq \bB^*$ is a Borel-torus pair in $\bG^*$, and $\Ad(m)$ is an element of $N_{G^*_{\ad}}(\bA^*) \cap \Stab_{G^*_{\ad}}(D')$ such that
 $\bG_K = \bG^*_K$, $\bA_K = \bA^*_K$, $G = G^*$ as abstract groups, and $\Fr(g) = \Ad(n) \circ \Fr^*(g)$ for all $g \in G$.   Let $\theta_m \subset \Delta = \Delta(\bG^*,\bB^*,\bA^*)$ denote those simple roots $\alpha$ for which $\bA^k \subset \ker(\alpha)$.
As we may, we choose $\theta_i \subset \Delta$ such that $\theta_m \subset \theta_i$ and $\Fr^*({\theta_i}) = {\theta_i}$.  Since  $\theta_m \subset \theta_i$, this means $\Fr^*(\Phi_{\theta_i}) = (\Ad(m) \circ \Fr^*)(\Phi_{\theta_i}) = \Phi_{\theta_i}$.   Note that $(\bA_{\theta_i})_K = (\bA^*_{\theta_i})_K$ is a $K$-split torus in $\bG_K=\bG^*_K$.

Suppose that $\bA_{\theta_1}$ and $\bA_{\theta_2}$ are stably conjugate in $\bG$.  Then there exists $n \in W$ such that $n \Phi_{\theta_1} = \Phi_{\theta_2}$ and $ \Fr(n)n\inv = (\Ad(m) \circ \Fr^*)(n) n\inv  \in W_{\theta_1}$.  Let $\bar{m}$ denote the image of $\Ad(m)$ in $W$.   Then $\bar{m} \in W_{\theta_m} \leq W_{\theta_1} \cap W_{\theta_2}$.   Thus $m \Fr^*(n \bar{m} n\inv) \in W_{\theta_1}$ and  
$\Fr^*(n\bar{m})(n\bar{m})\inv = (m \Fr^*(n \bar{m} n\inv))\inv \Fr(n) n\inv \in W_{\theta_1} $.
    Thus, we have $n \bar{m} \in W$ such that $n \bar{m} \Phi_{\theta_1} = \Phi_{\theta_2}$ and $\Fr^*(n \bar{m})(n \bar{m})\inv \in W_{\theta_1}$.  That is, $\bA^*_{\theta_1}$ and $\bA^*_{\theta_2}$ are stably conjugate in $\bG^*$.

    Since $\bG^*$ is $k$-quasi-split, we conclude that  $\bA^*_{\theta_1}$ and $\bA^*_{\theta_2}$ are rationally conjugate in $\bG^*$.  Thus, there exists $x \in N_{(G^*)^{\Fr}}(\bA^*)$ such that $x \Phi_{\theta_1} = \Phi_{\theta_2}$.   Let $\bP_1$ be the standard parabolic subgroup in $\bG$ corresponding to $\theta_1$ and set $\bP_2 = \Int(x) \bP_1$.   Since $m \in (L_{\ad})_{\theta_m}$, we have $\Fr(\bP_j) = \Ad(m) \Fr^*( \bP_j) = \bP_j$ for $j \in \{1, 2\}$, we conclude that $\bP_1$ and $\bP_2$ are parabolic $k$-subgroups in $\bG$.   They are conjugate by $x \in G^* = G$, hence these parabolic subgroups  are conjugate by $y \in G^{\Fr}$.  We have $\bP_2 = \Ad(y) \bP_1$ and $\Ad(y) \bL_{\theta_1}$ is a Levi factor of $\bP_2$.  Thus, there exists $u \in P_2^{\Fr}$ such that $uy \bL_{\theta_1} = \bL_{\theta_2}$.

\subsection{\texorpdfstring{$k$-embeddings}{k embeddings} of unramified tori}  \label{sec:$k$-embeddings}
 
 For many purposes in harmonic analysis, it is not enough to understand the stable conjugacy classes of (unramified) tori.  We therefore introduce the following refinement.
 
 \begin{defn}  Suppose $\bS$ is an unramified torus in $\bG$.   A \emph{$k$-embedding} of $\bS$ into $\bG$ is a map $f \colon \bS \rightarrow \bG$ such that 
 \begin{enumerate}
     \item there exists $g \in G$ such that $f(s) = {\lsup{g}s}$ for all $s \in S$ and
     \item $f$ is a $k$-morphism.
 \end{enumerate}
 \end{defn}
 
 \begin{example}  If $\bS_1$ and $\bS_2$ are stably conjugate unramified tori in $\bG$, then there exists a $k$-embedding $h \colon \bS_1 \rightarrow \bG$ such that $h[\bS_1] = \bS_2$.
 \end{example}

 \begin{rem}
 Suppose $\bS$ is an unramified torus in $\bG$ and $f \colon \bS \rightarrow \bG$ is a $k$-embedding.
 \begin{enumerate}
     \item If $g,h \in G$ such that $f(s) = \lsup{g}s = \lsup{h}s$ for all $s \in S$, then $g \in h C_G(S)$.
     \item Since $f$ is a $k$-morphism, we have $\Fr(f(s)) = f(s)$ for all $s \in S^{\Fr}$.
 \end{enumerate}
 \end{rem}
 
 We will be most interested in those $k$-embeddings of an unramified torus $\bS$ into $\bG$ for which the images of the $k$-embeddings are $\bS$.

 \begin{rem}  
    It would perhaps be more standard to say that  a {$k$-embedding} of $\bS$ into $\bG$ is a map $f \colon \bS \rightarrow \bG$ such that (a)
there exists $g \in \bG(\bar{k})$ such that $f(s) = {\lsup{g}s}$ for all $s \in S$ and (b)
 $f$ is a $k$-morphism.
    The equivalence of this definition and the one given can be seen by arguing as in Remark~\ref{rem:restrictiontoKOK}.
\end{rem}

 \subsubsection{Notation for Section \ref{sec:$k$-embeddings}}
 
  Suppose $\bS_1$ is an unramified torus and set $\bL_1 = C_\bG(\bS_1)$.   Suppose that $\bS_1$ corresponds to $(F_1,\theta_1,w_1) \in \tilde{I}^{\ee}$ and $\bS_1 = {\lsup{g_1}\bA_{\theta_1}}$ for $g_1 \in  \lsub{F_1}\bM(K) \cap G_{F_1}$ with $n_1 := \Fr(g_1)\inv g_1$ having image $w_1$ in $W$.  
 
 Recall  from Lemma~\ref{lem:anisotropicasexpected} that $\sfT_1 = C_{\sfG_F}(\sfS_1)$ is an $\ff$-minisotropic maximal $\ff$-torus in $\sfG_F$ that corresponds to $(\emptyset,w_1) \in \dot{I}^{\ee}(F_1)$.  Let $\bT_1 \leq \lsub{F_1}\bM$ be a lift of $\sfT_1$ that contains $\sfS_1$.   Thanks to Lemma~\ref{lem:alcoveinM} and Corollary~\ref{cor:improved541}, the $k$-rank of $\bL_1$ is equal to the $k$-rank of $\bT_1^k$, hence  $\bT_1  \leq \bL_1$ is a maximally $k$-split maximal $K$-split $k$-torus in $\bL_1$ that contains $\bS_1$.
 
 Finally, we may assume that $\BB(T_1)^\Fr$ is a subset of $\AA(A)^\Fr$.  Indeed, there exists $m \in \lsub{F_1}M^\Fr$ for which $\bT_1^k \leq \lsup{m}\bA$, hence $\BB(T_1)^\Fr \subset \AA(\lsup{m}A)^\Fr$.  Since $F_1 \subset \AA(\lsup{m}A)^\Fr$, there exists $\ell \in \lsub{F_1}M_{F_1}^\Fr$ such that
 $$\BB(\lsup{\ell}T_1)^\Fr = \ell \cdot \BB(T_1)^\Fr \subset \ell \cdot \AA(\lsup{m}A)^\Fr = \AA(A)^\Fr.$$
Without loss of generality, we may replace $g_1$ by $\ell g_1$.

 \subsubsection{Results on normalizers}
 
 Let $\bZ_1$ denote the center of $\bL_1$.   If $h \in C_G(\bZ_1)$, then since $\bS_1$ is the unique $K$-split subtorus in $\bZ_1$, we must have $h \in C_{G}(\bS_1) = L_1$.  Since $\bL_1 \leq C_{\bG}(Z_1)$, we have $C_{\bG}(Z_1) = \bL_1$

\begin{lemma}  \label{lem:SandT}
Suppose $g \in G$.  We have 
   $\lsup{g}\bS_1$ is defined over $k$ if and only if $\lsup{g}\bZ_1$ is defined over $k$.  
\end{lemma} 
 
 \begin{proof}
 If $\lsup{g}\bS_1$ is defined over $k$, then $ \lsup{g}\bL_1 = C_{\bG}(\lsup{g}S_1)$ is defined over $k$, and thus so too is its center, $\lsup{g}\bZ_1$.    On the other hand, if $\lsup{g}\bZ_1$ is defined over $k$,  then $\lsup{g}\bS_1$, the maximal $K$-split torus in  $\lsup{g}\bZ_1$, is also defined over $k$.
 \end{proof}
 
 \begin{cor}
 Suppose  $g \in G$ and  $\lsup{g}\bS_1$ is defined over $k$.  The map $\Int(g) \colon \bS_1 \rightarrow {\lsup{g}\bS_1}$ is a $k$-isomorphism  if and only  if $\Int(g) \colon \bZ_1 \rightarrow {\lsup{g}\bZ_1}$ is a $k$-isomorphism.
 \end{cor}
 
 \begin{proof}
 Thanks to Lemma~\ref{lem:SandT}
 both $\lsup{g}\bS_1$ and $\lsup{g}\bZ_1$ are defined over $k$.
 
 If  $\Int(g) \colon \bS_1 \rightarrow {\lsup{g}\bS_1}$ is a $k$-isomorphism, then $\Fr(g)^{-1} g \in L_1$ and so $\lsup{\Fr(g)^{-1} g}t = t$  for all $t \in Z_1$; hence $\Int(g) \colon \bZ_1 \rightarrow {\lsup{g}\bZ_1}$ is a $k$-isomorphism.
 
 On the other hand,  if  $\Int(g) \colon \bZ_1 \rightarrow {\lsup{g}\bZ_1}$ is a $k$-isomorphism, then since $\bS_1$ (resp. $\lsup{g}\bS_1$) is the unique maximal $K$-split $k$-torus in $\bZ_1$ (resp. $\lsup{g}\bZ_1$), the map $\Int(g) \colon  \bZ_1 \rightarrow {\lsup{g}\bZ_1}$ restricts to a $k$-isomorphism $\Int(g) \colon \bS_1 \rightarrow {\lsup{g}\bS_1}$.
 \end{proof}

 \begin{lemma} \label{lem:mapsofinterest}   If $g \in G$ such that  $\lsup{g}S_1 = S_1$, then there exists $\ell \in L_1$ such that  ${\ell g} \in N_G(\bT_1)$.  
 \end{lemma}
 
 \begin{proof}
If $\lsup{g}S_1 = S_1$, then we have  $g \in N_G(\bL_1)$. 
Both $\bT_1$ and  $\lsup{g}\bT_1$ are maximal $K$-split tori in $\bL_1$. Thus, there exists $\ell \in L_1$ such that
 $\lsup{\ell g}\bT_1 = \bT_1$.    
 \end{proof}

\begin{lemma}\label{lem:rationaln} 
If $g \in G^{\Fr}$ such that  $\lsup{g}S_1 = S_1$, then there exists $\ell \in L_1^{\Fr}$ such that  ${\ell g} \in N_{G^{\Fr}}(\bT_1) \leq  N_G(\bT_1)$. 
 \end{lemma}
 
 \begin{rem}
 For  $g$ and $\ell$ as in Lemma~\ref{lem:rationaln},  the map  $\Int(\ell g) \colon \bS_1 \rightarrow \bS_1 $ is a $k$-isomorphism.
 \end{rem}
 
 \begin{proof}
   Both $\bT_1$ and  $\lsup{g}\bT_1$ are maximally $k$-split maximal $(K,k)$-tori in $\bL_1$.  Thanks to~\cite[Theorem~6.1]{prasad:unramified} they are $L_1^{\Fr}$-conjugate.
\end{proof}   

\subsubsection{Classifying $k$-embeddings of $\bS_1$ into $\bG$ with image $\bS_1$}

\begin{defn}  For $H \subset G$, set
$$N^k(H,S_1) := \{h \in H \, | \,{\lsup{h}S_1} = S_1 \text{ and } \Int(h) \colon \bS_1 \rightarrow \bS_1 \text{ is a $k$-isomorphism}\}.$$
\end{defn}

\begin{example}  We have
$$N^k(L_1,S_1) = L_1 \text{, } N^k(C_G(\bT_1),S_1) = C_G(\bT_1) \text{, and } N^k(G^\Fr,S_1) = N_{G^\Fr}(S_1).$$
\end{example}

\begin{rem}
The quotient $ N^k(G,S_1) / L_1$ indexes the distinct $k$-embeddings of $\bS_1$ into $\bG$ that have image $\bS_1$.
\end{rem}

Note that $N^k(G,S_1)$, $L_1$, $N^k(N_G(\bT_1),S_1)$, and $N_{L_1}(\bT_1)$ are $\Fr$-modules.

\begin{lemma} \label{lem:firststeps}
There is a natural $\Fr$-equivariant (group) isomorphism
$$\xi \colon N^k(G,S_1)/L_1 \longrightarrow N^k(N_G(\bT_1),S_1)/N_{L_1}(\bT_1).$$
\end{lemma}

\begin{rem}
Since $L_1$ is a normal subgroup of $N^k(G,S_1)$ and $N_{L_1}(\bT_1)$ is a normal subgroup of $N^k(N_G(\bT_1),S_1)$,  we may quotient out on either the left or right for both the source and target of $\xi$.
\end{rem}

\begin{proof}
Suppose $g \in N^k(G,S_1)$.  From Lemma~\ref{lem:mapsofinterest} there exists $\ell \in L_1$ such that $\ell g \in N_G(\bT_1)$.   Since $\lsup{\ell}s = s$ for all $s \in S_1$, we have that $\Int(\ell g) \colon \bS_1 \rightarrow \bS_1$ is a $k$-isomorphism; hence $\ell g \in N^k(N_G(\bT_1),S_1)$.

If $\ell' \in L_1$ is another choice such that $\ell' g \in N^k(N_G(\bT_1),S_1)$, then $(\ell g)(\ell' g)\inv \in N_G(\bT_1)$.  Thus $\ell (\ell')\inv \in N_G(\bT_1) \cap L_1 = N_{L_1}(\bT_1)$.   So $\ell = \lambda \ell'$ for some $\lambda \in N_{L_1}(\bT_1)$.

In this way we can define a map $\bar{\xi} \colon N^k(G,S_1) \rightarrow  N_{L_1}(\bT_1) \setminus  N^k(N_G(\bT_1),S_1)$.   If $g, h \in N^k(G,S_1)$ with $\bar{\xi}(h) = \bar{\xi}(g)$, then there exist $\ell_g, \ell_h \in L_1$ such that $\ell_g g \in N_{L_1}(\bT_1) \ell_h h$, that is $g \in L_1 h$.   Hence $\bar{\xi}$ descends to a bijective map 
$$\xi \colon L_1 \setminus N^k(G,S_1) \longrightarrow N_{L_1}(\bT_1) \setminus N^k(N_G(\bT_1),S_1).$$

To see that $\xi$ is $\Fr$-equivariant, note that if $g \in N^k(G,S_1) $, then there exists $\ell \in L_1 = \Fr(L_1)$ such that $\xi(L_1 g)$ is the image of $\ell g \in N^k(N_G(\bT_1),S_1)$ in $N^k(N_G(\bT_1),S_1)/N_{L_1}(\bT_1)$.   We have $\Fr(\ell g) \in \Fr(N_G(\bT_1)) = N_G(\bT_1)$, hence
\[
\Fr(\xi(L_1 g)) = \Fr( \ell g N_{L_1}(T_1)) = (\Fr(\ell) \Fr(g))N_{L_1}(T_1) = \xi(\Fr(L_1 g)).     
\qedhere
\]
\end{proof}

\begin{defn}
Suppose $(\theta,w) \in \dot{I}$.  Put
$$W_{w \circ \Fr, \theta} := \{w'  \in N_W(W_{\theta})  \colon  w^{-1} \Fr(w')^{-1}  w w' \in W_{\theta}\}  .$$
\end{defn}

\begin{rem}
Note that 
$W_{w \circ \Fr, \theta}$ is a group,  $N_W(A_{\theta}) = N_W(W_{\theta})$, and $W_{\theta} \unlhd W_{w \circ \Fr, \theta}$.   Also, $W_{w \circ \Fr, \theta}$ consists of precisely those $w'  \in N_W(W_{\theta})$ for which  $w\inv \circ \Fr$ preserves the coset $w'W_{\theta}$ in $N_W(W_{\theta})/W_{\theta}$ .
\end{rem}

\begin{lemma} \label{lem:quotient}
There is a natural $\Fr$-equivariant group isomorphism  
$$ \eta \colon  N^k(N_G(\bT_1),S_1)/C_G(\bT_1) \longrightarrow W_{w_1 \circ \Fr, \theta_1}.$$
Here $\Fr$ acts on $W_{w_1 \circ \Fr, \theta_1}$ via $w_1\inv \circ \Fr$.
\end{lemma}

\begin{rem}
Since $C_G(\bT_1) \leq C_G(\bS_1)$, we have $C_G(\bT_1) = C_{L_1}(\bT_1)$.  Note that 
$$N^k(N_G(\bT_1),S_1)/C_{L_1}(\bT_1) \cong (N_{N_G(\bT_1)}(\bS_1)/C_{L_1}(\bT_1))^{\Fr}.$$
\end{rem}

\begin{proof}
 Suppose $n \in N^k(N_G(\bT_1),S_1)$.
 
 The map $\Int(g_1^{-1}) \colon N_G(\bT_1) \rightarrow N_G(\bA)$ is an isomorphism.  Since 
  $\lsup{n}S_1 = S_1$ and $\lsup{\Fr(n)}S_1 = S_1$, the images of $g_1^{-1}n  g_1$  and $g_1^{-1} \Fr(n) g_1$  in $W$ preserve $A_{\theta_1}$ and so belong to $N_W(W_{\theta_1})$.

 Set $m = \lsup{g_1^{-1}}n \in N_G(A)$.  Let $w$ denote the image of $m$ in $N_{W}(W_{\theta_1})$.   Since $(w_1\inv \circ \Fr) (\theta_1) = \theta_1$, we have $(w_1\inv \circ \Fr )(W_{\theta_1}) = W_{\theta_1}$ and $(w_1\inv \circ \Fr )(N_W(W_{\theta_1})) = N_W(W_{\theta_1})$.  Thus,  $w_1^{-1} \Fr(w)^{-1}  w_1 w \in N_W(W_{\theta_1})$.   Note that $w_1^{-1} \Fr(w)^{-1}  w_1 w$ is the image of  
 $$ (g_1^{-1} \Fr(g_1)) \Fr(m^{-1}) (\Fr(g_1)^{-1} g_1) m (g_1^{-1} g_1) = g_1^{-1}( \Fr(n)^{-1} n) g_1 \in N_G(\bA).$$
 Since $\Int(n) \colon \bS_1 \rightarrow \bS_1$ and $\Int(\Fr(n)) \colon \bS_1 \rightarrow \bS_1$ are $k$-isomorphisms we have $\lsup{\Fr(n)} s = \lsup{n}s$ 
 for all $s \in S_1^{\Fr}$.  Thus, $\Fr(n)^{-1} n \in N_{L_1}(\bT_1)$.
 Hence, the image of $g_1^{-1} (\Fr(n)^{-1} n) g_1$ in $N_W(W_{\theta_1}) \leq W$ belongs to $W_{\theta_1}$.
 Thus $w_1^{-1} \Fr(w)^{-1}  w_1 w \in W_{\theta_1}$.
 
Consequently, we can define a map
$$\eta \colon N^k(N_G(\bT_1), S_1) \longrightarrow  W_{w_1 \circ \Fr, {\theta_1}}$$
by letting $\eta(n)$ be the image of $g_1 \inv n g_1$ in $N_W(W_{\theta_1})$ for $n \in N^k(N_G(\bT_1), S_1)$.   Note that
$$\lsup{g_1\inv} \Fr(n) = g_1\inv \Fr(g_1) \Fr(g_1\inv n g_1) \Fr(g_1)\inv g_1,$$
 that is, $\eta(\Fr(n)) = (w_1\inv \circ \Fr) (\eta(n))$ for all $n \in N^k(N_G(\bT_1), S_1)$.

 We now show that $\eta$ is surjective.
Suppose $w \in W_{w_1 \circ \Fr, {\theta_1}}$. Choose a representative $m \in N_G(\bA)$ for $w$ and let $n = \lsup{g_1}m$.  Note that
$$\lsup{n}S_1 = \lsup{n g_1}A_{\theta_1} = \lsup{g_1 m g_1^{-1} g_1} A_{\theta_1} = \lsup{g_1 m}A_{\theta_1} =  \lsup{g_1}A_{\theta_1} = S_1.$$
Let $n_1 = \Fr(g_1 )^{-1} g_1 \in N_G(\bA)$, this is a lift of  $w_1$.  Let $p = n_1^{-1} \Fr(m)^{-1} n_1 m$.  By hypothesis, the image of $p$ in $N_W(W_{\theta_1}) \leq W$ belongs to $W_{\theta_1}$.  
Fix $s \in S_1^{\Fr}$, and set $a = \lsup{g_1^{-1}} s \in A_{\theta_1}$.   We have 
$$\Fr(\lsup{n}s) = \lsup{\Fr(n)}s = \lsup{\Fr(g_1) \Fr(m) \Fr(g_1^{-1})}(\lsup{ g_1}a )= \lsup{g_1}(\lsup{ n_1^{-1} \Fr(m) n_1}( a)) =  \lsup{g_1}( \lsup{m p\inv} a) = \lsup{g_1}( \lsup{m} a) = \lsup{n}s.$$
Hence $n \in N^k(N_G(\bT_1),S_1)$ and $\eta(n) = w$.

The injectivity of $\eta$ follows from the fact that the isomorphism  $\Int(g_1 ^{-1}) \colon N_G(\bT_1) \rightarrow N_G(\bA)$ induces an isomorphism $N_G(\bT_1)/ C_G(\bT_1) \rightarrow W =  N_G(\bA)/C_G(\bA)$.
\end{proof}

Note that $C_G(\bT_1) = C_{L_1}(\bT_1)$ is a normal subgroup of $N^k(N_G(\bT_1), S_1)$ and is contained in $N_{L_1}(\bT_1)$.   Since the image of $g_1\inv N_{L_1}(\bT_1) g_1$ in $W$ is $W_{\theta_1}$, Lemma~\ref{lem:firststeps} and Lemma~\ref{lem:quotient} show:

\begin{cor}  \label{cor:wherephidefined}
There is a natural $\Fr$-equivariant (group) isomorphism 
$$\varphi \colon N^k(G, S_1)/L_1 \longrightarrow W_{w_1 \circ \Fr, \theta_1}/W_{\theta_1}$$
where $\Fr$ acts on $W_{w_1 \circ \Fr, \theta_1}/W_{\theta_1}$ via $w_1 \inv \circ \Fr$. \qed
\end{cor}

\subsubsection{Classifying,  up to $G^{\Fr}$-conjugation,  $k$-embeddings of $\bS_1$ into $\bG$ with image $\bS_1$}

\begin{defn}
Suppose $\bS$ is an unramified torus in $\bG$.  Suppose $f,h \colon \bS \rightarrow \bG$ are two $k$-embeddings of $\bS$ into $\bG$.  We say \emph{$f$ is $G^\Fr$-conjugate to $h$} provided that there exists $x \in G^\Fr$ such that $\Int(x) \circ f = h$.
\end{defn}

We want to understand the set of $G^\Fr$-conjugacy classes of $k$-embeddings of $\bS_1$ into $\bG$ having image $\bS_1$.

\begin{lemma}  The set of $G^\Fr$-conjugacy classes of $k$-embeddings of $\bS_1$ into $\bG$ with image $\bS_1$ is parameterized by
\[
(N^k(G,S_1)/L_1)/(N_{G^\Fr}(S_1) L_1 / L_1 ) \cong N^k(G,S_1)/N_{G^\Fr}(S_1) L_1.
\]
\end{lemma}

\begin{proof}
Suppose $f,h \colon \bS_1 \rightarrow \bG$ are $k$-embeddings with images $\bS_1$.  We have $f$ is $G^\Fr$-conjugate to $h$ by $x \in G^{\Fr}$ if and only if $x \in N^k(G^\Fr,S_1) = N_{G^{\Fr}}(\bS_1) = (N_{G}(\bS_1))^\Fr$.
\end{proof}

Suppose $g \in N_{G^\Fr}(\bS_1)$.  Thanks to Lemma~\ref{lem:rationaln} there exists $\ell \in L_1^\Fr$ such that $n = g \ell \in N_{G^\Fr}(\bT_1) \cap N_{G^\Fr}(\bS_1) = N_{N_{G^\Fr}(\bT_1)}(\bS_1)$.  Hence, in the notation of Lemma~\ref{lem:firststeps}, $\xi(gL_1) = n N_{L_1}(\bT_1)$ and, as groups,
\[
\begin{split}
\xi(N_{G^\Fr}(\bS_1) L_1/ L_1) &= N_{N_{G^\Fr}(\bT_1)}(\bS_1) N_{L_1}(\bT_1) / N_{L_1}(\bT_1)\\
& \cong (N_{N_{G^\Fr}(\bT_1)}(\bS_1) C_{L_1}(\bT_1)/C_{L_1}(\bT_1))/(N_{L_1}(\bT_1)/C_{L_1}(\bT_1)).
\end{split}
\]

\begin{rem}
We also have, as groups,
\[
\xi(N_{G^\Fr}(\bS_1) L_1/ L_1)
\cong
N_{N_{G^\Fr}(\bT_1)}(\bS_1)/ N_{L_1^{\Fr}}(\bT_1).
\]
\end{rem}

Recall that $\bS_1$ corresponds to $(F_1,\theta_1,w_1) \in \tilde{I}^{\ee}$ and $\bS_1 = {\lsup{g_1}\bA_{\theta_1}}$ for $g_1 \in G_{F_1}$ with $\Fr(g_1)\inv g_1$ having image $w_1$ in $W$.

\begin{defn}
We let $W(F_1)$ denote the image of the stabilizer (not fixator)  in ${N_G(\bA)}$ of  $A(\AA(A)^\Fr, F_1)$. 
We set  
$$W(F_1,\theta_1,w_1) := W(F_1) \cap N_W(W_{\theta_1}) \cap W_{w_1 \circ \Fr, \emptyset} .$$
\end{defn}

\begin{rem}
The group $W(F_1,\theta_1,w_1) $ is a subgroup of $ W_{w_1 \circ \Fr, \theta_1}$.
\end{rem}

\begin{example}
If $w_1$ is the identity element of $W$, then $F_1$ is an alcove.
If $\Fr$ acts trivially and $w_1$ is the identity, then  $W(F_1,\theta_1,w_1) = W_{w_1 \circ \Fr, \theta_1} = N_W(W_{\theta_1}).$
\end{example}

\begin{lemma}
Suppose $\eta$ is the $\Fr$-equivariant group isomorphism of Lemma~\ref{lem:quotient}.   We have
$$\eta(N_{N_{G^\Fr}(\bT_1)}(\bS_1) C_{L_1}(\bT_1)/C_{L_1}(\bT_1)) \leq W(F_1,\theta_1,w_1)$$
with equality if $\bG$ is $K$-split or $\bG$ is simply connected.
\end{lemma}

\begin{proof}    
Suppose $n \in N_{N_{G^\Fr}(\bT_1)}(\bS_1)$.   We need to show that $w := \eta(n)$ is an element of $W(F_1) \cap N_W(W_{\theta_1}) \cap W_{w_1 \circ \Fr, \emptyset} $.   Since   $n \in G^{\Fr}$ normalizes $T_1$, it stabilizes $\BB(T_1)^{\Fr} = A(\AA(A)^\Fr,{F_1})$  (see~\cite[Lemma~2.2.1]{debacker:unramified}).  Since $g_1$ stabilizes $F_1$ and belongs to $\lsub{F_1}\bM(K)$, we conclude that $g_1$ stabilizes $A(\AA(A)^\Fr,{F_1})$ as well.  Thus,
 $\eta(n)$, the image of $g_1\inv n g_1$ in $W$, belongs to $W(F_1)$.  Let $n_1 = \Fr(g_1 )^{-1} g_1 \in N_G(\bA)$, this is a lift of  $w_1$. Since $w_1\inv \Fr(w)\inv w_1 w$ is the image in $W$ of 
$n_1\inv \Fr(g_1\inv n g_1)\inv n_1 (g_1\inv n g_1) = 1$ (recall that $\Fr(n)\inv n = 1$), we conclude that $\eta (n) \in W_{w_1 \circ \Fr, \emptyset}$.   Since $n$  normalizes $\bL_1 = C_\bG(\bS_1)$, we conclude that $\eta(n) \in N_W(W_{\theta_1})$.

We now show that if $\bG$ is $K$-split or simply connected, then $\eta$ is surjective.   Suppose $w \in W(F_1,\theta_1,w_1)$.   Choose $m \in {N_G(\bA)}$ lifting $w$.   Since $m$ stabilizes  $A(\AA(A)^\Fr, F_1)$ and $g_1 \in \lsub{F_1}{M}_{F_1}$ also stabilizes  $A(\AA(A)^\Fr, F_1)$, we have that $\lsup{g_1}m$ stabilizes $A(\AA(A)^{\Fr}, F_1)$.   Thus, for all $y \in A(\AA(A)^\Fr, F_1)$ we have $\lsup{g_1}m \cdot y = \Fr((\lsup{g_1}m) \cdot y)  = \Fr(\lsup{g_1}m) \cdot y$.  This implies 
that  $\Fr(\lsup{g_1}m)\inv  \cdot {\lsup{g_1}m}$ fixes $y$.    Since $w \in W_{w_1 \circ \Fr, \emptyset} $, we have $\Fr(\lsup{g_1}m)\inv  \cdot {\lsup{g_1}m} = g_1 (n_1\inv \Fr(m)\inv n_1 m) g_1\inv \in  C_G(\bT_1)$.
Set $\EbTtorus := C_{\bG}(\bT_1)$.
 If $\bG$ is simply connected, then the set of points in $\EbTtorusrat$ that fix $y$ belong to $(\EbTtorusrat \cap G_{F_1})$.  If $\bG$ is $K$-split, then $\EbTtorus = \bT_1$, and since $\bT_1$ is a maximal $K$-split torus in $\bG$, the set of points in $T_1$ that fix $y$ is equal to $T_1 \cap G_{F_1}$.  In either case, we have  $\Fr(\lsup{g_1}m)\inv  \cdot {\lsup{g_1}m}  \in \EbTtorusrat \cap G_F$.   By Lang-Steinberg, we can choose $t \in \EbTtorusrat \cap G_F$ such that 
$\Fr(\lsup{g_1}m)\inv  \cdot {\lsup{g_1}m}  = \Fr(t)t\inv$.  Thus, if $a = g_1 \inv t g_1 \in C_G(\bA)$, then $\lsup{g_1}(ma) \in G^{\Fr}$.  Note that $\eta(\lsup{g_1}(ma)) =w$.
\end{proof}

The following Corollary summarizes the results of Section~\ref{sec:$k$-embeddings}.  Recall from Corollary~\ref{cor:wherephidefined} that $\varphi \colon N^k(G, S_1)/L_1 \rightarrow W_{w_1 \circ \Fr, \theta_1}/W_{\theta_1} $ is a $\Fr$-equivariant (group) isomorphism.

\begin{cor}  
The set of $G^{\Fr}$-conjugacy classes  of $k$-embeddings of $\bS_1$ into $\bG$ with image $\bS_1$ is in natural bijective correspondence with 
$$\overline{W}(F_1, \theta_1, w_1) := W_{w_1 \circ \Fr, {\theta_1}}/ \varphi( N_{G^\Fr}(\bS_1)L_1/L_1) W_{\theta_1}.$$
If $\bG$ is $K$-split or simply connected, this is
$$W_{w_1 \circ \Fr, {\theta_1}} /  W(F_1,\theta_1,w_1) W_{\theta_1},$$
and, in general, we have  $\varphi( N_{G^\Fr}(\bS_1)L_1/L_1) \leq  W(F_1,\theta_1,w_1)$. \qed
\end{cor}

\subsection{Examples}
We now illustrate the results of this section by looking at three examples: $\Sp_4$, $\Gtwo_2$, and unramified $\SU_3$.

\begin{example}
We consider $\Sp_4$ and adopt the notation of Example~\ref{ex:initialsp4}. 
In Figure \ref{fig:Sp4$k$-embeddings}, we have added a number and a letter to each datum of $\Sp_4$.  The number  records the number of $k$-embeddings, up to $G^{\Fr}$-conjugacy, of the unramified torus corresponding to the datum into itself.   If two data have the same letter, then their corresponding unramified tori are stably conjugate.
\begin{figure}[ht]
\centering
\begin{tikzpicture}
\draw (0,0) node[anchor=north]{\shortstack{$(\emptyset, c):1:a$\\$(\emptyset, c^2):1:b$} \hphantom{hellodollythreefour}}
  -- (5,0) node[anchor=north]{$\hphantom{hellodollythree}$\shortstack{ $(\emptyset, c^2):2:b$ \\ $(\{\alpha\}, c^2):1:c$}}
  -- (5,5) node[anchor=west]{\shortstack{$(\emptyset,c):1:a$\\$(\emptyset,c^2):1:b$} $ \vphantom{P_{P_{P_{P_{P}}}}}$}
  -- cycle;
  \draw (2.5,0) node[anchor=north]{\shortstack{$(\emptyset, w_\beta):1:d$ \\ $(\{\beta + 2 \alpha \}, w_\beta):1:e$} };
  \draw (-1.5,2.5) node[anchor=west]{\shortstack{$(\emptyset, w_\alpha):1:f$ \\$(\{\alpha + \beta\},w_\alpha):1:c$} $\hphantom{hellodollythree}$};
   \draw (5,2.5) node[anchor=west]{\shortstack{$(\emptyset, w_{2 \alpha + \beta}):1:d$ \\ $(\{\beta\}, w_{\beta + 2 \alpha}):1:e$}};
   \draw (2.2,1.5) node[anchor=west]{\shortstack{$(\emptyset, 1):1:g$ \\ $(\{\alpha\}, 1):1:h$\\ $(\{\beta\},1):1:i$\\$(\Delta, 1):1:j$}};
\end{tikzpicture}
\caption{A parameterization of the $k$-embeddings of unramified tori for $\Sp_4$ \label{fig:Sp4$k$-embeddings}}
\end{figure}
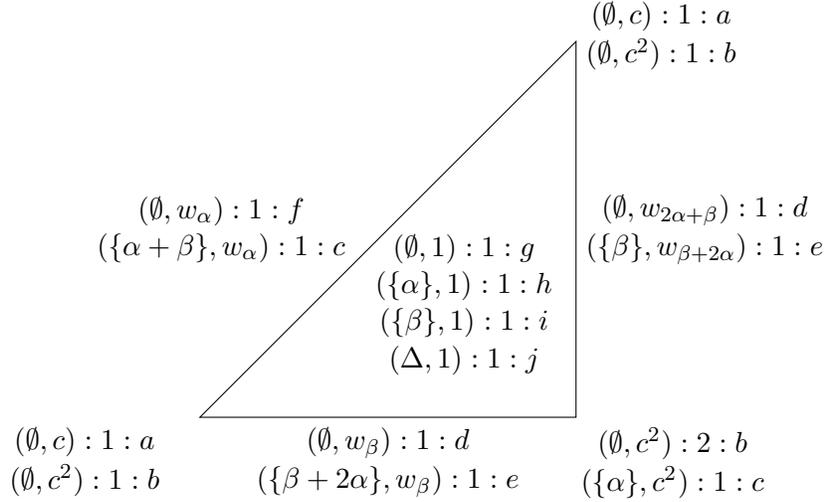
So, for example, the number of $k$-embeddings of an unramified torus corresponding to the datum $(\{\alpha + \beta\},w_\alpha)$ is two, up to $G^{\Fr}$-conjugacy.

\end{example}

\begin{example}  
In Figure~\ref{fig:g2$k$-embeddings} we have enumerated the data that classifies the rational conjugacy classes of unramified tori for the group $\Gtwo_2$.  The tick marks on the horizontal edge and the hypotenuse indicate that those two facets are equivalent.

\begin{figure}[ht]
\centering
\begin{tikzpicture}

\draw (0,0) 
  -- (4.04,7) 
  -- (0,7) 
  -- cycle;
  
  \draw(5.4,6.8) node[anchor=south]{\shortstack{ $(\emptyset,   c^2):2:b$}};
  
  \draw (1.5,0) node[anchor=north]{\shortstack{$(\emptyset, c ):1:a$\\$(\emptyset, c^2):1:b$\\$(\emptyset, c^3):1:c$} \hphantom{hellodollythreefour}};

  \draw (-2.2,6.2) node[anchor=south]{\shortstack{$(\emptyset,c^3):3:c$\\$(\{ \beta \},c^3):1:d$\\$( \{ \alpha + \beta \},c^3):1:e$} };

  \draw (2.1,3.2) node[anchor=west]{\shortstack{$(\emptyset, w_\beta):1:g$ \\ $(\{\beta + 2 \alpha \}, w_\beta):1:e$} };
  
  \draw (-4.2,3.2) node[anchor=west]{\shortstack{$(\emptyset, w_\alpha):1:f$ \\$(\{3\alpha + 2\beta\},w_\alpha):1:d$} $\hphantom{hellodollythree}$};

  \draw (2,6.8)--(2,7.2);
  
  \draw (2.18,3.41)--(1.82,3.59);

   \draw (.15,5.5) node[anchor=west]{\shortstack{$(\emptyset, 1):1:h$ \\ $(\{\alpha\}, 1):1:i$\\ $(\{\beta\},1):1:j$\\$(\Delta, 1):1:k$}};
\end{tikzpicture}
\caption{A parameterization of the $k$-embeddings of unramified tori for $\Gtwo_2$ \label{fig:g2$k$-embeddings}}
\end{figure}

We have chosen a set of simple roots $\alpha$ and $\beta$ for $\Gtwo_2$ with $\alpha$ short and $\beta$ long so that in Figure~\ref{fig:g2$k$-embeddings} the hypotenuse lies on a hyperplane defined by an affine root with gradient $\beta$ and the vertical edge lies on a hyperplane defined by an affine root with gradient $\alpha$.   We let $w_\alpha$ denote the simple reflection corresponding to $\alpha$, and  $w_\beta$ denotes the simple reflection corresponding to $\beta$.  We let $c$ denote the Coxeter element $w_\alpha w_\beta$.

As for $\Sp_4$, each label has three parts: a datum from $\dot{I}(F)$ where $F$ is the facet adjacent to the label; a number that records the number of $k$-embeddings, up to $G^{\Fr}$-conjugacy, of the unramified torus corresponding to the datum into itself; and a letter indicating the stable conjugacy class of the unramified torus corresponding to the datum.  

The centralizer of the unramified torus corresponding to the pair $(\{3\alpha + 2\beta\}, w_\alpha)$ is unramified  $\tilde{U}(1,1)$ while the centralizer of the unramified torus corresponding to the pair $(\{\beta\},c^3)$ is unramified $\tilde{U}(2)$.   The ornamental tilde indicates that the group $\tilde{U}(2)$ is of type $A_1$ for a long root.  The centralizer of the unramified torus corresponding to the pair $(\{\beta + 2\alpha \}, w_\beta)$ is unramified  ${U}(1,1)$ while the centralizer of the unramified torus corresponding to the pair $(\{\alpha + \beta\},c^3)$ is unramified ${U}(2)$. 
  The four unramified tori with labels of the form $(\theta, 1)$ are the $k$-split components of the centers of the four (up to conjugacy) distinct $k$-subgroups of $\Gtwo_2$ that occur as a Levi factor for a parabolic $k$-subgroup of $\Gtwo_2$.

\end{example}

\begin{example}  
In Figure~\ref{fig:su3$k$-embeddings} we have enumerated the data that classifies the rational conjugacy classes of unramified tori for the  group of $k$-rational points of unramified $\SU(3)$.  Thinking of this group as the $\Fr$-fixed points of $\SL_3(K)$, the dotted equilateral triangle is a $\Fr$-stable alcove of $\SL_3(K)$.

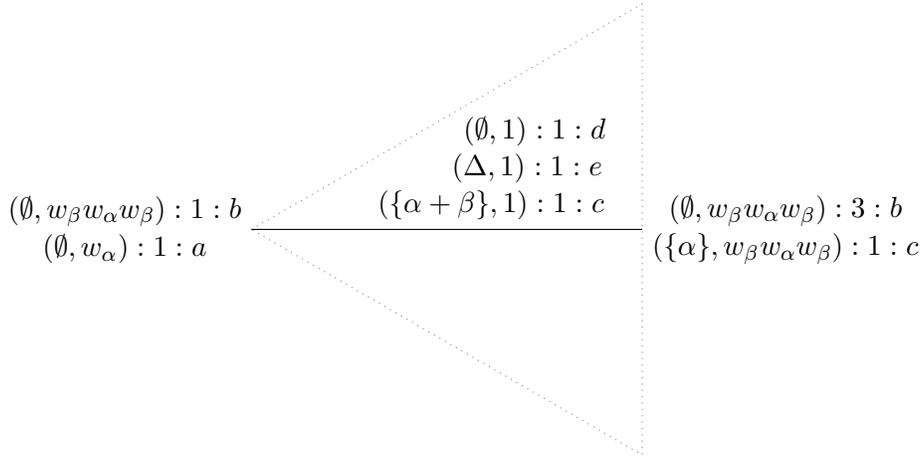
\begin{figure}[ht]
\centering
\begin{tikzpicture}

\draw[dotted, gray] (0,0) 
  -- (5.196,3) 
  -- (5.196,-3) 
  -- cycle;
 \draw (0,0) -- (5.196,0); 

  \draw(0,0) node[anchor=east]{\shortstack{$(\emptyset,w_\beta w_\alpha w_\beta):1:b$  \\$(\emptyset, w_\alpha) :1:a$}};
  
   \draw(5.196,0) node[anchor=west]{\shortstack{ $(\emptyset,w_\beta w_\alpha w_\beta):3:b$ \\ $(\{\alpha\} , w_\beta w_\alpha w_\beta) :1
  :c$}};
  
   \draw(2.8,0) node[anchor=south]{\shortstack{
   $\hspace{5em}(\emptyset,1):1:d$  \\ $\hspace{4.5em}(\Delta ,1) :1
  :e$ \\$\hspace{2em}(\{\alpha + \beta \},1):1:c$ }};

\end{tikzpicture}
\caption{A parameterization of the $k$-embeddings of unramified tori for unramified $\SU(3)$ \label{fig:su3$k$-embeddings}}
\end{figure}

We have chosen a set of simple roots $\alpha$ and $\beta$ for $\SL_3$   so that $\Fr(\alpha) = \beta$ and the hyperspecial vertex (of $\SU(3)$) pictured in Figure~\ref{fig:su3$k$-embeddings}  lies on the hyperplanes defined by  affine roots with gradients $\alpha$ and $\beta$ while the other vertex lies on a hyperplane defined by an affine root with gradient $\alpha + \beta$. 

As for $\Sp_4$ and $\Gtwo_2$, each label has three parts: a datum from $\dot{I}(F)$ where $F$ is the facet adjacent to the label; a number that records the number of $k$-embeddings, up to $G^{\Fr}$-conjugacy, of the unramified torus corresponding to the datum into itself; and a letter indicating the stable conjugacy class of the unramified torus corresponding to the datum.  

The centralizer of the unramified torus corresponding to the pair $(\{\alpha\}, w_\beta w_\alpha w_\beta)$ is unramified  ${U}(2)$ while the centralizer of the unramified torus corresponding to the pair $(\{\alpha + \beta\},1)$ is unramified ${U}(1,1)$.  
  The two unramified tori with labels of the form $(\theta, 1)$ where $\theta \in \{ \emptyset, \Delta \}$ are the maximally $K$-split components of the centers of the two (up to conjugacy) distinct $k$-subgroups of $\SU(3)$  that occur as a Levi factor for a parabolic $k$-subgroup of $\SU(3)$.

\end{example}

\section{Unramified twisted generalized Levis}
\label{sec:twistedgeneralized}

In this section we provide a parameterization of $\tilde{\UU}$, the set of $G^\Fr$-conjugacy classes of unramified twisted generalized Levi subgroups of $\bG$.

\begin{defn}  \label{defn:utgls}
A connected reductive $k$-subgroup $\bL$ of $\bG$  will be called an \emph{unramified twisted generalized Levi subgroup} of $\bG$ provided that
 it  contains  the centralizer of a maximal unramified torus of $\bG$.
\end{defn}

\begin{example}   Every unramified twisted Levi  is an unramified twisted generalized Levi.
\end{example}

\subsection{Closed and quasi-closed root subsystems}

Suppose $\bS$ is a maximal $K$-torus in $\bG$.
Since $\bG$ is $K$-quasi-split, we have that $\EbStorus := C_{\bG}(\bS)$ is a maximal $k$-torus in $\bG$.  If $\mu \subset \Phi(\bG,\EbStorus)$, then define $\bG_\mu := \langle \EbStorus, \bU_\alpha \, | \, \alpha \in \mu \rangle$.

\begin{defn}
A subset $\mu$ of $\Phi(\bG,\EbStorus)$ is said to be a \emph{quasi-closed} provided that if $\beta \in \Phi(\bG,\EbStorus)$ and $\bU_\beta \subset \bG_{\mu}$, then $\beta \in \mu$.   A subset $\Upsilon$ of  $\Phi(\bG,\EbStorus)$ is called  \emph{closed} provided that for all $\alpha, \beta \in \Upsilon$ we have $\alpha + \beta \in \Phi(\bG,\EbStorus)$ if and only if $\alpha + \beta \in \Upsilon$.  
\end{defn}

\begin{rem} As discussed in~\cite[Section 3]{borel-tits:groupes} (see also, ~\cite[XXIII, Corollaire  6.6]{schemasengroupesIII}), every closed subset of $\Phi(\bG,\EbStorus)$ is quasi-closed and the converse is true if the characteristic of $k$ is not $2$ or $3$.  See Footnote 17 in~\cite[XXIII]{schemasengroupesIII} or the \emph{Remarque} following \cite[Proposition~2.5]{borel-tits:groupes} for a specific list of cases where the converse fails.
\end{rem}

A subset $\rho$ of $\Phi(\bG,\bA)$ is said to be a \emph{quasi-closed root subsystem} provided that the set of roots
$$\{ \alpha \in \Phi(\bG,\Ebtorus) \colon \res_{\bA} \alpha \in \rho \}$$
is quasi-closed in $\Phi(\bG,\Ebtorus)$.   Here $\Ebtorus$ is the maximal $k$-torus $C_{\bG}(\bA)$.  Similarly, a subset $\delta$ of  $\Phi(\bG,\bA)$ is called a \emph{closed root subsystem} provided that 
$$\{ \alpha \in \Phi(\bG,\Ebtorus) \colon \res_{\bA} \alpha \in \delta \}$$
 is closed in $\Phi(\bG,\Ebtorus)$.

We let $\tilde{\Theta} = \tilde{\Theta}(\bG,\bA)$ denote the set of bases of quasi-closed root  subsystems  of $\Phi(\bG,\bA)$.

\begin{example} If $\Xi \subset \Delta \subset \Phi(\bG,\bA)$, then $\Xi \in \tilde{\Theta}$, and, in general, $\Theta \subset \tilde{\Theta}$.   More exotically, the long roots in the root system of $\Sp_4$ form a closed root subsystem and so, in the notation of Example~\ref{ex:initialsp4}, $\{\beta, \beta + 2 \alpha \} \in \tilde{\Theta}$.  Note, however, that the short roots in the root system of $\Sp_4$ do \textbf{not} form a closed root subsystem, but they do form a quasi-closed root subsystem when the characteristic of $k$ is $2$.  Thus, in the notation of Example~\ref{ex:initialsp4}, we have that  $\{\alpha, \beta +  \alpha \} \in \tilde{\Theta}$ whenever the characteristic of $k$ is $2$.
\end{example}

\begin{rem}  
Both $W$ and $\Fr$ act on $\tilde{\Theta}$.
\end{rem}

Suppose $\Xi \in \tilde{\Theta}$. Let $\Phi_\Xi$ denote the $\mathbb{Z}$-span of $\Xi$ in $\Phi$ and let $W_{\Xi} \leq W$ denote the associated  Weyl group.   Let $\bM_{\Xi}$ be the connected reductive $K$-group in $\bG$ generated by $\Ebtorus$ and the root groups $\bU_\alpha$ for $\alpha \in \Phi(\bG,\Ebtorus)$ with $\res_{\bA} \alpha \in \Phi_{\Xi}$.   Note that $W_\Xi$ is the Weyl group $W(M_\Xi,\bA)$.

\subsection{A result about parahoric subgroups and unramified twisted generalized Levi subgroups}

Suppose $\bH$ is an unramified twisted generalized Levi subgroup of $\bG$, that is, $\bH$ is a 
connected reductive $k$-subgroup of $\bG$ that contains the centralizer of a maximal unramified torus of $\bG$.   Since  $\bH$ contains a maximal unramified torus of $\bG$, the building of $\BB(H)$ embeds into the building of $\BB(G)$. There is not a canonical embedding, but all such embeddings have the same image.

\begin{lemma}
If $x \in \BB(H)$, then $H_x \leq G_x \cap H$.
\end{lemma}

\begin{proof}
Let $\bS$ be a maximal $K$-split torus of $\bH$ such that $x \in \AA(S) \subset \BB(H)$.  Since the maximal $K$-split tori in $\bH$ form a single $H$-conjugacy class and since $\bH$  contains the centralizer of a maximal unramified torus of $\bG$, we conclude that  $\bS' = C_\bG(\bS)$ is a subgroup of  $\bH$.   

Note that $S'_0$, the parahoric subgroup of $S'$, is equal to $S_0 S'_{0^+}$ where $S_0$ is the parahoric subgroup of $S$ and $S'_{0^+}$ is the pro-unipotent radical of $S'_0$. 

The parahoric subgroup $H_x$ is generated by $S'_0$ and the groups $U_\psi$ for $\psi \in \Psi(\bH, \bS, \nu)$ with $\psi(x) \geq 0$.  Suppose $\psi \in \Psi(\bH, \bS, \nu)$ and $\psi(x) \geq 0$.   If $u \in U_\psi$, then $u$ is unipotent in $G$ and fixes $x$, hence $u \in G_x$. Since $G_x$ contains $S'_0$, we conclude that $H_x \leq G_x \cap H$.
\end{proof}

\subsection{Some indexing sets}

For a $G^\Fr$-facet $F \subset \AA(A)^\Fr$, set
$$ \calI(F) = \{(\Xi, w) \, | \, \Xi \in \tilde{\Theta}, w \in W_F, \Fr(\Phi_{\Xi}) = w \Phi_{\Xi} \}.$$
For $(\Xi, w), (\Xi', w') \in \calI(F)$, we write $(\Xi, w) \stackrel{F}{\sim} (\Xi', w')$ provided that there exists $m \in W_F$ such that 
\begin{itemize}
    \item $m \Phi_{\Xi} = \Phi_{\Xi'}$
    \item $\Fr(m) w m\inv \in w'(W_F \cap W_{\Xi'})$
\end{itemize}

\begin{lemma}
The relation $\stackrel{F}{\sim}$ is an equivalence relation on $\calI(F)$. \qed
\end{lemma}

We will say that $(\Xi, w) \in \calI(F)$ is \emph{$F$-elliptic} provided that  for all $(\Xi',w') \in \calI(F)$ with $(\Xi, w) \stackrel{F}{\sim} (\Xi',w')$ we have that $w'$ does not belong to a $\Fr$-stable proper parabolic subgroup of $W_F$.
We set
$$\calI^{\ee}(F) := \{(\Xi, w)  \in \calI(F) \, | \,  \text{$(\Xi,w)$ is $F$-elliptic} \}.$$

Define
$$\calI = \{ (F, \Xi, w) \, | \, \text{ $F$ is a $G^\Fr$-facet in $\AA(A)^\Fr$ and } (\Xi,w) \in \calI(F) \}.$$
For $(F, \Xi, w), (F', \Xi', w') \in \calI$ we write 
 $(F, \Xi,w) \approx (F',\Xi',w')$ provided that there exists an element $m \in \affFrW$ for which $A(\AA(A)^\Fr, F') = A(\AA(A)^{\Fr},mF)$ and with the identifications of  $\sfG_{F'} \overset{i}{=} \sfG_{mF}$  and $\X^*(\sfA_{F'}) \overset{i}{=} \X^*(\sfA_{mF}) = \X^*(\bA)$ thus induced we have that $(\Xi',w') \stackrel{F'}{\sim} (m\Xi, \lsup{m}w)$ in $\calI(F') \overset{i}{=} \calI(mF)$.

\begin{lemma}
The relation $\approx$ is an equivalence relation on $\calI$. \qed
\end{lemma}

\begin{defn}
We will say that $(F, \Xi, w) \in \calI$ is \emph{elliptic} provided that $(\Xi,w) \in \calI^{\ee}(F)$.
We set
$$\calI^{\ee} := \{(F, \Xi, w)  \in \calI \, | \,  (\Xi,w) \in \calI^{\ee}(F) \}.$$
\end{defn}

\begin{rem}  If
$(F,\Xi,w)\in \calI^{\ee}$, then $(\emptyset,w)\in \dot{I}^{\ee}(F) \subset \calI^{\ee}(F)$.
\end{rem}

\subsection{Parameterizing \texorpdfstring{$\tilde{\UU}$}{U}}

Suppose $\mu = (F, \Xi, w) \in \calI$. Choose $g \in G_F$ such that $\Fr(g)\inv g \in N_{G_F}(\bA)$ has image $w$ in $W_F$.  Let $\bS = \lsup{g}\bA$.  Since $\Fr(\bS) = \bS$ and   $\Fr(g \Phi_\Xi) = \Fr(g) \Fr(\Phi_\Xi) = \Fr(g) w \Phi_\Xi = \Fr(g) \Fr(g)\inv g \Phi_\Xi = g \Phi_\Xi$, the connected reductive $K$-group $\bL_{\mu,g} := \lsup{g}\bM_{\Xi}$ is a  $k$-group.  Since $\bL_{\mu,g}$ contains  $C_\bG(\bS)$, we conclude that $\bL_{\mu,g}$  is an unramified twisted generalized Levi.

\begin{lemma} \label{lem:okdefn}
The $G_F^\Fr$-conjugacy class of $\bL_{\mu,g}$ depends only on $\mu$.
\end{lemma}

\begin{proof}
Fix  $\hat{g} \in G_F$ such that $\Fr(\hat{g})\inv \hat{g} \in N_{G_F}(\bA)$ has image $w$ in $W_F$ and notice that
$\hat{g}\inv \Fr(\hat{g}) \Fr(g)\inv g \in \Ebtorusrat$.  Set $\hat{S} = \lsup{\hat{g}}\Ebtorusrat$ and  note that $\Fr(g \hat{g}\inv)\inv g \hat{g}\inv \in \hat{S} \cap G_F$.   Since $\hat{S}_0 = \hat{S} \cap G_F$ is the parahoric subgroup of $\hat{S}$, from Lemma~\ref{lem:parahoriccohom} we have $\cohom^1(\Fr, \hat{S}_0)$ is trivial, and so there exists $s \in \hat{S}_0$ such that $\Fr(g \hat{g}\inv s\inv) =  g \hat{g}\inv s\inv \in G_F$.   We conclude that $s \hat{g} g\inv \in G^\Fr_F$.  Since
$$\lsup{\hat{g}}\bM_\Xi =  \lsup{s\hat{g}}\bM_\Xi = \lsup{s\hat{g}g\inv g}\bM_\Xi = \lsup{s \hat{g}g\inv} \bL_{\mu,g},$$
we conclude that $\lsup{\hat{g}}\bM_\Xi$ is $G_F^\Fr$-conjugate to $\bL_{\mu,g}$.
\end{proof}

Thanks to Lemma~\ref{lem:okdefn} the following definition makes sense.
\begin{defn} \label{defn:842}
Define  $j \colon \calI \rightarrow \tilde{\UU}$ by setting $j(\mu)$ equal to the $G^\Fr$-conjugacy class of $\bL_{\mu,g}$.
\end{defn}

\begin{lemma} \label{lem:thelemmaweneed}
Suppose $F \subset \AA(A)^\Fr$ is a $G^\Fr$-facet.
Suppose $(\Xi_i,w_i) \in \calI(F)$ and  $g_i \in G_F$ such that $\Fr(g_i)\inv g_i \in N_{G_F}(\bA)$ has image $w_i$ in $W_F$  for $i \in \{1,2\}$.  Set $\bL_i = \lsup{g_i}\bM_{\Xi_i}$.
The $G_F^\Fr$-conjugacy classes of $L_1$ and $L_2$ coincide if and only if
$(\Xi_1,w_1) \stackrel{F}{\sim} (\Xi_2,w_2)$.
\end{lemma}

\begin{proof}
``$\Leftarrow$''
  Since  $(\Xi_1,w_1) \stackrel{F}{\sim} (\Xi_2,w_2)$, there exists $n \in W_F$ such that 
\begin{itemize}
    \item $n \Phi_{\Xi_1} = \Phi_{\Xi_2}$ and
    \item $\Fr(n) w_1 n\inv \in w_2 (W_F \cap W_{\Xi_2})$.
\end{itemize}

Choose $\dot{n} \in N_{G_F}(\bA)$ such that the image of $\dot{n}$ in $W_F$ is $n$.  Choose $g_i \in G_F$ such that the image of $\Fr(g_i)\inv g_i \in N_{G_F}(\bA)$ in $W_F$ is $w_i$.    Set
$$h := (g_2\inv \Fr(g_2)) \Fr(\dot{n}) (\Fr(g_1)\inv g_1) \dot{n}\inv.$$
Since $h$ belongs to $N_{G_F}(\bA)$ and has image $w_2\inv \Fr(n) w_1 n\inv$ in $W$, we conclude that $h$ belongs to $M_{\Xi_2} \cap N_{G_F}(\bA)$.  Thus
$$\lsup{g_2}h = \Fr(g_1 \dot{n}\inv g_2\inv)\inv (g_1 \dot{n}\inv g_2\inv)$$
is an element of $G_F \cap L_2$.  From Lemma~\ref{lem:parahoriccohom} we have that $\cohom^1( \Fr, (L_2)_F)$ is trivial, so there exists $\ell \in (L_2)_F \leq  G_F$ such that $\lsup{g_2}h = \Fr(\ell) \ell \inv$.   Thus
$$g_1 \dot{n}\inv g_2\inv \ell = \Fr(g_1 \dot{n}\inv g_2\inv \ell).$$
So $g_1 \dot{n}\inv g_2\inv \ell \in G_F^\Fr$ and 
$$\lsup{g_1 \dot{n}\inv g_2\inv \ell } \bL_2 = \lsup{g_1 \dot{n}\inv}\bM_{\Xi_2} = \lsup{g_1} \bM_{\Xi_1} = \bL_1.$$

``$\Rightarrow$''
Since the $G_F^\Fr$-conjugacy classes of $L_1$ and $L_2$ coincide,  there exists $x \in G_F^\Fr$ such that $\lsup{x}L_1 = L_2$.  Without loss of generality we can replace $g_2$ by $xg_2$ and assume $\bL_1 = \bL_2$.   Set $\bL = \bL_1$.

Since $\lsup{g_1}\bA$ and $\lsup{g_2}\bA$ are maximal $K$-split $k$-tori in $\bL$ and $F \subset \BB(\lsup{g_1}A)^\Fr \cap \BB(\lsup{g_2}A)^\Fr \subset \BB(L)^\Fr$, there exists $\ell \in L_F \leq G_F \cap L$ such that $\lsup{\ell g_1}\bA = \lsup{g_2} \bA$.  Let $m$ denote the image of $g_1\inv \ell\inv g_2 \in N_{G_F}(\bA)$ in $W$.   Since $\lsup{g_1}\bA$ and $\lsup{g_2}\bA$ are $k$-tori, we have
$$\lsup{\Fr(\ell)\inv}(\lsup{g_2}A) = \lsup{\ell \inv}(\lsup{g_2}A)$$
which implies that $g_2\inv \Fr(\ell) \ell \inv g_2 \in N_{G_F}(\bA)$  
has image in $W$ belonging to $W_{\Xi_2} \cap W_F$.

Note that
\[ 
\begin{split}
\Phi_{\Xi_1} &= \Phi(\bM_{\Xi_1},\bA)
=g_1\inv \Phi(\lsup{g_1}\bM_{\Xi_1},{\lsup{g_1}\bA}) = g_1\inv \Phi(\bL,{\lsup{g_1}\bA}) =   g_1\inv \ell \inv \Phi(\bL,{\lsup{\ell g_1}\bA})\\
&=g_1\inv \ell \inv \Phi(\bL,{\lsup{g_2}\bA}) =
g_1\inv \ell \inv g_2 \Phi(\bM_{\Xi_2},\bA) =
g_1\inv \ell \inv g_2 \Phi_{\Xi_2}  \\
&= m \Phi_{\Xi_2}
\end{split}
\]
and $\Fr(m)\inv w_1 m$ is the image in $W$ of
\[ 
(\Fr(g_2\inv \ell g_1) ) (\Fr(g_1)\inv g_1) (g_1\inv \ell\inv g_2) = \Fr(g_2)\inv(\Fr(\ell) \ell\inv) g_2 = \Fr(g_2)\inv g_2 (g_2\inv  \Fr(\ell) \ell\inv g_2)
\]
which has image in $w_2(W_{\Xi_2} \cap W_F)$.   Consequently, $(\Xi_1,w_1) \stackrel{F}{\sim} (\Xi_2,w_2)$.
\end{proof}

\begin{lemma}  \label{lem:maybehelp}
Suppose $(F,\Xi,w) \in \calI^{\ee}$, $g \in G_F$ such that the image of $\Fr(g)\inv g \in N_{G_F}(\bA)$ in $W_F$ is $w$, and $\bL = \lsup{g}\bM_{\Xi}$. Then  $F$ is  a maximal $G^{\Fr}$-facet in $\BB(L)^\Fr$. 
\end{lemma}

\begin{proof} 
Let $\bS = \lsup{g}\bA$.
If $F$ is not maximal, then there exists a $G^\Fr$-facet $H$ in $\BB(L)^\Fr$ such that $F \subset \overline{H}$ and $F \neq H$.   Since we may choose $x \in G_F^\Fr$ such that $xH \subset \BB(A)^\Fr$, without loss of generality we may assume that $F$ and $H$ are $G^\Fr$-facets in $\AA(A)^\Fr \cap \BB(L)^{\Fr}$.

Since $H \subset \BB(L)^\Fr$, there is a maximally $k$-split maximal $(K,k)$-torus $\bT \leq \bL$ such that $H \subset \BB(T)^\Fr$.  Consequently, we can choose $h \in G_H \leq G_F$ such that $\bT = \lsup{h} \bA$ and the image, $w'$, of $\Fr(h)\inv h \in N_{G_H}(\bA)$ lies in $W_{H} \leq W_{F} \leq W$.  Choose a basis $\Delta_L$ for $\Phi(\bL,\bT)$ and set $\Xi' = h\inv \Delta_L$.  We have $\lsup{h}\bM_{\Xi'} = \bL$.

We now show $(\Xi,w) \stackrel{F}{\sim} (\Xi',w')$.
Since $\bT$ and $\bS$ are maximal $K$-split $k$-tori in $\bL$ and  $F \subset \BB(T)^\Fr \cap \BB(S)^\Fr$, there exists $\ell \in L_F \leq G_F$ such that $\lsup{\ell}\bS = \bT$.  
Let $m$ denote the image of $g\inv \ell\inv h \in N_{G_F}(\bA)$ in $W$.   Since $\lsup{g}\bA$ and $\lsup{h}\bA$ are $k$-tori, we have
$$\lsup{\Fr(\ell)\inv}(\lsup{h}A) = \lsup{\ell \inv}(\lsup{h}A)$$
which implies that $h\inv  \Fr(\ell) \ell\inv h \in N_{G_F}(\bA)$ 
has image in $W$ belonging to $W_{\Xi'} \cap W_F$.

Note that
\[ 
\begin{split}
\Phi_{\Xi} &= \Phi(\bM_{\Xi},\bA)
=g\inv \Phi(\lsup{g}\bM_{\Xi},{\lsup{g}\bA}) = g\inv \Phi(\bL,{\lsup{g}\bA}) =   g\inv \ell \inv \Phi(\bL,{\lsup{\ell g}\bA})\\
&=g\inv \ell \inv \Phi(\bL,{\lsup{h}\bA}) =
g\inv \ell \inv h \Phi(\bM_{\Xi'},\bA) =
g\inv \ell \inv h \Phi_{\Xi'}  \\
&= m \Phi_{\Xi'}
\end{split}
\]
and $\Fr(m)\inv w m$ is the image in $W$ of
\[ 
(\Fr(h\inv \ell g) ) (\Fr(g)\inv g) (g\inv \ell\inv h) = \Fr(h)\inv(\Fr(\ell) \ell\inv) h = (\Fr(h)\inv h) \cdot (h\inv \Fr(\ell) \ell\inv h)
\]
which has image in $w'(W_{\Xi'} \cap W_F)$.   Consequently, $(\Xi,w) \stackrel{F}{\sim} (\Xi',w')$.

  Since $w'$ belongs to $W_H$,  a $\Fr$-stable proper parabolic subgroup of $W_F$, this contradicts the assumption that $(F,\Xi,w)$ is elliptic.
\end{proof}

\begin{theorem}  \label{thm:jisbijective}
The map $j$ defined in Definition~\ref{defn:842} induces a  bijection from ${\calI}^{\ee}/\!\approx$ to $\tilde{\UU}$.
\end{theorem}

\begin{proof}
We first show that $j$ is surjective.   
Suppose $\bL$ is an unramified twisted generalized Levi subgroup of $\bG$.   Choose a $G^\Fr$-facet $F \subset \BB(L)^\Fr$ that is maximal among the set of $G^\Fr$-facets in $\BB(L)^\Fr$.   Let $\bS \leq \bL$ be a maximally $k$-split maximal $(K,k)$-torus in $\bL$ such that $F \subset \BB(S)^\Fr \subset \BB(L)^\Fr$.   Without loss of generality, we assume, after conjugating everything in sight by an element of $G^\Fr$, that $F \subset \BB(S)^\Fr \subset \AA(A)^\Fr \cap \BB(L)^\Fr$.   Choose $g \in G_F$ such that $S = \lsup{g}A$.    Let $w$ be the image of $\Fr(g)\inv g \in N_{G_F}(\bA)$ in $W_F \leq W$.
Choose a basis $\Delta_L$ for $\Phi(\bL,\bS)$ and let $\Xi = g\inv \Delta_L \in \tilde{\Theta}$.   
By construction,  $(F, \Xi, w) \in \calI$ and $j((F,\Xi,w))$ is the $G^\Fr$-conjugacy class of $L$.

To complete the proof of surjectivity, we need to show that $(F,\Xi, w)$ is elliptic.  If it is not elliptic, then there exist $(\Xi',w') \in \calI(F)$ with $(\Xi,w) \stackrel{F}{\sim} (\Xi', w')$ and a $G^{\Fr}$-facet $H$ in $\AA(A)^{\Fr}$ with $F \subset \bar{H}$ and $F \neq H$ such that $w'$ lies in $W_H$.   Since $w' \in W_H$, there exists $h \in G_{H} \subset G_F$ such that  the image of $\Fr(h)^{-1} h$ in $W_{H} \leq W_F$ is $w'$.    Since  $(\Xi,w) \stackrel{F}{\sim} (\Xi', w')$, from Lemma~\ref{lem:thelemmaweneed} we have 
$\lsup{xh}\bM_{\Xi} = \bL$ for some $x \in G_F^{\Fr}$.  Note that $\lsup{xh}\bA \leq \bL$.   Hence,  $x H = xh H \subset \AA(\lsup{xh}\bA)^\Fr \leq \BB(L)^{\Fr}$, contradicting the maximality of $F$.

We now show that if $ \mu_i = (F_i,\Xi_i, w_i)$ for $i \in \{1,2\}$ are two elements of  $\calI^{\ee}$ with $j(\mu_1) = j(\mu_2) $, then  
$\mu_1 \approx \mu_2$.

Choose $g_i \in G_{F_i}$ such that $\Fr(g_i)\inv g_i \in N_{G_F}(\bA)$ has image $w_i$ in $W_{F_i} \leq W$.  Set $\bL_i = \lsup{g_i} \bM_{\Xi_i}$ and $\bS_i = \lsup{g_i}\bA$.  Thanks to Lemma~\ref{lem:maybehelp}, we know that $F_i$ is a maximal $G^{\Fr}$-facet in $\BB(L_{i})^\Fr$. 
Since  $F_i$ is a maximal $G^{\Fr}$-facet in $\BB(L_{i})^\Fr$ and $F_i \subset \BB(S_i)^\Fr \subset \BB(L_{i})^\Fr$, the torus $\bS_i$ is a maximally $k$-split maximal $(K,k)$-torus in $\bL_{i}$.

Since $j(\mu_1) = j(\mu_2) $, there exists $y \in G^\Fr$  such that $L_{1} = \lsup{y}L_{2}$.  Since $\lsup{y}\bS_2$ and $\bS_1$ are maximally $k$-split maximal $(K,k)$-tori in $\bL_{1}$, from~\cite[Lemma~6.1]{prasad:unramified} there exists $\ell \in L_{1}^{\Fr}$ such that $\bS_1 = \lsup{\ell y}\bS_2$.  Since $F_1 \subset \AA(A)^{\Fr}$, there exists $x \in G_F^\Fr$ such that $\BB(\lsup{x}\bS_1)^\Fr \subset \AA(A)^\Fr$.   Thus, after replacing $\mu_1$ with $x \cdot \mu_1$ and $\mu_2$ with $x \ell y \cdot \mu_2$, we may assume that  $ \bL_{1} = \bL_{2}$,  $\bS_1 = \bS_2$, and  $F_1, F_2 \subset \BB(S)^\Fr \subset \BB(L)^\Fr \cap \AA(A)^\Fr$.  
Let $\bL = \bL_1$ and $\bS = \bS_1$.  Since $F_1$ and $F_2$ are maximal in $\BB(L)^\Fr$, they are maximal in $\BB(S)^\Fr$ and so $\emptyset \neq A(F_1, \AA(A)^\Fr ) = A(F_2, \AA(A)^\Fr )$.

Let $\sfS_i$ denote the image of $S \cap G_{F_i}$ in $\sfG_{F_i}$.  Since $S$ is a lift of $(F_i,\sfS_i)$, we conclude that $\sfS_1 \overset{i}{=} {\sfS_2}$ in $\sfG_{F_1} \overset{i}{=} \sfG_{F_2}$; this means
$(\emptyset,w_1) \stackrel{F}{\sim} (\emptyset,w_2)$ in $\dot{I}(F_1)\overset{i}{=} \dot{I}(F_2)$.    Let $\bar{g}_i$ denote the image of $g_i$ in $G_{F_i}$.   Let $n = \bar{g}_2\inv \bar{g}_1$ in $\sfG_{F_1} \overset{i}{=} \sfG_{F_2}$.   Note that $n \in N_{\sfG_{F_1}}(\sfA) \overset{i}{=} N_{\sfG_{F_2}}(\sfA)$, and so it has image $\dot{n} \in   W_{F_1} = W_{F_2} \leq W$.  Moreover, $\Fr(\dot{n}) w_1 \dot{n}\inv = w_2$.  Since 
$\lsup{g_2 \inv g_1}\bM_{\Xi_1} = \bM_{\Xi_2}$, we conclude that $\dot{n} \Phi_{\Xi_1} \overset{i}{=} \Phi_{\Xi_2}$ in $X^*(\sfA_{F_1}) \overset{i}{=} X^*(\sfA_{F_2}) = X^*(\bA)$.   Consequently,
$\mu_1 \approx \mu_2$.
\end{proof}

\begin{example} \label{ex:goodp} In Figure~\ref{fig:last} we provide, up to rational conjugacy, a parameterization of unramified twisted Levis that are not unramified twisted Levis for 
 the groups $\Sp_4$ (when the characteristic of $k$ is not $2$) and $\Gtwo_2$ (when the characteristic of $k$ is not $3$).

\begin{figure}[ht]
\centering
\begin{tikzpicture}
\draw (0,0)
  -- (5,0) 
  -- (5,5) 
  -- cycle;
   \draw (2,1.5) node[anchor=west]{$(\{\beta, \beta + 2 \alpha \},1)$};
    \draw (.8,1.3) node[anchor=west, rotate = 45]{$(\{\beta, \beta + 2 \alpha \},w_\alpha)$};

   \draw (11,0)
  -- (13.9,5) 
  -- (11,5) 
  -- cycle;
  \draw (10.6,4.2) node[anchor=east, rotate = 90]{$(\{\beta, \beta + 3 \alpha \}, w_\alpha)$};
    \draw (12.25,3.35) node[anchor=south, rotate = 90]{\shortstack{$(\{\alpha, 2\beta + 3 \alpha \}, 1)$\\ $(\{\beta, \beta + 3 \alpha \},1)$}};
 \draw (12.5,5.1) -- (12.5,4.9);
    \draw (12.53,2.44) -- (12.35,2.55);

\end{tikzpicture}

\caption{A parameterization of  the  rational classes of unramified twisted generalized Levis that are not unramified twisted Levis for $\Sp_4$ ($\characteristic(k) \neq 2$) and $\Gtwo_2$ ($\characteristic(k) \neq 3$). \label{fig:last}}
\end{figure}
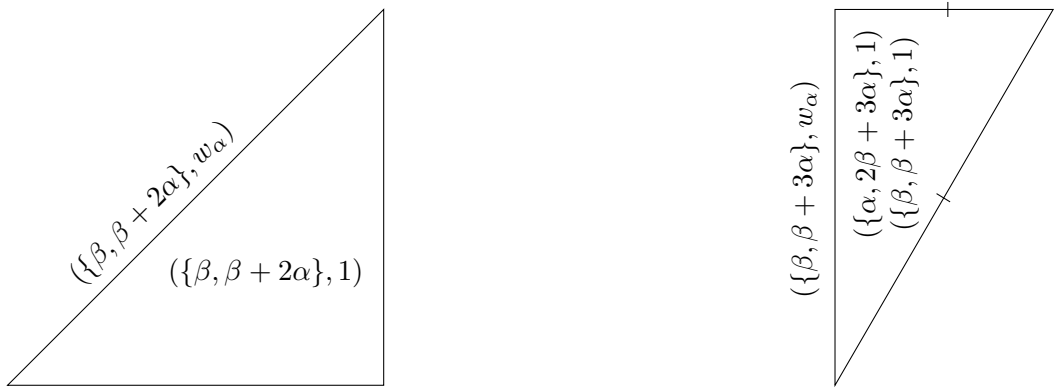

For the group $\Sp_4$ with $\characteristic(k) \neq 2$, the label $(\{\beta, \beta + 2 \alpha \},1)$ corresponds to the rational conjugacy class of unramified twisted generalized Levis that are isomorphic to $\SL_2 \times \SL_2$  and the label $(\{\beta, \beta + 2 \alpha \},w_\alpha)$ corresponds to the rational conjugacy class of unramified twisted generalized Levis that are isomorphic to $R_{E/k}(\SL_2)$ where $E$ is the unramified quadratic extension of $k$.

For the group $\Gtwo_2$ with $\characteristic(k) \neq 3$, the labels $(\{\alpha, 2\beta + 3 \alpha \}, 1)$ and $(\{\beta, \beta + 3 \alpha \},1)$ correspond to the rational conjugacy classes of unramified twisted generalized Levis that are isomorphic to  $\SO_4$ and $\SL_3$, respectively.   The label $(\{\beta, \beta + 3 \alpha \}, w_\alpha)$ corresponds to the  rational conjugacy class of unramified twisted generalized Levis that are isomorphic to (unramified) $\SU_3$.

\end{example}

\begin{example} In Figure~\ref{fig:lastp} we provide, up to rational conjugacy, a parameterization of unramified twisted Levis that are not twisted Levis for the groups $\Sp_4$, when the characteristic of $k$ is $2$, and $\Gtwo_2$, when the characteristic of $k$ is $3$.

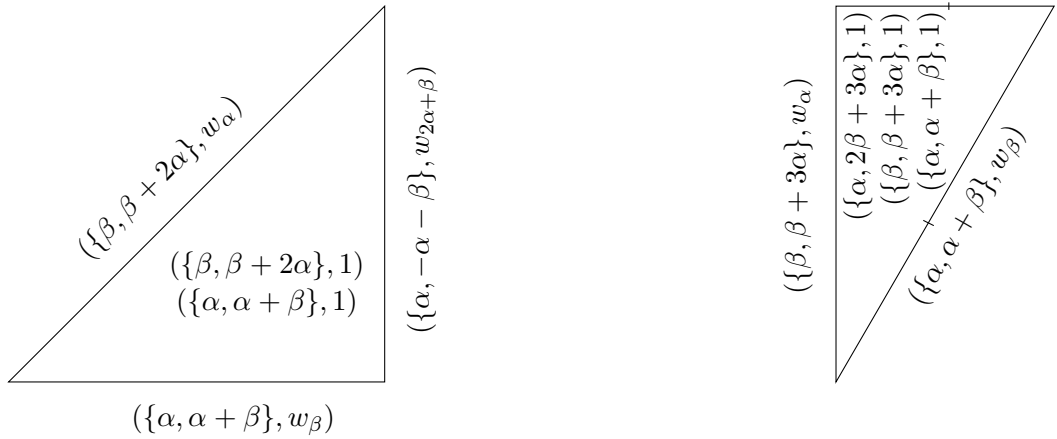
\begin{figure}[ht]
\centering
\begin{tikzpicture}
\draw (0,0)
  -- (5,0) 
  -- (5,5) 
  -- cycle;
   \draw (2,1.3) node[anchor=west]{\shortstack{$(\{\beta, \beta + 2 \alpha \},1)$  \\ $(\{\alpha, \alpha + \beta\}, 1)$}};
    \draw (.9,1.6) node[anchor=west, rotate = 45]{$(\{\beta, \beta + 2 \alpha \},w_\alpha)$};
    \draw (5.5,.5) node[anchor=west, rotate = 90]{$(\{\alpha, -\alpha - \beta \},w_{2 \alpha + \beta})$};
    \draw (1.5,-.5) node[anchor=west]{$(\{\alpha, \alpha + \beta \},w_\beta)$};

   \draw (11,0)
  -- (13.9,5) 
  -- (11,5) 
  -- cycle;
  \draw (10.5,4.2) node[anchor=east, rotate = 90]{$(\{\beta, \beta + 3 \alpha \}, w_\alpha)$};
    \draw (12.6,3.54) node[anchor=south, rotate = 90]{\shortstack{$(\{\alpha, 2\beta + 3 \alpha \}, 1)$\\ $\hspace{.5em}(\{\beta, \beta + 3 \alpha \},1)$ \\ $\hspace{1em}(\{\alpha, \alpha + \beta  \},1)$}};
  \draw (13.5,3.5) node[anchor=east, rotate = 60]{$(\{\alpha, \alpha + \beta \}, w_\beta)$};
     \draw (12.5,5.05) -- (12.5,4.95);
    \draw (12.3,2.08) -- (12.16,2.15);

\end{tikzpicture}

\caption{A parameterization of  the rational conjugacy classes of unramified twisted generalized Levis that are not unramified twisted Levis for $\Sp_4$ when $\characteristic(k) = 2$ and for $\Gtwo_2$ when $\characteristic(k) = 3$. \label{fig:lastp}}
\end{figure}

The unexplained labels are as in Example~\ref{ex:goodp}.

For the group $\Sp_4$ with $\characteristic(k) = 2$, the label $(\{\alpha , \alpha + \beta\},1)$ corresponds to the rational conjugacy class of unramified twisted generalized Levis that are isomorphic to $\SO_4$   while the labels
$(\{\alpha, \alpha + \beta \},w_\beta)$
and 
$(\{\alpha, -\alpha - \beta \},w_\beta)$
$(\{\beta, \beta + 2 \alpha \},w_{2 \alpha + \beta})$ correspond to the two distinct rational  conjugacy classes of unramified twisted generalized Levis that are isomorphic to the non-split quasi-split form of $\SO_4$. 

For the group $\Gtwo_2$ with $\characteristic(k) = 3$, the label $(\{\alpha, \beta +  \alpha \}, 1)$ corresponds to the  rational conjugacy class of unramified twisted generalized Levis that are isomorphic to $\PGL_3$.   The label $(\{\alpha, \beta +  \alpha \}, w_\beta)$ corresponds to the  rational conjugacy class of unramified twisted generalized Levis that are isomorphic to  $\PU_3$.

\end{example}

\subsection{Relations among unramified twisted generalized Levis}

As a Corollary to Theorem~\ref{thm:jisbijective} we have:

\begin{cor} \label{cor:maybethistime} Suppose $\bL$ and $\tilde{\bL}$ are  unramified twisted generalized Levi subgroups in $\bG$.  There exists $x \in G^\Fr$ such that  $\lsup{x}\bL \leq \tilde{\bL}$ if and only if there exist 
$(F,\Xi,w), (\tilde{F},\tilde{\Xi},\tilde{w}) \in  \calI^{\ee}$  and  $\Xi' \in \tilde{\Theta}$ such that
\begin{enumerate}
  \item \label{item:genleviawk} $\Phi_\Xi \subset \Phi_{\Xi'}$,
    \item $F$ is in the closure of $\tilde{F}$,
    \item \label{item:explainme}  $(\Xi',w) \stackrel{F}{\sim} (\tilde{\Xi},\tilde{w})$, 
    \item  $\bL \in j((F,\Xi,w))$, and
    \item  $\tilde{\bL} \in j((\tilde{F},\tilde{\Xi},\tilde{w}))$.
\end{enumerate}
Moreover, if $\bL$ and $\tilde{\bL}$ are  unramified twisted Levi subgroups in $\bG$, then statement~(\ref{item:genleviawk}) may be replaced by the statement: $\Xi \subset \Xi'$.
\end{cor}

\begin{rem} If $F$ is in the closure of $\tilde{F} \subset \AA(A^\Fr)$, then we have $G_{\tilde{F}} \leq G_{F}$ and so $W_{\tilde{F}} \leq W_F \leq W$.  Hence it makes sense to think of $\tilde{w}$ as an element of $W_F$ in statement~(\ref{item:explainme}) of Corollary~\ref{cor:maybethistime}.
\end{rem}

\begin{proof}  The last statement of the lemma, about unramified twisted Levi subgroups, is immediate because for  $\Xi, \Xi' \in {\Theta}$ we have  $\Phi_{\Xi} \subset \Phi_{\Xi'}$ if and only if there exists a basis $\Xi''$ for $\Phi_{\Xi'}$ such that $\Xi \subset \Xi''$.

``$\Leftarrow$''  Choose $g \in G_F$ such that the image of $\Fr(g)\inv g \in N_{G_F}(\bA)$ has image $w$ in $W_F$.  Choose $\tilde{g} \in G_{\tilde{F}}$ such that the image of $\Fr(\tilde{g})\inv \tilde{g} \in N_{G_{\tilde{F}}}(\bA)$ has image $\tilde{w}$ in $W_{\tilde{F}}$. Recall that $\bL_{(F,\Xi,w),g} := \lsup{g} \bM_{\Xi}$,  $\bL_{(F,\Xi',w),g} := \lsup{g} \bM_{\Xi'}$,  $\bL_{(\tilde{F},\tilde{\Xi},\tilde{w}),\tilde{g}} := \lsup{\tilde{g}} \bM_{\tilde{\Xi}}$, etc.   

Since $\Phi_\Xi \subset \Phi_{\Xi'}$, we have $\bL_{(F,\Xi,w),g} \leq \bL_{(F, \Xi',w),g}$.   Since 
$F$ is in the closure of $\tilde{F}$, we have $\tilde{g} \in G_{\tilde{F}} \leq G_F$ and $\tilde{w} \in W_{\tilde{F}} \leq W_F$; hence $\bL_{({F},\tilde{\Xi},\tilde{w}),\tilde{g}}$ makes sense and is equal to  $\bL_{(\tilde{F},\tilde{\Xi},\tilde{w}),\tilde{g}}$. Since  $(\Xi',w) \stackrel{F}{\sim} (\tilde{\Xi},\tilde{w})$, from Lemma~ \ref{lem:thelemmaweneed} there exists $k \in G_F^\Fr$ such that  $\lsup{k} \bL_{(F,\Xi',w),g} =  \bL_{(F,\tilde{\Xi},\tilde{w}),\tilde{g}}$.
Since  $\bL \in j((F,\Xi,w))$, from Theorem~\ref{thm:jisbijective} there exists $h \in G^{\Fr}$ such that $\lsup{h}\bL =  \bL_{(F,\Xi,w),g}$.
Since $\tilde{\bL} \in j((\tilde{F},\tilde{\Xi},\tilde{w}))$, from Theorem~\ref{thm:jisbijective} there exists $\tilde{h} \in G^{\Fr}$ such that $\tilde{\bL} = \lsup{\tilde{h}} \bL_{(\tilde{F},\tilde{\Xi},\tilde{w}),\tilde{g}}$.  If $x = \tilde{h}kh$, then $x \in G^{\Fr}$ and 
\[ \lsup{x}\bL =  \lsup{\tilde{h}kh}\bL = \lsup{\tilde{h}k} \bL_{(F,\Xi,w),g} \leq \lsup{ \tilde{h}k} \bL_{(F,\Xi',w),g}  = \lsup{\tilde{h}} \bL_{(F,\tilde{\Xi},\tilde{w}),\tilde{g}}  =  \tilde{\bL}.\]

``$\Rightarrow$''   We suppose $\lsup{x}\bL \leq \tilde{\bL}$.   Choose a $G^\Fr$-facet $F$ in $\BB(G)^\Fr$ such that $x\inv F$ is a maximal $G^\Fr$-facet in $\BB(L)^{\Fr}$.  Let $\bS$ be a $(K,k)$-torus such that $\lsup{x\inv} \bS$ is a maximally $k$-split $(K,k)$-torus in $\bL$ and $ F \subset \BB(S)^\Fr$. Since $\lsup{x}L \leq \tilde{L}$, there exists a maximally $k$-split $(K,k)$-torus $\tilde{\bS}$ in $\tilde{\bL}$ such that $F \subset \BB(S)^\Fr \subset \BB(\tilde{S})^\Fr \subset \BB(\tilde{L})^\Fr$.  We can choose $y \in G^\Fr$ such that $y \BB(\tilde{S})^{\Fr} \subset \AA(A)^{\Fr}$.   After conjugating everything in sight by $y$, we have that $F$ is a maximal $G^\Fr$-facet in $\lsup{x}\BB(L)^\Fr$
and
\[F \subset \BB(S)^\Fr \subset \BB(\tilde{S})^\Fr \subset \AA(A)^\Fr.\]

Choose $g \in G_{F}$ such that $\lsup{g}\bA = \bS$.  Let $\Xi$ be a basis for $\Phi(\lsup{g\inv x}\bL, \bA)$, and let $w$ denote the image of $\Fr(g)\inv g \in N_{G_F}(\bA)$ in $W_F$. 
Let $\Xi'$ be a basis for $\Phi(\lsup{g\inv}\tilde{\bL}, \bA)$.   Since $\lsup{x}\bL \leq \tilde{\bL}$, we have $\Phi_{\Xi} \subset \Phi_{\Xi'}$.

Let $\tilde{F}$ be a maximal $G^\Fr$-facet in $\BB(\tilde{S})$ that contains $F$ in its closure.  Choose $\tilde{g} \in G_{\tilde{F}}$ such that $\lsup{\tilde{g}}\bA = \tilde{\bS}$.  Let $\tilde{\Xi}$ be a basis for $\Phi(\lsup{\tilde{g}\inv}\bL, \bA)$, and let $\tilde{w}$ denote the image of $\Fr(\tilde{g})\inv \tilde{g} \in N_{G_{\tilde{F}}}(\bA)$ in $W_{\tilde{F}}$.

Since $F$ is in the closure of $\tilde{F}$, we have $\tilde{g} \in G_{\tilde{F}} \leq G_F$ and $\tilde{w} \in W_{\tilde{F}} \leq W_F$; hence $\bL_{({F},\tilde{\Xi},\tilde{w}),\tilde{g}}$ makes sense and is equal to  $\bL_{(\tilde{F},\tilde{\Xi},\tilde{w}),\tilde{g}}$. 
Since   $\tilde{\bL} =  \bL_{(F,\Xi',w),g} =  \bL_{(F,\tilde{\Xi},\tilde{w}),\tilde{g}}$, from Lemma~ \ref{lem:thelemmaweneed} we have
$(\Xi',w) \stackrel{F}{\sim} (\tilde{\Xi},\tilde{w})$.
 
 As in the proof of Theorem~\ref{thm:jisbijective}, since $F$ was chosen to be a maximal $G^\Fr$-facet in $\BB(\lsup{x}L)$, we have $(F,\Xi,w) \in \calI^{\ee}$
 and $\lsup{x}{\bL} \in j((F,\Xi,w))$.  Since $j((F,\Xi,w))$ is a single $G^\Fr$-conjugacy class of unramified twisted generalized Levi subgroups of $\bG$, we have $\bL \in j((F,\Xi,w))$.

Similarly, since  $\tilde{F}$ was chosen to be a maximal $G^\Fr$-facet in $\BB(\tilde{L})$, we have
 $(\tilde{F},\tilde{\Xi},\tilde{w}) \in \calI^{\ee}$ and $\tilde{\bL} \in j((\tilde{F},\tilde{\Xi},\tilde{w}))$.
\end{proof}

 \subsection{Twisted generalized Levi subgroups for reductive groups over quasi-finite fields}

We close with a  generalization of the material in Section~\ref{sec:leviff}.
Let $\sfG$, $\sfB$, $\sfA$ etc. be as in Section~\ref{sec:leviff}.  We denote by  $\Phi_\sfG = \Phi(\sfG, \sfA)$ the roots of $\sfG$ with respect to $\sfA$ and by  $\Phi^+_\sfG = \Phi^+(\sfG, \sfB, \sfA)$ the corresponding set of positive roots. 

  For $\rho \subset \Phi_\sfG$  define
$\sfG_\rho$ to be the group generated by $\sfA$ and the root groups $\sfU_\alpha$ for $\alpha \in \rho$. The subset $\rho$ of $\Phi_\sfG$ is said to be quasi-closed provided that if $\beta \in \Phi_\sfG$ and $\sfU_\beta \subset \sfG_\rho$, then $\beta \in \rho$.
We denote by $\tilde{\Delta}_\sfG = \tilde{\Delta}(\sfG, \sfB, \sfA)$ the set 
\[ \{ \Xi \subset \Phi^+_\sfG \, | \, \text{ $\Xi$ is a basis for a quasi-closed subset of $\Phi_\sfG$} \}. \]

\begin{defn}
A reductive subgroup $\sfL$ of $\sfG$ is called a \emph{twisted generalized Levi $\ff$-subgroup of $\sfG$} provided that $\sfL$ is defined over $\ff$ and $\sfL$ is a full rank reductive subgroup of $\sfG$.   We let $\LL'$ denote the set of twisted generalized Levi $\ff$-subgroups of $\sfG$, and we let $\tilde{\LL}'$ denote the set of $\sfG^{\Fr}$-conjugacy classes in $\LL'$.
\end{defn}
 
Let $I'_\sfG$ denote the set of pairs $(\Xi,w)$ where $\Xi \subset \tilde{\Delta}_\sfG$ and $w \in W_\sfG$ such that $\Fr (\Xi) = w \Xi$.   For $(\Xi',w')$ and $(\Xi,w) \in I'_\sfG$ we write $(\Xi',w') \sim (\Xi,w)$ provided that there exists an element $\dot{n} \in W_\sfG$ for which
\begin{itemize}
\item  $   \Xi = \dot{n} \Xi ' $ and
\item $ w = \Fr(\dot{n}) w' \dot{n}\inv $.
\end{itemize}
One checks that $\sim$ is an equivalence relation on the set $I'_\sfG$.

\begin{lemma}
There is a natural bijective correspondence between $I'_\sfG / \! \sim$ and $\tilde{\LL}'$. 
\end{lemma}

\begin{proof}    For $\Xi \in \tilde{\Delta}$  define
$\sfM_\Xi$ to be the group generated by $\sfA$ and the root groups $\sfU_\alpha$ for $\alpha$ in the root system spanned by $\Xi$.   We let  $W_{\sfG,\Xi}$ denote the corresponding subgroup of $W_\sfG$.  

With the definitions above and appropriate minor  modifications, the proof mimics that of Lemma~\ref{lem:ffversion}.
\end{proof}

\begin{example}  We consider $\sfG = \G2$ and adopt the notation of Example~\ref{ex:g2finiteinitial}.   In Table~\ref{table:gtwofiniteB}  a complete list of representatives for the elements of $I'_\sfG/ \! \sim$ that are not in $I_\sfG/\! \sim $.    Recall that $I_\sfG$ is defined in Section~\ref{subsec:ffparam}.  We also indicate the type of the corresponding twisted generalized Levi $\ff$-subgroup of $\G2$.

\begin{table}[!ht]
\centering
\begin{tabular}{ |c||c|c| }
 \hline
 Pair & Type of twisted generalized Levi $\ff$-group & conditions on $\ff$\\
 \hline \hline
 $(\{\alpha, 2 \beta + 3 \alpha \},  1)$ & $\SO_4$ & none \\
 $(\{\beta, \beta + 3 \alpha \}, 1)$ &  ${\SL}_3$ & none\\
 $(\{\beta, \beta + 3 \alpha \}, w_\alpha )$ &  ${\SU}(3)$ & none\\
 $(\{\alpha, \beta +  \alpha \}, 1 )$ &  $\PGL_3$ & $\characteristic(\ff) = 3$\\
  $(\{\alpha, \beta +  \alpha \}, w_\beta )$ &  $\PU_3$ & $\characteristic(\ff) = 3$\\
 \hline
 \end{tabular}
 \caption{$\G2$: A set of representatives for $(I'_\sfG \setminus I_\sfG)/ \! \sim$  \label{table:gtwofiniteB}}
\end{table}
\end{example}

\FloatBarrier

\appendix

\section{Existence of \texorpdfstring{$K$-minisotropic maximal $k$-tori}{K minisotropoic maximal k tori}}  \label{sec:appendixone}

\begin{center}
    Jeffrey D. Adler
\end{center}

When $k$ has characteristic zero,
it follows from~\cite[Section 15]{knesser:galoisII} or~\cite[Theorem~6.21]{platonov-rapinchuk:algebraic}  that every connected reductive $k$-group $\bG$ contains a $k$-minisotropic maximal $k$-torus.  When the residue field of $k$ is finite, it follows from~\cite[Section 2.4]{debacker:unramified} that $\bG$ contains a $k$-minisotropic maximally $K$-split maximal $k$-torus.  In this appendix, we show that when the residue field of $k$ is finite,  $\bG$  contains a $k$-minisotropic maximal $k$-torus that is as ramified as possible.

\begin{lemma}
If $\bG$ is a connected reductive $k$-group and the residue field of $k$ is finite,
then $\bG$ contains a $K$-minisotropic maximal $k$-torus.
\end{lemma}

\begin{proof}
First observe that our result is true for general linear and $k$-quasi-split unitary groups.  For if $n\geq 1$, the field $k$ has a totally ramified, separable extension of degree $n$.
For such a field $E$, the torus $R_{E/k}\GL_1$ embeds
as a $K$-minisotropic maximal $k$-torus in $\GL_n$.
Given a quadratic Galois extension $L/k$, we can choose $E$ to not contain $L$, in which case the kernel of the map
$N_{EL/E}\colon R_{EL/k}\GL_1 \rightarrow R_{E/k}\GL_1$
embeds as a $K$-minisotropic maximal $k$-torus in the quasi-split unitary group $U_{n,L/k}$.

Second, we reduce to the case where $\bG$ is absolutely simple.  Observe that our result is true for $\bG$ if and only if it is true for $\bG/\bZ$, where $\bZ$ is the center of $\bG$, and so we may assume that $\bG$ is adjoint.
Write $\bG=\prod_{i=1}^r R_{E_i/k}\bG_i$, where each $E_i/k$ is a finite separable extension, and each $\bG_i$ is an absolutely simple $E_i$-group (see~\cite[Section~6.21]{borel-tits:groupes}).
If each group $\bG_i$ contains a $KE_i$-anisotropic maximal $E_i$-torus $\bT_i$, then $\prod_i R_{E_i/k}\bT_i$ is a $K$-anisotropic maximal $k$-torus in $\bG$ (see~\cite[Corollary~6.19]{borel-tits:groupes}).
Therefore, we may replace $\bG$ by $\bG_i$ and $k$ by $E_i$, and assume that $\bG$ is absolutely simple.

Third, we reduce to the case where $\bG$ is $k$-quasi-split.  For suppose that $\bG_0$ is a $k$-quasi-split inner form of $\bG$.
From a result of Kottwitz~\cite[Section 10]{kottwitz:stable} (see~\cite[Section 3.2]{kaletha:regular} for the characteristic free version) 
a $k$-anisotropic torus $\bT$ $k$-embeds in $\bG$ if and only if it $k$-embeds in $\bG_0$.  Whether or not $\bT$ is $K$-anisotropic is independent of the $k$-embedding, so we may as well replace $\bG$ by $\bG_0$ and assume that $\bG$ is $k$-quasi-split.

Fourth, suppose that $\bG$ has a full-rank, semisimple $k$-subgroup $\bH$ that is, up to isogeny, a product of groups of type $A$.  Then we have already seen that $\bH$ has a $K$-anisotropic maximal $k$-torus, and thus so does $\bG$.  Therefore, it will be enough to show that $\bG$ contains such a subgroup.
Choose a maximal $k$-torus in a $k$-Borel subgroup of $\bG$.  These choices determine an absolute root system $\Phi$ and a system $\Delta$ of simple roots, both of which are acted upon by $\Gal(E/k)$, where $E$ is the splitting field of our torus.
It will be enough to show that $\Phi$ contains a closed, full-rank subsystem, invariant under the action of $\Gal(E/k)$, that is a product of systems of type $A$.

Identify $\Delta$ with its Dynkin diagram, and let $\widetilde\Delta_v$ be the diagram obtained by deleting a vertex $v$ from the extended Dynkin diagram of $\Delta$.
A theorem of Borel and de Siebenthal tells us that each $\widetilde\Delta_v$ is the Dynkin diagram of a maximal, full-rank, closed subsystem of $\Phi$.
If $v$ is fixed by the action of $\Gal(E/k)$, then so is our subsystem.
Therefore, it will be enough to show that by iterating this process (i.e. replacing a Dynkin diagram by its extended diagram, and then deleting a $\Gal(E/k)$-invariant vertex), one can eventually obtain a product of diagrams of type $A$.

Doing so is straightforward in the cases where $\bG$ is $k$-split, i.e.  $E=k$.  Thus we only need to consider absolutely simple groups of type $^2D_n$ ($n\geq 4$),  $^3D_4$,  $^6D_4$, and $^2E_6$ (see~\cite[\S2 and Table II]{tits:classification}).  If $\bG$ has type $^2D_n$ ($n\geq 4$), then $\Phi$ contains a closed subsystem of type $D_{n-2}\times R_{E/k}A_1$.
If $\bG$ has type $^3D_4$, then $\Phi$ contains a closed subsystem of type $A_1 \times R_{E/k} A_1$.
If $\bG$ has type $^6D_4$, then $\Phi$ contains a closed subsystem of type $A_1 \times R_{E'/k} A_1$, where $E'/k$ is a cubic extension contained in $E$.
If $\bG$ has type $^2E_6$, then $\Phi$ contains a closed subsystem of type $A_2 \times R_{E/k} A_2$.
\end{proof}

\end{document}